\documentclass[reqno]{amsart}
\usepackage[english]{babel}
\usepackage{cite} 
\usepackage[dvips]{graphicx}

\usepackage[usenames, dvipsnames]{xcolor}
\usepackage[utf8x]{inputenc}
\usepackage[T1]{fontenc}
\usepackage{tikz}
\usepackage{pgfplots}
\pgfplotsset{compat=1.9}
\usepackage{transparent} 
\usepackage[backref=page]{hyperref}
\usepackage{hyperref}

\usepackage{caption}
\usepackage{float}
\usepackage{multirow}
\usepackage{adjustbox}
\usepackage{amsfonts,amsmath,amsxtra,amsthm,amssymb,latexsym, bm}
 \newcommand\dom{\operatorname{dom}}

\def\dg{\mathcal{G}}

\def\ldos{\mathcal{L}_{-,Z}}
\def\luno{\mathcal{L}_{+,Z}}
\newcommand{\xb}{X_{N,\perp}}
\newcommand{\xa}{X_{1,\perp}}
\newcommand{\xm}{X_{\mathrm{mean}}}
\newcommand{\hz}{\mathcal{H}^{\delta'}_Z}
\def\ker{\operatorname{Ker}}
\def\dom{D(\mathcal{H}_{Z}^{\delta'})}

\def\V{\mathbf{V}}

\usepackage{import}
\usepackage{transparent}

\newcommand{\R}{\mathbb{R}}
\newcommand{\C}{\mathbb{C}}

\renewcommand{\le}{\leqslant}
\renewcommand{\ge}{\geqslant}
\renewcommand{\leq}{\leqslant}
\renewcommand{\geq}{\geqslant}

\newcommand{\be}{\begin{equation}}
\newcommand{\en}{\end{equation}}
\newcommand{\ee}{\end{equation}}

\newcommand{\bt}{\begin{theorem}}
\newcommand{\et}{\end{theorem}}
\newcommand{\bp}{\begin{proof}}
\newcommand{\ep}{\end{proof}}
\newcommand{\bc}{\begin{cor}}
\newcommand{\ec}{\end{cor}}
\newcommand{\bl}{\begin{lemma}}
\newcommand{\el}{\end{lemma}}
\newcommand{\bprop}{\begin{proposition}}
\newcommand{\eprop}{\end{proposition}}

\newtheorem{theorem}{Theorem}[section]
\newtheorem{remark}{Remark}
\newtheorem{lemma}[theorem]{Lemma}
\newtheorem{definition}{Definition}
\newtheorem{proposition}[theorem]{Proposition}

\numberwithin{theorem}{section} \numberwithin{definition}{section}
\numberwithin{equation}{section}

\newcommand{\RNum}[1]{\uppercase\expandafter{\romannumeral #1\relax}}
\def\R{\mathbb{R}}

\def\C{\mathbb{C}}

\newcommand{\vertiii}[1]{{\left\vert\kern-0.25ex\left\vert\kern-0.25ex\left\vert #1 
		\right\vert\kern-0.25ex\right\vert\kern-0.25ex\right\vert}}

\theoremstyle{definition}

\begin{author}
{Jaime Angulo Pava and Alexander Mu\~noz }
\end{author}

\begin{title}[(In)stability of discontinuous standing waves for NLS on star graphs]
{On (in)stability of discontinuous standing waves for the NLS with a delta-prime on star graphs}
\end{title}
\date{}
\begin{document}
\maketitle

\centerline{Department of Mathematics,
IME-USP}
 \centerline{Rua do Mat\~ao 1010, Cidade Universit\'aria, CEP 05508-090,
 S\~ao Paulo, SP, Brazil.
 }
 \centerline{\it {angulo@ime.usp.br}, alexd@usp.br}
 
\section*{Abstract}
We investigate the existence and (in)stability of standing waves that are discontinuous at the vertex for the one-dimensional nonlinear Schr\"odinger equation (NLS) with a power nonlinearity, posed on a star graph consisting of a finite number $N$ of half-lines and endowed with a delta-prime interaction. Our non-variational approach characterizes all possible configurations of such standing wave profiles and shows, rather surprisingly, that the components of any discontinuous-at-the-vertex stationary solution split into two groups, each determined by one of exactly two distinct types of shifted soliton for the NLS on a half-line. More precisely, for every $N\geqq 2$ and every $n\in\{1,\dots,N-1\}$, there are two groups consisting of exactly $n$ and $N-n$ components, with each group sharing a common shift. The stability properties of these discontinuous-at-the-vertex stationary states are analyzed via the Grillakis-Shatah-Strauss framework and the Grillakis-Jones Instability Theorem, where non-standard techniques are required to establish the necessary spectral properties of the linearization operators, in particular their Morse and deficiency indices. Our approach yields a complete description of these indices. The results of this manuscript optimally extend previous findings in the literature concerning delta-prime type interactions on star graphs. Moreover, the methods developed here have the prospect of being adapted to study the stability of other discontinuous-at-the-vertex stationary solutions of the NLS on non-compact metric graphs, such as looping-edge graphs.

\qquad\\
\textbf{Mathematics Subject Classification (2020)}. Primary 35Q51, 35Q55, 81Q35, 35R02; Secondary 47E05.\\
\textbf{Key words}. Schr\"odinger equation, quantum graphs, $\delta'$ interactions.
\section{Introduction}
We consider the nonlinear Schr\"odinger equation (NLS)
\begin{equation}\label{NLS}
i\textbf{U}_t + \Delta\textbf{U}+ |\textbf{U}|^{p-1}\textbf{U} =\textbf{0},\quad p>1,
\end{equation}
where the action of the Laplacian operator $\Delta$ on a general metric graph $\mathbb{G}$ is given by
\begin{equation}
-\Delta: (u_e)_{e\in E} \to (-u''_e)_{e\in E}.
\end{equation}  

In this work we focus on star graphs, a metric graph composed of a finite number $N$ of infinite half-lines attached to a common vertex. If the half-lines are identified with $[L, \infty)$ with $L\ge 0$, we obtain a particular metric graph structure, that we will denote with $\dg$, represented by $\dg=V\cup E$ where $E = \{[L, \infty), \dots, [L, \infty)\}$, $V = \{\nu\}$ and along the coordinate system of each edge $e$, the single vertex $\nu$ is located exactly at $x_e=L$ (see Figure~\ref{Fig1}). We use the subscripts $j$ to refer to the edges $e_j = [L, \infty)$, $j = 1, \dots, N$. A wave function $\textbf{U}$ defined on $\dg$ will be understood as an $N$-tuple of functions $\textbf{U} = (\psi_1,\dots,\psi_N)$, where each $\psi_j$ is defined on $e_j = [L, \infty)$ for $j = 1, \dots, N$. 

A key starting point for studying the dynamical properties of the NLS on $\mathcal{G}$ is to require the Laplacian operator $-\Delta$ to be self-adjoint, with domain $D(-\Delta)\subset L^2(\dg)$, which encodes the coupling condition at the graph vertex $\nu=L$. Previous works have investigated the dynamics of standing waves solutions $\textbf{U}(x, t)=e^{i\omega t} \Theta_\omega(x)$, $\Theta_\omega=(\Psi_j)_{j=1}^N$, by the flow of  the NLS on the graph $\dg$, with  the profile $\Theta_\omega$ satisfying the system
\begin{equation}
\label{elliptic}
-\Delta \Theta_\omega +\omega \Theta_\omega- |\Theta_\omega|^{p-1}\Theta_\omega =\textbf{0}, \quad p>1
\end{equation}
with the {\it a priori} asymptotic behavior $\Theta_\omega\to 0$ as $x\to +\infty$. The best known coupling conditions in the graph vertex are the so-called delta ($\delta$-henceforth) an delta-prime ($\delta '$-henceforth) interactions. For the presence of a  $\delta$-interaction, the Laplace  operator $-\Delta=\mathcal{H}_Z^{\delta}$ emerges with the domain $D_{Z, \delta}$, with $Z\in \R$, and 
\begin{equation}
\label{domainintrodelta}
D_{Z, \delta}:=\Big\{\mathbf{U}=(\psi_j)_{j=1}^N\in H^2(\dg) \Bigm| \psi_1(L)=\cdots=\psi_N(L), \quad \sum_{j=1}^N \psi'_j(L)=Z\psi_1(L) \Big\},
\end{equation}
where the continuity of the wave functions  at the vertex is mandatory. NLS with $\delta$-interaction has been
widely studied in the literature (see \cite{AdaNoj15, AdaNoj14, AdaNoj14a},  \cite{AngGol17a, AngGol17b}, \cite{KNP}, and references therein). The existence and (in)stability of standing waves, as well as local and/or global well-posedness theories, have been already established in the literature. On the other hand, for the $\delta '$-interactions, the Laplace  operator $-\Delta=\mathcal{H}_Z^{\delta '}$ emerges with the domain  $\dom$, with $Z\in \R$, and 
\begin{equation}
\label{domainintro}
D(\mathcal{H}_Z^{\delta'}):=\Big\{\mathbf{U}=(\psi_j)_{j=1}^N\in H^2(\dg) \Bigm| \psi_1'(L)=\cdots=\psi_N'(L), \quad \sum_{j=1}^N \psi_j(L)=Z\psi_1'(L) \Big\}, 
\end{equation}
where continuity of the wave functions at the vertex is not required. The NLS with $\delta'$-interaction has been less studied in the literature for arbitrary $N$. In the case $Z<0$ (the attractive case), existence and orbital stability of standing waves with profile $\Theta_\omega$ whose components are all equal, $\psi_1=\cdots=\psi_N$, and of tail-type ($\psi_1'(x)<0$ for all $x$), were initially studied in \cite{AngGol17b}, and later extended in \cite{G}. Moreover, in \cite[Theorems~4.2 and 5.23]{G}, it was also studied the existence and instability of some discontinuous (asymmetric) standing wave solutions in $\dom$ via variational methods (minimization problems on a specific Nehari manifold). In the present work we characterize all possible configurations of such standing wave profiles via non-variational techniques and so we extend the existence of this kind of solutions to an optimal range of existence and provide an almost complete (in)stability picture  all these configurations (see Table \ref{table:1}) in  Appendix D.

In contrast with the theory in metric graphs, for the case $\mathbb{G}=\mathbb{R}$ with a $\delta'$-interaction located at the origin, extensive studies have been conducted in recent years (see for instance \cite{ABG, AdaNo, AngGol17, AngGol17b, G}). We note that in this case, sign-changing and/or non-symmetric standing-wave profiles are possible; see Figure~1 in \cite{ABG}.

A natural question concerns the range of admissible self-adjoint realizations of $-\Delta$ on the star graph $\dg$ that retain the derivative-continuity structure of $\dom$. The $\delta'$-condition in \eqref{domainintro} is in fact only one member of a much larger family. Building on the systematic parametrization of all self-adjoint extensions of the Laplacian on looping-edge graphs (graphs obtained by attaching a circle to the common vertex of a star graph) carried out in \cite{AnMu}, one identifies a broad class of extensions whose dynamics are governed by a Hermitian matrix $M'\in M_N(\mathbb{C})$ acting through the vertex condition $M'\vec\psi(L)=\vec\psi\,'(L)$, where $\vec\psi(L)=(\psi_1(L),\dots,\psi_N(L))^\top$ and $\vec\psi\,'(L)=(\psi_1'(L),\dots,\psi_N'(L))^\top$. Each Hermitian choice of $M'$ prescribes a different coupling at the vertex while preserving the half-line structure: for instance, $M'=0$ yields the Neumann decoupling $\psi_j'(L)=0$; the choice $M'=\alpha I_N$ produces uniform Robin conditions $\psi_j'(L)=\alpha\psi_j(L)$; and the rank-one choice $M'=\tfrac{1}{Z}\mathbf{1}\mathbf{1}^\top$ recovers precisely the $\delta'$-interaction \eqref{domainintro} studied here (see \cite[Remark~6]{AnMu}). We emphasize that the domain $\dom$ adopted in the present manuscript occupies an intermediate position within this hierarchy: it is rich enough to give rise to a genuinely nontrivial study of stable-unstable transitions, yet structured enough to permit a complete and explicit (in)stability theory. In this sense, the $\delta'$-interaction serves as a natural building block from which more complex vertex couplings, and the richer dynamics they entail, can subsequently be explored.

The $\delta'$ framework for the NLS on star graphs bridges theoretical quantum mechanics and applied nonlinear physics across several domains. In nonlinear optics, it serves as a foundational model for light propagation in branched optical waveguides, fiber-optic arrays, and photonic networks, where the $\delta'$ vertex condition simulates point-like structural defects, splitters, or junctions that govern how optical solitons are trapped, transmitted, or reflected. In Bose-Einstein Condensation (BEC), it is widely applied to the study of ultracold atoms confined to ramified, quasi-one-dimensional geometries (e.g., $Y$-shaped or star-shaped optical traps), where the $\delta'$ interaction effectively models localized potential barriers or constrictions along the condensate's path (see Tokuno {\it et al.}~\cite{TOD}). In condensed matter physics, the model describes quantum transport phenomena, such as electron tunneling and charge-carrier flow through complex branching molecular networks, branched nanowires, and quantum graph-based circuit systems. In neuronal networks and fluid dynamics, it models pulse and signal propagation in complex branching structures, including neuronal signal routing and fluid flow in branching conduits (see \cite{BurCas01, GSD, Fid15, Mug15, KNP, SMSSK} for further details and references).

\begin{figure}
    \centering
    \resizebox{0.4\linewidth}{!}{
\begingroup%
  \makeatletter%
  \providecommand\color[2][]{%
    \errmessage{(Inkscape) Color is used for the text in Inkscape, but the package 'color.sty' is not loaded}%
    \renewcommand\color[2][]{}%
  }%
  \providecommand\transparent[1]{%
    \errmessage{(Inkscape) Transparency is used (non-zero) for the text in Inkscape, but the package 'transparent.sty' is not loaded}%
    \renewcommand\transparent[1]{}%
  }%
  \providecommand\rotatebox[2]{#2}%
  \newcommand*\fsize{\dimexpr\f@size pt\relax}%
  \newcommand*\lineheight[1]{\fontsize{\fsize}{#1\fsize}\selectfont}%
  \ifx\svgwidth\undefined%
    \setlength{\unitlength}{171.20774757bp}%
    \ifx\svgscale\undefined%
      \relax%
    \else%
      \setlength{\unitlength}{\unitlength * \real{\svgscale}}%
    \fi%
  \else%
    \setlength{\unitlength}{\svgwidth}%
  \fi%
  \global\let\svgwidth\undefined%
  \global\let\svgscale\undefined%
  \makeatother%
  \begin{picture}(1,0.47545646)%
    \lineheight{1}%
    \setlength\tabcolsep{0pt}%
    \put(0,0){\includegraphics[width=\unitlength,page=1]{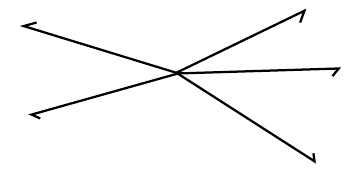}}%
    \put(0.00172034,0.42803755){\color[rgb]{0,0,0}\makebox(0,0)[lt]{\lineheight{1.25}\smash{\begin{tabular}[t]{l}$+\infty$\end{tabular}}}}%
    \put(0.03279585,0.10432344){\color[rgb]{0,0,0}\makebox(0,0)[lt]{\lineheight{1.25}\smash{\begin{tabular}[t]{l}$+\infty$\end{tabular}}}}%
    \put(0.89164774,0.00333745){\color[rgb]{0,0,0}\makebox(0,0)[lt]{\lineheight{1.25}\smash{\begin{tabular}[t]{l}$+\infty$\end{tabular}}}}%
    \put(0.97030182,0.25241311){\color[rgb]{0,0,0}\makebox(0,0)[lt]{\lineheight{1.25}\smash{\begin{tabular}[t]{l}$+\infty$\end{tabular}}}}%
    \put(0.88096716,0.46879935){\color[rgb]{0,0,0}\makebox(0,0)[lt]{\lineheight{1.25}\smash{\begin{tabular}[t]{l}$+\infty$\end{tabular}}}}%
    \put(0.48934885,0.19346817){\color[rgb]{0,0,0}\makebox(0,0)[lt]{\lineheight{1.25}\smash{\begin{tabular}[t]{l}$L$\end{tabular}}}}%
    \put(0,0){\includegraphics[width=\unitlength,page=2]{N=5sample.pdf}}%
  \end{picture}%
\endgroup%
}
    \caption{A metric star graph with $N=5$ and edge lengths identified with $[L,\infty)$.}
    \label{Fig1}
\end{figure}

The purpose of the present work is to address the existence and (in)stability of standing waves solutions \emph{discontinuous} at the vertex for the NLS equation  and lying in the $\delta'$-type operator domain $\dom$ in \eqref{domainintro} with $Z<0$. Our study characterizes all possible configurations of such standing wave profiles and shows, in a rather surprising way, that any discontinuous-at-the-vertex stationary solution $\Theta_\omega$, with $\omega$ fixed, has its components divided into two groups, determined exactly by two distinct types of shifted soliton for the NLS. This soliton profile is given by the unique positive decaying solution (up to translation)  of
\begin{equation}
\label{elliptichalf}
- \Psi'' +\omega \Psi - \Psi^{p} =0, \quad\omega>0,\;\; x\geqq L.
\end{equation}
The results in this manuscript optimally extend and complement some findings in the literature regarding $\delta'$ type interactions on star graphs. For a better explanation of our results, along the manuscript we will deal with the cubic case $p=3$ but we remark that the sharp theory developed below extends with optimality to the cases $p=2$ and $p=4$ since the arguments and computations are identical in spirit to those for $p=3$ (see Remark~\ref{p-thres}); we restrict our detailed presentation only to $p=3$ to avoid overloading the paper and reserving Section~\ref{sec:generalp} to discuss the general case $p>1$ (where some natural restrictions emerge in the study).

A key feature of this setting is that the derivative continuity condition defining domain \eqref{domainintro}, together with the structure of the half-line stationary shifted soliton solution of \eqref{elliptichalf}
\begin{equation}\label{solitonintro}
\Psi_a(x)=\sqrt{2\omega}\,\operatorname{sech}\!\big(\sqrt{\omega}(x-L)+a\big), \quad x\ge L,
\end{equation} 
force the components of the stationary solution to be defined by {\it only two admissible type of shifts} $a$, denoted by $a_A(\omega)$ and $a_B(\omega)$. These shifts determine  the components $\Psi_{a_j}$ of any  discontinuous-at-the vertex stationary solution $\Theta_\omega$. When $Z<0$ such components are  given by
\[
a_A(\omega):=\operatorname{arctanh}(t_A(\omega))>0, \qquad a_B(\omega):=\operatorname{arctanh}(t_B(\omega))>0,
\] 
where $ t_A(\omega), t_B(\omega)$ are chosen through the vertex sum condition prescribed in $D(\mathcal{H}_Z^{\delta'})$ of a smooth way with regard to $\omega$ and they satisfy
\begin{equation}\label{tAB}
0<t_A(\omega)<\frac{1}{\sqrt{2}}<t_B(\omega)<1,\;\;\;\; t^2_A(\omega)+ t^2_B(\omega)=1
\end{equation} 
where the phase-velocity $\omega$ satisfies specific and natural constraints (see the proof of Theorem \ref{thm:existence} and Proposition \ref{prop:existencet} below in Section 4). 

\medskip
An important geometric distinction between profiles with shift of type $A$ and 
profiles with shift of type $B$, for $a_A\neq a_B$, is that type $A$ contain the right-hand 
inflection point of the soliton in \eqref{solitonintro} while type $B$ does not  have  inflection points  (see Figure~\ref{fig:placeholder} for some profile configurations of our interest). We will refer to the profiles with shift of type $B$ as \emph{pure tail} profiles, in the sense that 
$\Psi_{a_B}''(x)>0$ for all $x\geqq L$.

We note that when $a_A\neq a_B$ no profile can satisfy the 
condition $\Psi_a''(L)=0$. In other words, truly discontinuous configurations cannot have shifts $a$ that place $L$ at an inflection point of the soliton. 

In our setting, $p=3$, the case $a_A(\omega)=a_B(\omega)=\operatorname{arctanh}(1/\sqrt{2})$ ($t_A(\omega)=t_B(\omega)$) occurs only at the single transition frequency 
\begin{equation}\label{eq:trans-w}
\widehat{\omega}:=\frac{2N^2}{Z^2}
\end{equation}
for certain clustering configurations of the stationary solution $\Theta_\omega$ (see the proof of Theorem \ref{thm:existence} and Remark~\ref{rem:w=2N^2/Z^2}). In such case all the components $\Psi_{a_j}$ coincide and the resulting $\Theta_\omega$ is continuous at the vertex with every component satisfying $\Psi_{a_j}''(L)=0$. As we will see from our existence and  (in)stability theories, this frequency $\widehat{\omega}$ plays the role of a threshold value in three specific scenarios: firstly, this value produce inconclusive stability information about the associated standing waves (see Theorem \ref{thm:insta},  Propositions  \ref{prop:l2}, \ref{prop:ker-L1Z}, or Theorem \ref{thm:morse1}). Secondly, this frequency becomes a transition limit for stable-unstable dynamics of profiles in the case  $a_A\neq a_B$ (see Theorem \ref{thm:insta} and Table~\ref{table:1}). Lastly, at this frequency, we obtain an inversion of the shift type of the components $\Psi_{a_j}$ for some unbalanced clustering configurations (see Theorem \ref{thm:existence}).

\medskip
Before establishing our main existence theorem of discontinuous-at-the-vertex stationary solutions of \eqref{NLS}, $p=3$, in $D(\mathcal{H}_Z^{\delta'})$, let us establish some basic notations. We consider $N\ge 2$ and $n\in \{1,\dots,N-1\}$. The existence of discontinuous-at-the-vertex stationary solutions will be
organized into two groups of exactly $n$ and $N-n$ component sharing a common shift. 
Such solutions take the form
\begin{equation}
\label{eq:thetan}
\Theta_{\omega,n}
:=
\bigl(\Psi_{a_1},\dots,\Psi_{a_N}\bigr), 
\qquad a_1=\dots=a_n\neq a_{n+1}=\dots=a_N, \qquad a_1,a_N\in\{a_A(\omega),\, a_B(\omega)\},
\end{equation}
 where each time $a_1$ defines a type $A$ profile then $a_N$ defines a type $B$ profile and vice-versa (see Figure~\ref{fig:placeholder} for the case $N=4$). We note that in the excluded cases $n=0, n=N$ (and also when $N=1$), on each edge we get the same component profile $\Psi_{a_1}$, generating a continuous-at-the-vertex stationary solution $\Theta_{\omega, 0}$ with $ a_1=\dots=a_N$, $\Psi_{a_1}=\cdot\cdot\cdot=\Psi_{a_N}\equiv \Psi_{\omega, 0}$, and
 \[
 \Psi_{ \omega, 0}(x)=\sqrt{2\omega}\,\operatorname{sech}\!\Big(\sqrt{\omega}(x-L)+{\tanh}^{-1} \Big(\frac{-N}{Z\sqrt{\omega}}\Big)  \Big), \quad x\ge L,
\]
 for $\omega>N^2/Z^2$ (see \eqref{eq:two-star-intro}). Stability properties of this branch of standing wave profiles were studied in Angulo and Goloshchapova in \cite{AngGol17b} and it is worth emphasizing that the frequency 
$\widehat{\omega}$ in \eqref{eq:trans-w} also represents a threshold value for the transition of  stable-unstable dynamics of the profiles $\Psi_{\omega, 0}$ in their work.

 \medskip
For fixed $n\in\{1,\dots,N-1\}$, the existence of standing waves discontinuous at the vertex, $\Theta_{\omega,\,n}$,  gathered
exactly by groups of $n$ and $N-n$ edges and lying in the domain 
$D(\mathcal{H}_Z^{\delta'})$, is ensured whenever the frequency satisfies
\begin{equation}
    \label{eq:two-star-intro}
    \omega>\omega_n^*:=\frac{\big(n^{2/3}+(N-n)^{2/3}\big)^{3}}{Z^2},
\end{equation}
which is the explicit existence-threshold (namely, for $\omega<\omega_n^*$ no positive discontinuous stationary profile in $\dom$  exists) associated with the minimum of a function $\alpha_n$ encoding the boundary conditions (see 
\eqref{eq:alphan}, \eqref{eq:cubic-infimum} and Figure \ref{fig:p=3} below). When $\omega=\omega_n^*$ one can still exhibit stationary solutions $\Theta_{\omega_n^*,\,n}$ which are discontinuous and in $\dom$ for unbalanced configurations   $2n\neq N$, however, despite the existence, the (in)stability analysis developed here fails at this frequency because the definition of the shifts $a_A$, $a_B$ via $\alpha_n$ takes place exactly at a critical point of $\alpha_n$ (see subsection 4.2- \eqref{eq:cubic-infimum}), degenerating most of the arguments in the spectral analysis (see Table \ref{table:1} or the proof of Proposition \ref{prop:ker-L1Z}) .

\smallskip
It is evident from the definition of $\omega_n^*$ that the more half-lines 
attached to the common vertex, the higher the frequencies $\omega$ need to 
be to guarantee existence. Moreover, no positive discontinuous stationary profile in $\dom$ 
exists whenever
\[
\omega< \omega_1^*=\frac{\big(1+(N-1)^{2/3}\big)^{3}}{Z^2}=\min_{1\le n\le N-1}\omega_n^*,
\]
regardless of the choice of $n$.

\bigskip
In Section~\ref{sec:exists} we prove the following  main existence  result of discontinuous standing waves  at the vertex $\nu=L$ for the cubic-NLS as a consequence of Proposition \ref{prop:existencet}.

\begin{figure}
    \centering
    \makebox[\textwidth][c]{\resizebox{\linewidth}{!}{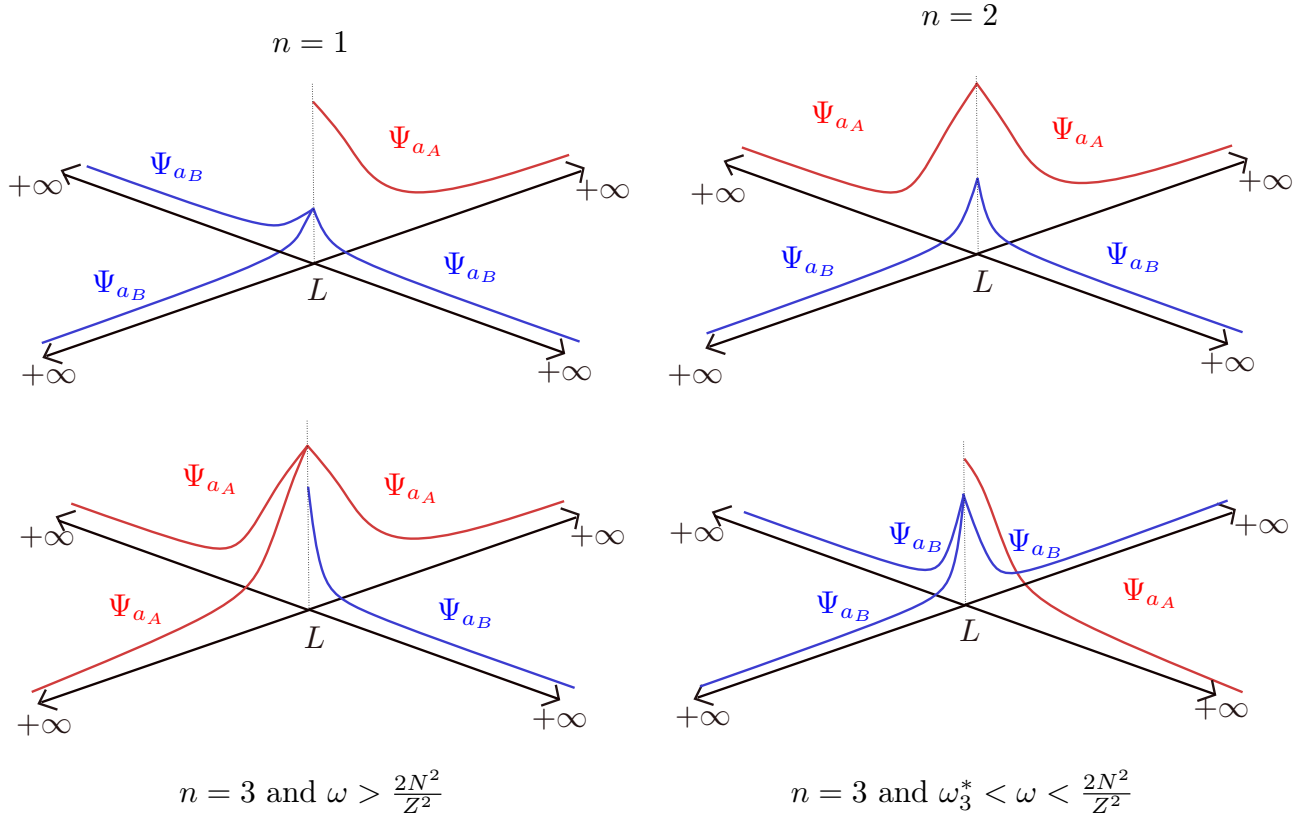}}
    \caption{For $N=4$ schematic of all the different positive profiles configurations in $\dom$ which are discontinuous at the vertex and where $\Psi_{a_A}$ and $\Psi_{a_B}$ are tail profiles \eqref{soliton} with shift of type $A$ and $B$ respectively.}
    \label{fig:placeholder}
\end{figure}

\begin{theorem}\label{thm:existence}
Let $Z<0$, $N\ge 2$ and fix $n\in\{1,\dots,N-1\}$. Let $\omega_n^*>0$  given by \eqref{eq:two-star-intro} (the existence-threshold value). Then for every 
$\omega>\omega_n^*$ there exist two shifts $0<a_A(\omega)<a_B(\omega)$ such 
that the stationary profile $\Theta_{\omega,n}$ defined in \eqref{eq:thetan} 
belongs to $\dom$. Moreover, for $n$ fixed, the map
\[
(\omega_n^*,\infty)\ni \omega \longmapsto e^{it\omega}\Theta_{\omega,n} \in D(\mathcal{H}_{Z}^{\delta'})
\]
defines a smooth curve of standing waves for the NLS equation \eqref{NLS} on $\dg$ that are discontinuous at the vertex.

The specific assignment of shift types ({\it the configuration of the profile on $\mathcal G$}) is given as follows:
\begin{enumerate}
\item[1)] if $2n\le N$ then $a_1=a_A(\omega)$ and $a_N=a_B(\omega)$. In other words, the configuration always produces exactly $n$ type-$A$ profiles together with $N - n$ type-$B$ profiles, for any $\omega> \omega_n^*$.
\item[2)] if $2n>N$ we have the existence of the transition frequency $\widehat{\omega}=2N^2/Z^2$ and
\[
\begin{cases}
a_1=a_A(\omega),\ \ a_N=a_B(\omega), & \text{if } \omega> \widehat{\omega},\\[4pt]
a_1=a_B(\omega),\ \ a_N=a_A(\omega), & \text{if } \omega_n^*< \omega<\widehat{\omega}.
\end{cases}
\]
At $\omega=\widehat{\omega}$ the two shifts types reverse, collapsing into a continuous configuration, $a_A(\omega)=a_B(\omega)=\operatorname{arctanh}(1/\sqrt{2})$.
\end{enumerate}
\end{theorem}

With regard only to the configuration of these stationary solutions, Theorem \ref{thm:existence} deserves some useful comments. For $N\ge 2$ fixed, we have the following number of discontinuous-at-the-vertex standing waves profiles (taking or not permutations of the edges into account):
\begin{enumerate}
\item[a)] for $\omega > 2N^2/Z^2$: there are exactly $N-1$ families of different profiles,
\item[b)] for $\omega < 2N^2/Z^2$: there are also $N-1$  families of profiles not including permutations of the edges. If we include permutations of the edges, we get $n$ different families of discontinuous profiles where $n$ represents the largest amount of elements in the set $$\{k\in\{1,\dots,N-1\} \mid 2k\le N\}$$ (see Proposition \ref{nonexis} and Remark~\ref{rem:w=2N^2/Z^2}). In the case $N=4$ showed in Figure \ref{fig:placeholder}, we have exactly $2$ different families of  configurations (cases $n=1,2$) because the configuration for $n=3$ agrees with that of $n=1$ (it is mandatory to clarify that the configurations are the same but the shifts values $a_A(\omega),\ a_B(\omega)$ are not, so $\Theta_{\omega,1}\neq \Theta_{\omega,3}$ even after permutation of the edges). 
\end{enumerate}

\bigskip
In the second part of this work, we develop our stability framework. We begin by noting that the basic symmetry associated with the NLS model \eqref{NLS} on a star graph is phase invariance: if $\textbf{U}$ is a solution of \eqref{NLS}, then $e^{i\theta}\textbf{U}$ is also a solution for any $\theta\in[0,2\pi]$. It is therefore natural to define orbital stability for the NLS model through this symmetry, as follows:

\begin{definition}\label{dsta}
A standing wave of the form $\mathbf{U}(x,t)=e^{i\omega t}\Theta(x)$ is called
\textit{orbitally stable} in a Banach space $X$ if the following holds: given any
$\varepsilon>0$, there exists $\eta>0$ such that whenever an initial datum
$\mathbf{U}_0\in X$ satisfies
\[
\|\mathbf{U}_0-\Theta\|_{X}<\eta,
\]
the corresponding solution $\mathbf{U}(t)$ of \eqref{NLS} with
$\mathbf{U}(0)=\mathbf{U}_0$ is globally defined for all $t\in\mathbb{R}$ and
remains close to the orbit of $\Theta$ in the sense that
\[
\sup_{t\in\mathbb{R}}\inf_{\theta\in\mathbb{R}}
\bigl\|\mathbf{U}(t)-e^{i\theta}\Theta\bigr\|_{X}<\varepsilon.
\]
If this property fails, then the standing wave $\mathbf{U}(x,t)=e^{i\omega t}\Theta$
is said to be \textit{orbitally unstable} in $X$.
\end{definition}

In the framework of \eqref{NLS}, the space $X$ in Definition~\ref{dsta} will be
taken as the energy space $H^{1}(\dg)$, endowed with the action of
$-\Delta$ on the domain $D(\mathcal{H}_Z^{\delta'})$.

The (in)stability of the discontinuous standing waves profiles constructed in Theorem \ref{thm:existence}  is investigated through the  framework developed by Grillakis, Shatah, and Strauss (GSS) \cite{GSS2}. In order to apply the GSS approach, one must rely on a suitable local and/or global well-posedness theory for the nonlinear Schr\"odinger flow on the graph. 
Such a local well-possedness theory is already available in the setting of $\delta'$ interactions, having been established by Angulo and Goloshchapova in \cite{AngGol17b} for $p>1$ (see also \cite{G}). The global well-posedness is obtained for $1<p<5$ by extending the local theory through the standard continuation argument based on conserved quantities, namely the energy and charge functionals
\begin{equation}\label{Ener}
E_Z(\mathbf{U}) = \frac{1}{2} \|\nabla \mathbf{U}\|^2_{L^2(\dg)} - \frac{1}{p+1} \|\mathbf{U}\|^{p+1}_{L^{p+1}(\dg)} + \frac{1}{2Z} \left| \sum_{i=1}^N \psi_i(L) \right|^2, 
\qquad Z \neq 0 \quad\quad\quad (\text{energy}),
\end{equation}
and
\begin{equation}\label{mass}
Q(\mathbf{U}) = \|\mathbf{U}\|^2_{L^2(\dg)}, \hskip3.3in (\text{mass}),
\end{equation}
both naturally defined in the energy space $H^1(\dg)$. 

\medskip
The spectral ingredients required by the GSS - method are developed in Section~\ref{sec:spectral} for $p=3$, where we carry out a detailed analysis of the linearized operators arising from the second variation of the action functional 
$$
\mathbf{S}(\mathbf{U})=E_Z(\mathbf{U})+\omega Q(\mathbf{U}),
\qquad \mathbf{U}\in H^1(\dg),
$$
around the stationary profiles $\Theta_{\omega,n}$, namely, the Morse and nullity indices of $\mathbf{S}''(\Theta_{\omega,n})$. The Morse-index analysis identifies all the directions of instability and we classify them in terms of relations between $N\ge 2$, $n\in\{1,\dots,N-1\}$ and the size of frequency $\omega>\omega_n^*$ (see Theorem \ref{thm:morse1} and Table~\ref{table:2} in Appendix D). A sharp description of the negative-eigenvalue count  (Morse index) of the scalar Neumann operator on a half-line
\begin{equation}\label{eq:Lplus-halfline-p-intro}
\mathfrak L_{+,a}
:=
-\partial_x^2+\omega-3\,\Psi^{2}_{a}(x),
\qquad
D(\mathfrak L_{+,a})
=
\{u\in H^2([L,\infty)):\ u'(L)=0\},
\end{equation}
for any shift $a>0$ and $\omega>0$ is required in our spectral analysis and therefore recorded in Lemma \ref{lem:LA-negative-generalp} in the Appendix. More precisely, $n(\mathfrak L_{+,a})=1$ for $a\in (0, a_*)$ and $n(\mathfrak L_{+,a})=0$ for $a\geqq a_*$, where $a_*=\operatorname{arctanh}(\tfrac{1}{\sqrt{2}})$ with $\Psi''_{a_*}(L)=0$. Note $a_*$ is also exactly the shift corresponding to the continuous inversion of the shift types that happens at frequency $\widehat{\omega}$ when $2n>N$ (see Theorem \ref{thm:existence}). Based on an analysis on the half-line, the nullity index is obtained via a quadratic-form analysis and  Sturm-Liouville theory (see Propositions \ref{prop:l2} and \ref{prop:ker-L1Z}). {\it We note that this  spectral analysis remains unchanged for any $p > 1$ and  for any {\it a priori} discontinuous-at-the-vertex standing waves profiles in $D(\mathcal{H}_Z^{\delta'})$ for the NLS  \eqref{NLS} (see Section \ref{sec:spectral-generalp}}.

Our spectral approach is based on the useful orthogonal direct sum decomposition of $L^2(\dg)$ given for $n\in \{1,\dots,N-1\}$ by
\begin{equation}
    \label{eq:l2orthintro} L^2(\dg)
=
X_{\mathrm{mean}}(n)\ \oplus\ X_{1,\perp}(n)\ \oplus\ X_{N,\perp}(n),
\end{equation}
where for the indices sets $
A\equiv A(n):=\{1,\dots,n\}$, $B\equiv B(n):=\{n+1,\dots,N\}$ one has two decoupling subspaces \[
\begin{split}
    X_{1,\perp}(n):=\Big\{\mathbf v\in L^2(\dg) \Bigr|\, v_j\equiv 0\ \forall j\in B,\;
\sum_{j\in A} v_j(x)\equiv 0\Big\}, \\
X_{N,\perp}(n)=\Big\{\mathbf v\in L^2(\dg)\Bigr|\, v_j\equiv0\ \forall j\in A,\ \ 
\sum_{j\in B} v_j(x)=0\Big\}
\end{split}
\] and a coupling space $$X_{\mathrm{mean}}(n):=\Big\{\mathbf v\in L^2(\dg)\Bigr|\,   v_1\equiv \cdots \equiv v_n,\;\;
v_{n+1}\equiv \cdots \equiv v_N\Big\}.$$ In this form, we split our spectral study of  $\mathbf{S}''(\Theta_{\omega,n})$ to each of the spaces $D(\mathcal{H}_Z^{\delta'})\cap X_{\mathrm{mean}}(n)$, $D(\mathcal{H}_Z^{\delta'})\cap X_{1,\perp}(n)$ and $D(\mathcal{H}_Z^{\delta'})\cap X_{N,\perp}(n)$.
 The application of the GSS criterion in Theorem~\ref{2main} is then reduced to the sign analysis of the slope $\partial_\omega Q(\Theta_{\omega,n})$ along the existence curve. 
The map $\omega\mapsto Q(\Theta_{\omega,n})$ is smooth on 
$(\omega_n^*,\infty)$, but its monotonicity depends on the geometry of the 
configuration $\Theta_{\omega,n}$. From Theorem~\ref{thm:slope-condition} (summarized in Table~\ref{table:3} in Appendix D)
the sign of the slope $\partial_\omega Q(\Theta_{\omega,n})$ on $(\omega_n^*,\infty)$
also depends on the cluster configuration: when $2n\ge N$, the slope is strictly positive throughout $(\omega_n^*,\infty)$, while when $2n<N$, the slope is strictly negative just above the existence threshold
and becomes strictly positive for large $\omega$. Hence there is an intermediate
frequency $\omega_n^\sharp>\omega_n^*$ (see \eqref{velha}) at which the slope vanishes and changes sign,
separating the instability regime $\omega_n^*<\omega<\omega_n^\sharp$ from the
(potentially) stable regime $\omega>\omega_n^\sharp$ (see Tables \ref{table:1} and \ref{table:3} in Appendix D).

\smallskip
In Section~\ref{sec:mainproof} we prove the following main (in)stability result for the cubic-NLS   as an application of Propositions~\ref{prop:l2}, \ref{prop:ker-L1Z}, Theorem~\ref{thm:morse1}, and Theorem~\ref{2main} in Appendix.

\begin{theorem}\label{thm:insta}
Let $Z<0$, $N\ge 2$ and $n\in\{1,\dots,N-1\}$ fixed. Suppose $\omega>\omega_n^*$ and  $\omega_n^\sharp$ defined in \eqref{velha}. Then the standing wave 
$e^{i\omega t}\Theta_{\omega,n}$ for the NLS equation \eqref{NLS} provided by 
Theorem~\ref{thm:existence} is: 
\begin{enumerate}
\item orbitally stable in $H^1(\dg)$ in the cases \begin{enumerate}
    \item $n=1$ and $N=2$.
    \item $n=1$, $N>2$ and $\omega>\omega_n^\sharp$. 
\end{enumerate}
\item orbitally stable in $\xm\cap H^1(\dg)$ in the cases \begin{enumerate}
    \item $2n=N$.
    \item $2n>N$ and $\omega>2N^2/Z^2.$
    \item $2n<N$ and $\omega>\omega_n^\sharp$. 
\end{enumerate}
\item orbitally unstable in $H^1(\dg)$ in the cases \begin{enumerate}
    \item $2n=N$, $N>2$.
    \item $2n<N$, $N>2$, $n\ge 2$ and $\omega>\omega_n^\sharp$.
    \item $2n<N$ and $\omega<\omega^\sharp_n$.
    \item $2n>N$, $N>2$ and $\omega>2N^2/Z^2$. 
    \item $2n>N$ and $\omega<2N^2/Z^2$.
\end{enumerate}
\end{enumerate}
\end{theorem}

For the reader's convenience, in Appendix~\ref{ap:D}, together with the Morse index and slope condition analysis, we summarize the cases present in Theorem~\ref{thm:insta} in a comprehensive table, Table \ref{table:1}, including information on why our approach is inconclusive in some of the threshold  values $\omega_n^\sharp$, $\omega_n^*$ and $\omega=2N^2/Z^2.$

\medskip
As a rough explanation of the conclusion of Theorem~\ref{thm:insta}, we note the inflection point of the soliton is geometrically degenerate, it is placed at the one shift where the soliton derivative is ``flat'' at the vertex and that flatness is exactly what makes the linearized operator lose the spectral information the GSS method needs. Whenever a configuration drifts toward $a_*$, our tools go silent. This happens at the transition frequency $\widehat{\omega}=2N^2/Z^2$ (where for the unbalanced case $2n>N$ the type-$A$ and type-$B$ groups literally swap roles by sliding through the inflection point and momentarily coinciding) and at the existence threshold $\omega_n^*$ (where the shifts are pinned to a critical point of the auxiliary function $\alpha_n$ and the analysis again degenerates). Away from these special frequencies the geometry is ``generic'' and the broad lesson of the stability tables is that discontinuous clustering is overwhelmingly destabilizing: almost every configuration we can decide is orbitally unstable in $H^1(\dg)$. The single exception is the most lopsided split of all: $n=1$ (one half-line wearing one profile type and all $N-1$ others sharing the opposite type) which becomes orbitally stable once the frequency climbs past the intermediate value $\omega_n^\sharp$ (for $N=2$ this threshold already coincides with the existence threshold $\omega_n^*$, so stability holds throughout). In short, splitting the vertex into two genuinely different clusters almost always breaks the wave, and the lone survivor is the configuration that is as close as possible to not splitting at all.

For the case $p>1$, $p\neq 3$, in Section ~\ref{sec:generalp}  we present the construction of discontinuous-at-the vertex stationary solutions (Theorem \ref{thm:existence-p})  and an associated  (in)stability  theory (Theorem \ref{thm:insta-highfreq-p}) with  some restrictions on the frequency $\omega$ (see Remark \ref{rem:generalp-remarks}).

\medskip
\noindent\textbf{Organization of the article}. In Section~\ref{sec:not} we describe some basic notation. In Section~\ref{sec:local} we develop a useful decomposition of $L^2(\dg)$ and establish the local and/or global well-posedness theory. Section~\ref{sec:exists} is devoted to the existence of discontinuous standing 
waves in $D(\mathcal{H}_Z^{\delta'})$ and the proof of Theorem~\ref{thm:existence}. 
Section~\ref{sec:spectral} carries out the spectral analysis of the linearized 
operators $\mathcal L_{\pm,Z}$ obtained from the second variation $\mathbf{S}''(\Theta_{\omega,n})$ of the action functional $\mathbf{S}$: positivity and kernel of $\mathcal L_{-,Z}$ 
(Proposition~\ref{prop:l2}), triviality of $\ker(\mathcal L_{+,Z})$ 
(Proposition~\ref{prop:ker-L1Z}), and the Morse-index analysis
(Theorem~\ref{thm:morse1}).
Section~\ref{sec:insta} establishes the slope condition along the existence 
branch profiles (Theorem~\ref{thm:slope-condition}), in Section~\ref{sec:mainproof}  we proves the main (in)stability result, 
Theorem~\ref{thm:insta}. In Section~\ref{sec:generalp} we exhibit the main differences on the analysis for general values of $p>1$, $p\neq 3$. In Appendices~\ref{ap:1} and \ref{ap:2} we present a crucial study of half-line operators in terms of their Morse index and their regularity with respect to the $H^2$-eigenvalues. In Appendix~\ref{ApA} we recall the GSS approach for (in)stability analysis. Finally Appendix~\ref{ap:D} summarized the main results for $p=3$ in fast-reference tables.

\section{Notation}\label{sec:not}
Let $-\infty\leq a<b\leq\infty$. We denote by $L^2(a,b)$  the  Hilbert space equipped with the inner product $(u,v)=\int\limits_a^b u(x)\overline{v(x)}dx$.  By $H^n(\Omega)$  we denote the classical  Sobolev spaces on $\Omega\subset \mathbb R$ with the usual norm. For a metric star graph $\dg$, we define the  $ L^ p(\dg)$-spaces by $
 L^ p(\dg)= \bigoplus_j L^p( L, +\infty)$, $p>1$, with the natural norms. For $\textbf{U}=(g_1,\dots,g_N), \textbf{V}=(\tilde{g}_1,\dots,\tilde{g}_N) \in L^2(\dg)$,  the natural inner product in $L^2(\dg)$  is defined by
 $\langle\textbf{U}, \textbf{V}\rangle= \sum_{j=1}^N \int_L^{\infty} g_j(x)\overline{\tilde{g}_j(x)}dx$.  For any $n\geqq 0$, we have the  $H^n(\dg)$-Sobolev spaces,
 $
 H^n(\dg)=\bigoplus_{j=1}^ N H^ n( L, +\infty).
 $ 

 Let $T$ be a  closed densely defined symmetric operator in the Hilbert space $H$. The domain of $T$ is denoted by $D(T)$. $\sigma(T)$ represents the spectrum of $T$, $\sigma_{ess}(T)$ the essential spectrum of $T$. With $T^*$ we denote the adjoint operator of $T$.  The number of negative eigenvalues counting multiplicities (or Morse index) of $T$ is denoted by  $n(T)$.

\section{Preliminaries}\label{sec:local}
We begin this section introducing a helpful decomposition of the space $L^2(\dg)$, which will be a fundamental piece in our spectral study in Section  \ref{sec:spectral}.
\subsection{Orthogonal decomposition}
For $n\in\{1,\dots,N-1\}$ fixed define the indices sets
\[
A\equiv A(n):=\{1,\dots,n\},\qquad B\equiv B(n):=\{n+1,\dots,N\}.
\] 
We introduce the following $n$-dependent three subspaces of $L^2(\dg)$:

\begin{align*}
X_{\mathrm{mean}} \equiv X_{\mathrm{mean}}(n)
&:=\Big\{\mathbf v=(v_j)\in L^2(\dg) \Bigm|\  v_1\equiv \cdots \equiv v_n,\;\;
v_{n+1}\equiv \cdots \equiv v_N\Big\},\\[2mm]
X_{1,\perp}\equiv X_{1,\perp}(n)
&:=\Big\{\mathbf v=(v_j)\in L^2(\dg) \Bigm| \ v_j\equiv 0\ \forall j\in B,\ \ 
\sum_{j\in A} v_j(x)\equiv 0\Big\},\\[2mm]
X_{N,\perp}\equiv X_{N,\perp}(n)
&:=\Big\{\mathbf v=(v_j)\in L^2(\dg)\Bigm|\  v_j\equiv0\ \forall j\in A,\ \ 
\sum_{j\in B} v_j(x)=0\Big\}.
\end{align*}

Note in case the cardinality of $A$ or $B$ is one, then the associated decoupling space degenerates to $X_{j,\perp}=\{0\}$, $j=1$ or $j=N$ respectively. We may refer to the spaces $\xa$ and $\xb$ as the \textit{decoupling} spaces.

\begin{lemma}
\label{lem:cluster-decomp}
Let $n\in \{1,\dots,N-1\}$ be fixed. The Hilbert space $L^2(\dg)$ admits the orthogonal direct sum decomposition
\begin{equation}
    \label{eq:l2orth} L^2(\dg)
=
X_{\mathrm{mean}}\ \oplus\ X_{1,\perp}\ \oplus\ X_{N,\perp}.
\end{equation}
\end{lemma}

\begin{proof}
We proceed in two steps.

\smallskip
\noindent\textbf{Orthogonality.}
Let $\mathbf v=(v_j)\in X_{\mathrm{mean}}$ and $\mathbf w=(w_j)\in X_{1,\perp}$. Then
\[
\langle \mathbf v,\mathbf w\rangle_{L^2(\dg)}
=\sum_{j\in A}\int_L^\infty v_1(x)\overline{w_j(x)}\,dx
=\int_L^\infty v_1(x)\overline{\sum_{j\in A} w_j(x)}\,dx=0,
\]
since $\sum_{j\in A}w_j(x)=0$. The same argument shows that
$X_{\mathrm{mean}}\perp X_{N,\perp}$.

Moreover, if $\mathbf v\in X_{1,\perp}$ and $\mathbf w\in X_{N,\perp}$, then
\[
\langle \mathbf v,\mathbf w\rangle_{L^2(\dg)}
=\sum_{j=1}^N\int_L^\infty v_j(x)\overline{w_j(x)}\,dx=0,
\]
since $\mathbf v$ vanishes on $B$ and $\mathbf w$ vanishes on $A$.
Hence the three subspaces are pairwise orthogonal.

\smallskip
\noindent\textbf{Completeness.}
Let $\mathbf v=(v_1,\dots,v_N)\in L^2(\dg)$ be arbitrary. Define the averages
\[
\bar v_A(x):=\frac1n\sum_{j\in A} v_j(x),\qquad
\bar v_B(x):=\frac1{N-n}\sum_{j\in B} v_j(x),
\]
and set
\[
\mathbf v_{\mathrm{mean}}:=
(\bar v_A,\dots,\bar v_A,\bar v_B,\dots,\bar v_B),
\]
where $\bar v_A$ is repeated $n$ times and $\bar v_B$ is repeated $N-n$ times.

Define further
\[
\mathbf v_{A,\perp}:=(v_1-\bar v_A,\dots,v_n-\bar v_A,0,\dots,0),
\]
and
\[
\mathbf v_{B,\perp}:=(0,\dots,0,v_{n+1}-\bar v_B,\dots,v_N-\bar v_B).
\]
Then clearly
\[
\mathbf v=\mathbf v_{\mathrm{mean}}+\mathbf v_{A,\perp}+\mathbf v_{B,\perp}.
\]
Moreover,
\[
\sum_{j\in A} (v_j-\bar v_A)=0,\qquad \sum_{j\in B} (v_j-\bar v_B)=0,
\]
so $\mathbf v_{A,\perp}\in X_{1,\perp}$ and $\mathbf v_{B,\perp}\in X_{N,\perp}$, while
$\mathbf v_{\mathrm{mean}}\in X_{\mathrm{mean}}$ by construction.
\end{proof}

\begin{lemma}
\label{lem:domain-decomp-deltaprime}
Let $P_{\mathrm{mean}},P_{1,\perp},P_{N,\perp}$ be the orthogonal projections from $L^2(\dg)$ onto the respective subspaces in \eqref{eq:l2orth}
given component-wise for $\mathbf v=(v_j)$, with $\bar v_A$ and $\bar v_B$ defined above, by
\begin{align}\label{eq:projs}
\nonumber(P_{\mathrm{mean}}\mathbf v)_j(x)
&=
\begin{cases}
\bar v_A(x), & j\in A,\\
\bar v_B(x), & j\in B,
\end{cases}
\\[1mm]
(P_{1,\perp}\mathbf v)_j(x)
&=
\begin{cases}
v_j(x)-\bar v_A(x), & j\in A,\\
0, & j\in B,
\end{cases}
\\[1mm]
\nonumber(P_{N,\perp}\mathbf v)_j(x)
&=
\begin{cases}
0, & j\in A,\\
v_j(x)-\bar v_B(x), & j\in B.
\end{cases}
\end{align}
Then $D(\mathcal H^{\delta'}_Z)$ is invariant by the projections above. Moreover, 
$D(\mathcal H^{\delta'}_Z)$ decomposes as the orthogonal direct sum (with the possibility of at most one of the decoupling spaces being zero):
\begin{equation}
    \label{eq:decompdom}
    D(\mathcal H^{\delta'}_Z)
=
\big(D(\mathcal H^{\delta'}_Z)\cap X_{\mathrm{mean}}\big)
\ \oplus\
\big(D(\mathcal H^{\delta'}_Z)\cap X_{1,\perp}\big)
\ \oplus\
\big(D(\mathcal H^{\delta'}_Z)\cap X_{N,\perp}\big),
\end{equation}
and for each $\mathbf v\in D(\mathcal H^{\delta'}_Z)$ we have the unique decomposition
\[
\mathbf v=P_{\mathrm{mean}}\mathbf v+P_{1,\perp}\mathbf v+P_{N,\perp}\mathbf v
\quad\text{with pairwise orthogonal summands in }L^2(\dg).
\]

\end{lemma}

\begin{proof}
 Let $\mathbf v\in D(\mathcal H^{\delta'}_Z)$ and set $d:=v_1'(L)=\cdots=v_N'(L)$. 
First, for $P_{\mathrm{mean}}\mathbf v$ we compute
\[
(P_{\mathrm{mean}}\mathbf v)_j'(L)
=
\begin{cases}
\frac1n\sum_{k\in A} v_k'(L)=d, & j\in A,\\[1mm]
\frac1{N-n}\sum_{k\in B} v_k'(L)=d, & j\in B,
\end{cases}
\]
so derivative matching holds with the same common derivative $d$. Furthermore,
\begin{equation*}
    \begin{split}
        \sum_{j=1}^N (P_{\mathrm{mean}}\mathbf v)_j(L)
&=
\sum_{j\in A}\bar v_A(L)+\sum_{j\in B}\bar v_B(L)
\\&=
\sum_{j\in A} v_j(L)+\sum_{j\in B} v_j(L)
=
\sum_{j=1}^N v_j(L)=Z\,d,
    \end{split}
\end{equation*}
and since $(P_{\mathrm{mean}}\mathbf v)_1'(L)=d$ we conclude $P_{\mathrm{mean}}\mathbf v\in D(\mathcal H^{\delta'}_Z)$.

\smallskip
Next consider $P_{1,\perp}\mathbf v$. For $j\in A$,
\[
(P_{1,\perp}\mathbf v)_j'(L)=v_j'(L)-\bar v_A'(L)=d-d=0,
\]
while for $j\in B$ we have $(P_{1,\perp}\mathbf v)_j\equiv 0$ and thus also $(P_{1,\perp}\mathbf v)_j'(L)=0$.
Hence derivative matching holds with common value $0$. Moreover,
\[
\sum_{j=1}^N (P_{1,\perp}\mathbf v)_j(L)
=\sum_{j\in A}\big(v_j(L)-\bar v_A(L)\big)
=\sum_{j\in A}v_j(L)-n\bar v_A(L)=0,
\]
so the $\delta'$ relation becomes $0=Z\cdot 0$, and therefore $P_{1,\perp}\mathbf v\in D(\mathcal H^{\delta'}_Z)$.
The proof for $P_{N,\perp}\mathbf v$ is identical.

\smallskip
Now, since $L^2(\dg)=X_{\mathrm{mean}}\oplus X_{1,\perp}\oplus X_{N,\perp}$ orthogonally and
$\mathbf v=P_{\mathrm{mean}}\mathbf v+P_{1,\perp}\mathbf v+P_{N,\perp}\mathbf v$ for every $\mathbf v\in L^2(\dg)$, 
the invariance of the subspaces imply that the same decomposition holds on $D(\mathcal H^{\delta'}_Z)$ with summands in the intersections
$D(\mathcal H^{\delta'}_Z)\cap X_{\mathrm{mean}}$, $D(\mathcal H^{\delta'}_Z)\cap X_{1,\perp}$, and $D(\mathcal H^{\delta'}_Z)\cap X_{N,\perp}$.
\end{proof}

\subsection{Local and global well-posedness}
In the sequel we describe a result proved in \cite{AngGol17b} (see a more detailed description in \cite[Section~2]{G}) with regard the well-posedness of \eqref{NLS} with $\delta'$ interactions.

\begin{theorem}[\cite{AngGol17b}]\label{global}
   Let $p>1$. Then for any $\mathbf{U}_0\in H^1(\dg)$ there exists $T>0$ such that equation \eqref{NLS} has a unique solution $\mathbf{U}\in C([-T,T];H^1(\dg))$ satisfying $\mathbf{U}(0)=\mathbf{U}_0.$ For each $T_0\in (0,T)$ the mapping \[
    H^1(\dg)\ni \mathbf{U}_0 \mapsto \mathbf{U}\in C([-T_0,T_0];H^1(\dg))
    \] is continuous. In particular, for $p>2$ this map is at least $C^2$. 

    \medskip
    Moreover the conservation of energy and charge in \eqref{Ener}-\eqref{mass} holds:
    \begin{equation*}
    E_Z(\mathbf{U}(t))=E_{Z}(\mathbf{U}_0), \qquad  Q(\mathbf{U}(t))=Q(\mathbf{U}_0), \qquad t\in[-T,T]. 
    \end{equation*} Consequently, for $1<p<5$, $T$ can be taken arbitrarily large. 
\end{theorem}
\begin{proof}
    See Theorem~3.22 in \cite{AngGol17b}.
\end{proof}

\medskip
The following result will be essential in the proof of Theorem \ref{thm:insta}.

\begin{proposition}\label{lem:group}
    The unitary group $\{e^{-it\mathcal{H}^{\delta'}_Z}\}_{t \in \mathbb R}, $ associated to the linear part of \eqref{NLS}, preserves each of the subspaces in the decompositions \eqref{eq:l2orth} and \eqref{eq:decompdom}.
\end{proposition}
\begin{proof}
The main point in the analysis is to get an explicit formula for the resolvent associated with the operator $-\Delta=\mathcal{H}_Z^{\delta'}$ with domain $D(\mathcal{H}_Z^{\delta'})$ in \eqref{domainintro}, which in turn yield an explicit representation of the propagator $e^{-it\mathcal{H}^{\delta'}_Z}$ through functional calculus for  unbounded operators.

We star with the spectral structure  of $\mathcal{H}_Z^{\delta'}$. For $Z<0$ we have $\sigma_{ess}(\mathcal{H}_Z^{\delta'})=[0,\infty)$ and $\sigma_p(\hz)=\{-z_0^2\}=\{-N^2/Z^2\}$, with corresponding eigenfunction $\mathbf{V}_{z_0}=(e^{\frac{N}{Z}(x-L)},\dots,e^{\frac{N}{Z}(x-L)})$. When $Z>0$ the only difference is that $\sigma_p(\hz)=\emptyset$; see \cite[Remark~3.23]{AngGol17b}. We give the details of the proof of the Proposition only for the case $Z<0$, as its proof also establishes the case $Z>0$.

Next, for $\mathbf{V}=(v_1,\dots,v_N)\in L^2 (\mathcal G)$ and $\mu>0$ (without loss of generality) and $\mu^{2}\neq N^2/Z^2$,  the resolvent $\mathcal{R}_\mu=(\hz+\mu^2 I)^{-1}$ of $\hz$ has components 
\begin{equation}
    \label{eq:components}
    (\mathcal{R}_\mu \V)_j(x)=c_j e^{-\mu(x-L)}+\frac{1}{2\mu}\int_{L}^\infty v_j(y)e^{-\mu|x-y|}\,dy,\qquad x>L
\end{equation}
where the constants $c_j=c_j(\mu, \V)$ are uniquely chosen such $(\Phi_j)\equiv ((\mathcal{R}_\mu \V)_j) \in D(\mathcal{H}_Z^{\delta'})$. By convenience of the reader, we give an explicit formula for $c_j$. Setting $t_j(\mu):=\tfrac{1}{2}\int_L^\infty v_j(y)e^{-\mu y}\,dy$, the boundary values and derivatives at $x=L$ are
\[
\Phi_j(L)=c_j+\frac{t_j(\mu)}{\mu}, \qquad \partial_x\Phi_j(L)=-\mu c_j+t_j(\mu).
\]
Imposing $(\Phi_j)\in\dom$ yields the two conditions
\begin{equation}\label{eq:twocond}
    \kappa:=-\mu c_1 + t_1(\mu)=\cdots=-\mu c_N + t_N(\mu), \qquad \sum_{j=1}^N \left(c_j+\frac{t_j(\mu)}{\mu}\right)=Z\kappa.
\end{equation}
From the first condition in \eqref{eq:twocond} we have $c_j=\mu^{-1}(t_j(\mu)-\kappa)$ for every $j$. Substituting into the second condition and solving for $\kappa$ gives
\begin{equation}
    \label{eq:kappa}
\kappa=\frac{2\sum_{j=1}^Nt_j(\mu)}{N+Z\mu},
\end{equation}
and therefore the coefficients are determined explicitly by
\begin{equation}
    \label{eq:cj_closed}
    c_j=\frac{1}{\mu}\left(t_j(\mu)-\frac{2\sum_{k=1}^N t_k(\mu)}{N+Z\mu} \right).
\end{equation}

\medskip\noindent
For $Z<0$, by functional calculus for unbounded operators (\cite[Remark~3.2]{AngGol17b}), the group $e^{-it\hz}$ admits the representation for $\mathbf{V}\in L^2 (\mathcal G)$
\begin{equation}
    \label{eq:3.4}
    e^{-it\hz}\V(x)=\frac{i}{\pi}\int_{-\infty}^\infty e^{-it\tau^2}\tau\, (\mathcal{R}_{i\tau}\V)(x)\,d\tau + e^{itz_0^2}\langle \V,\V_{z_0} \rangle \V_{z_0}(x), 
\end{equation}
Then, since each subspace $X_{\mathrm{mean}}$, $X_{1,\perp}$, $X_{N,\perp}$ is closed in $L^2(\mathcal{G})$, it suffices by \eqref{eq:3.4} to verify invariance under the resolvent $\mathcal{R}_\mu$ and under the eigenfunction projection separately.

\bigskip
Let us first show the invariance of $X_{1,\perp}$.
Let $\V\in \xa$, so that $v_j\equiv 0$ for $j\in B$ and $\sum_{j\in A}v_j\equiv 0$. These conditions give $t_j(\mu)=0$ for $j\in B$ and $\sum_{j\in A}t_j(\mu)=0$, hence $\sum_{j=1}^Nt_j(\mu)=0$. By \eqref{eq:kappa} we get $\kappa=0$, and \eqref{eq:cj_closed} then gives $c_j=t_j(\mu)/\mu$ for $j\in A$ and $c_j=0$ for $j\in B$. From \eqref{eq:components} it follows immediately that $(\mathcal{R}_\mu\V)_j=0$ for $j\in B$, and
\[
\sum_{j\in A}(\mathcal{R}_\mu\V)_j(x)=\frac{e^{-\mu (x-L)}}{\mu}\sum_{j\in A}t_j(\mu)+\frac{1}{2\mu}\int_{L}^\infty e^{-\mu|x-y|}\sum_{j\in A}v_j(y)\,dy=0,
\]
so $\mathcal{R}_\mu\V\in \xa$. For the eigenfunction projection, since $\sum_{j\in A}v_j\equiv 0$ we have
\[
\langle \V,\V_{z_0}\rangle = \sum_{j=1}^N\int_L^\infty v_j(x)e^{\frac{N}{Z}(x-L)}dx = \sum_{j\in A}\int_L^\infty v_j(x)e^{\frac{N}{Z}(x-L)}dx = \int_L^\infty \sum_{j\in A}v_j(x)\, e^{\frac{N}{Z}(x-L)}dx = 0,
\]
so the eigenfunction term in \eqref{eq:3.4} vanishes. We conclude that $e^{-it\hz}\V\in\xa$.

\medskip
An analogous argument, with the roles of $A$ and $B$ interchanged, shows that $\mathcal{R}_\mu\V\in \xb$ and that the eigenfunction projection vanishes whenever $\V\in\xb$, so $e^{-it\hz}\V\in \xb$.

Let us now show the invariance of $X_{\mathrm{mean}}$.
Let $\V=(v_A,\dots,v_A,v_B,\dots,v_B)\in \xm$, where $v_A$ is the common value on $A$ and $v_B$ is the common value on $B$. Then
\[
t_j(\mu)\equiv t^A_\mu := \frac{1}{2}\int_{L}^\infty v_A(y)e^{-\mu y}\,dy \ \ \text{for } j\in A, \qquad t_j(\mu)\equiv t^B_\mu := \frac{1}{2}\int_{L}^\infty v_B(y)e^{-\mu y}\,dy \ \ \text{for } j\in B.
\]
From \eqref{eq:cj_closed} yields $c_j$ coefficients that are constant within each cluster with values
\[
c_j=\frac{t_\mu^A-\kappa}{\mu}=:c_A \quad \text{for } j\in A, \qquad c_j=\frac{t_\mu^B-\kappa}{\mu}=:c_B \quad \text{for } j\in B.
\]
Therefore, for each cluster $\#\in\{A,B\}$, all components indexed by $\#$ are equal with the expression
\[
(\mathcal{R}_\mu \V)_j(x)=c_\# \,e^{-\mu (x-L)}+\frac{1}{2\mu}\int_L^\infty v_\#(y)e^{-\mu|x-y|}\,dy, \qquad j\in \#,
\]
so $\mathcal{R}_\mu \V\in \xm$. Since $\V_{z_0}$ has all components equal, it belongs to $\xm$, and so does the eigenfunction term in \eqref{eq:3.4}. We conclude that $e^{-it\hz}\V\in\xm$.

\bigskip
Finally, the invariance of the decomposition \eqref{eq:decompdom} follows at from the invariance of each $L^2(\mathcal{G})$-subspace 
established above, we conclude that for each $\#\in\{\mathrm{mean};\,1,{\perp};\,N,{\perp}\}$, 
if $\V\in\dom\cap X_{\#}$ then $e^{-it\hz}\V\in\dom$ and $e^{-it\hz}\V\in X_{\#}$, 
hence $e^{-it\hz}\V\in\dom\cap X_{\#}$.
\end{proof}

\bigskip
For $1<p<5$, Theorem~\ref{global} already yields global well-posedness for arbitrary data in $H^1(\mathcal G)$.  For $p=5$, global existence is guaranteed only \emph{below certain mass threshold} we will discuss in the following. The arguments are classic, but for the reader's convenience, we provide an overview of the proof. Consider the conserved quantities \eqref{Ener}--\eqref{mass}, written for $p=5$ and $Z\neq 0$ as
\begin{equation}\label{eq:Ecrit}
E_Z(\mathbf U)=\frac12\|\mathbf U'\|_{L^2(\mathcal G)}^2+\frac{1}{2Z}\Big|\sum_{j=1}^N u_j(L)\Big|^2-\frac16\|\mathbf U\|_{L^6(\mathcal G)}^6,
\qquad
Q(\mathbf U)=\frac12\|\mathbf U\|_{L^2(\mathcal G)}^2.
\end{equation}

\begin{lemma}\label{lem:GN}
Let $\Psi$ be the half--soliton associated to \eqref{NLS} with $N=1$, $p=5$, i.e.\ the restriction to $[L,\infty)$ of the classical line soliton $\phi$ of $-\phi''+\phi=\phi^{5}$ peaked at the endpoint  ($\Psi'(L)=0$ and $\Psi=\Psi_{0, 5, 1}$ as in \eqref{eq:Theta-p}), and set
\[
K(\mathbb R_+):=\frac{3}{\|\Psi\|_{L^2(L,\infty)}^{4}}.
\]
Then the Gagliardo--Nirenberg inequality
\begin{equation}\label{eq:GN}
\|\mathbf U\|_{L^6(\mathcal G)}^6\le C\,\|\mathbf U'\|_{L^2(\mathcal G)}^2\,\|\mathbf U\|_{L^2(\mathcal G)}^4
\end{equation}
holds for any $\mathbf U \in H^1(\mathcal G)$, with the  sharp constant 
$$
C_{opt}= \underset{\mathbf U \in H^1(\mathcal G)-\{\mathbf 0\}} {sup} \frac{\|\mathbf U\|_{L^6(\mathcal G)}^6}{\|\mathbf U'\|_{L^2(\mathcal G)}^2\,\|\mathbf U\|_{L^2(\mathcal G)}^4}
$$
given by $C_{opt}=K(\R_+)$, attained by the single--edge half--soliton $(\Psi,0,\dots,0)$ and its permutations, rotations and critical rescalings.

\end{lemma}

\begin{proof}
We first record the sharp one--dimensional $L^2$--critical Gagliardo--Nirenberg
inequality on the half--line (see \cite{ASTcri}; the sharp constant on the
whole line is due to \cite{WE}, the half--line case following by even
reflection): for every $w\in H^1(L,\infty)$,
\begin{equation}\label{eq:GNhalf}
\|w\|_{L^6(L,\infty)}^6\le K(\mathbb R_+)\,\|w'\|_{L^2(L,\infty)}^2\,\|w\|_{L^2(L,\infty)}^4 .
\end{equation}
The optimizer of \eqref{eq:GNhalf} is the half--soliton $\Psi$, at which,
evaluating \eqref{eq:GNhalf} together with the identities
$\|\Psi\|_{L^2}^2=2\|\Psi'\|_{L^2}^2$ and $\|\Psi\|_{L^6}^6=3\|\Psi'\|_{L^2}^2$,
yields $K(\mathbb R_+)=3\|\Psi\|_{L^2}^{-4}$. Note by the critical scaling that
$\|\Psi_\omega\|_{L^2}=\|\Psi\|_{L^2}$ for any $\omega>0$, with $\Psi_\omega=\Psi_{0, 5, \omega}$ in  \eqref{eq:Theta-p}.

\smallskip
Applying \eqref{eq:GNhalf} on each edge and using $\|u_j\|_{L^2}^4\le\big(\sum_k\|u_k\|_{L^2}^2\big)^2=\|\mathbf U\|_{L^2(\mathcal G)}^4$,
\[
\|\mathbf U\|_{L^6(\mathcal G)}^6
\le K(\mathbb R_+)\sum_{j=1}^N\|u_j'\|_{L^2}^2\,\|u_j\|_{L^2}^4
\le K(\mathbb R_+)\,\|\mathbf U\|_{L^2(\mathcal G)}^4\sum_{j=1}^N\|u_j'\|_{L^2}^2
=K(\mathbb R_+)\,\|\mathbf U'\|_{L^2(\mathcal G)}^2\,\|\mathbf U\|_{L^2(\mathcal G)}^4,
\]
so $C_{opt}\le K(\mathbb R_+)$. Testing with $\mathbf W=(\Psi,0,\dots,0)$, where every norm reduces to that of $\Psi$ and \eqref{eq:GNhalf} is an equality, gives equality in \eqref{eq:GN}; hence $C_{opt}=K(\mathbb R_+)$ in $H^1(\dg)$.

\end{proof}

\begin{remark}\label{rmk:theta-not-extremal}
The symmetric $\boldsymbol\Theta=(\Psi,\dots,\Psi)$, although a genuine bound state and an element of $X_{\mathrm{mean}}(n)$ for every $n$, is \emph{not} an extremizer of \eqref{eq:GN}. In fact, using $\|\boldsymbol\Theta\|_{L^q(\mathcal G)}^q=N\|\Psi\|_{L^q(L, \infty)}^q$ one finds
\[
\frac{\|\boldsymbol\Theta\|_{L^6}^6}{\|\boldsymbol\Theta'\|_{L^2}^2\,\|\boldsymbol\Theta\|_{L^2}^4}
=\frac{1}{N^2}\,\frac{3}{\|\Psi\|_{L^2(L,\infty)}^4}=\frac{C_{opt}}{N^2}<C_{opt}.
\]
\end{remark}

\begin{theorem}\label{thm:GWPcrit}
Let $p=5$ and let $\Psi$ be the half--soliton of \eqref{NLS} with $N=1$.
If $\mathbf U_0\in H^1(\mathcal G)$ satisfies
\begin{equation}\label{eq:subthreshold}
\|\mathbf U_0\|_{L^2(\mathcal G)}<\|\Psi\|_{L^2(L,\infty)},
\end{equation}
then the corresponding solution $\mathbf U$ of \eqref{NLS} with $\mathbf U(0)=\mathbf U_0$ is global  with the properties $\mathbf U \in C(\mathbb R;H^1(\mathcal G))$ and $\displaystyle\sup_{t\in\mathbb R}\|\mathbf U(t)\|_{H^1(\mathcal G)}<\infty$.
\end{theorem}

\begin{proof}
By Theorem~\ref{global} the existence time depends only on $\|\mathbf U_0\|_{H^1(\mathcal G)}$, so the solution extends to a maximal interval $(-T_{\min},T_{\max})$ and the classical \emph{blow--up alternative} holds (see Cazenave \cite{caz}). Since $\|\mathbf U(t)\|_{L^2(\mathcal G)}=\|\mathbf U_0\|_{L^2(\mathcal G)}$ by charge conservation, it suffices to bound $\|\mathbf U'(t)\|_{L^2(\mathcal G)}$ uniformly in $t$.

\smallskip

Set $C=C_{\mathrm{opt}}$, the optimal constant in the Gagliardo--Nirenberg inequality \eqref{eq:GN} on $H^1(\mathcal G)$, and $M^\ast:=(3/C)^{1/4}=\|\Psi\|_{L^2(L,\infty)}$. By assumption $\|\mathbf U_0\|_{L^2(\mathcal G)}<M^\ast$. Put
\[
\eta:=\frac{\|\mathbf U_0\|_{L^2(\mathcal G)}}{M^\ast}\in[0,1),
\]
so that $C\,\|\mathbf U_0\|_{L^2(\mathcal G)}^4=C\,(M^\ast)^4\eta^4=3\eta^4$. Conservation of energy \eqref{eq:Ecrit} gives, at any time $t$,
\begin{equation}\label{eq:Eident}
\tfrac12\|\mathbf U'(t)\|_{L^2(\mathcal G)}^2
= E_Z(\mathbf U_0)-\frac{1}{2Z}\Big|\sum_{j=1}^N u_j(t,L)\Big|^2+\tfrac16\|\mathbf U(t)\|_{L^6(\mathcal G)}^6 ,
\end{equation}
while by \eqref{eq:GN} and charge conservation
\[
\tfrac16\|\mathbf U(t)\|_{L^6(\mathcal G)}^6
\le \frac{C}{6}\,\|\mathbf U_0\|_{L^2(\mathcal G)}^4\,\|\mathbf U'(t)\|_{L^2(\mathcal G)}^2
=\frac{\eta^4}{2}\,\|\mathbf U'(t)\|_{L^2(\mathcal G)}^2 .
\]
Substituting into \eqref{eq:Eident} yields
\begin{equation}\label{eq:dagger}
\frac{1-\eta^4}{2}\,\|\mathbf U'(t)\|_{L^2(\mathcal G)}^2
\le E_Z(\mathbf U_0)-\frac{1}{2Z}\Big|\sum_{j=1}^N u_j(t,L)\Big|^2 .
\end{equation}
If $Z>0$ the boundary term is nonpositive and $\|\mathbf U'(t)\|_{L^2(\mathcal G)}^2\le 2E_Z(\mathbf U_0)/(1-\eta^4)$. If $Z<0$, the trace inequality $|w(L)|^2\le 2\|w\|_{L^2}\|w'\|_{L^2}$ with Cauchy--Schwarz gives
\[
\Big|\sum_{j=1}^N u_j(t,L)\Big|^2\le 2N\,\|\mathbf U_0\|_{L^2(\mathcal G)}\,\|\mathbf U'(t)\|_{L^2(\mathcal G)},
\]
and Young's inequality (weight $\tfrac{1-\eta^4}{2}$) absorbs the gradient term, leaving
\[
\frac{1-\eta^4}{4}\,\|\mathbf U'(t)\|_{L^2(\mathcal G)}^2
\le E_Z(\mathbf U_0)+\frac{N^2\|\mathbf U_0\|_{L^2(\mathcal G)}^2}{(1-\eta^4)\,Z^2}<\infty .
\]
In either case $\sup_t\|\mathbf U'(t)\|_{L^2(\mathcal{G})} <\infty$, hence $\sup_t\|\mathbf U(t)\|_{H^1(\mathcal{G})}<\infty$, and the blow--up alternative forces $T_{\min}=T_{\max}=\infty$.
\end{proof}

\section{Existence of discontinuous standing waves in $\dom$}\label{sec:exists}

In this section we show, in the star graph $\dg$, the existence of up to $N-1$ (modulo reordering of the edges and depending of the position of the phase-velocity $\omega$ with respect to the threshold value $\omega=2N^2 /Z^2$,  see for instance Figure \ref{fig:placeholder} in the case $N=4$) different discontinuous-at-the-vertex standing waves solutions $e^{i\omega t}\Theta_\omega$, $\Theta_\omega=(\Psi_{a_i})_{i=1}^N$, for the NLS equation \eqref{NLS} with focusing cubic nonlinearity $p=3$ and  that belong to the set
\begin{equation}\label{Cdelta'}
D(\mathcal{H}_Z^{\delta'})=\Big\{\mathbf{U}=(\psi_1,\dots,\psi_N)\in H^2(\dg) 
\Bigm| \psi_1'(L)=\cdots=\psi_N'(L), \ \ \sum_{j=1}^N \psi_j(L)=Z \psi_1'(L) \Big\},
\end{equation}
provided $Z<0$. The results of this section extend \cite[Theorem~4.2]{G} to an optimal range of existence $\omega_n^*<2N^2/Z^2$. 

We consider the half-line profiles $\Psi_j$ as positive solutions of
\[
-\Psi''+\omega \Psi-\Psi^{3}=0, \qquad x\ge L.
\]
The unique positive decaying solution of this differential equation is the one-dimensional soliton
\[
\mathfrak{S}_\omega(y)=\sqrt{2\omega}\,\operatorname{sech}\!\big(\sqrt{\omega}\,y\big),\quad \omega>0
\]
so that the half-line profiles  $\Psi_{a_j}$ take the form
\begin{equation}\label{soliton}
\Psi_{a_j}(x) = \mathfrak{S}_\omega\!\Big(x-L+\tfrac{1}{\sqrt{\omega}}\,a_j\Big), 
\qquad x \geq L, \qquad 1\le j\le N,  \omega>0,
\end{equation}
where the shifts $a_j$ are to be determined via the $\delta'$-conditions in  \eqref{Cdelta'}.

\subsection{Derivative matching condition}
For each half-line profile $\Psi_{a_j}$ the boundary values at the vertex $\nu=L$ read
\[
\Psi_{a_j}(L)=\sqrt{2\omega}\,\operatorname{sech}(a_j),
\qquad
\Psi_{a_j}'(L)=-\omega\sqrt{2}\,\operatorname{sech}(a_j)\tanh(a_j).
\]
Since $Z<0$ and we will further need to verify the sum condition 
\[
0<\sum_{j=1}^N \Psi_{a_j}(L)=Z\,\Psi_{a_1}'(L),
\] 
it follows that $\Psi_{a_j}'(L)<0$ for all $j$, and therefore $a_j>0$. The continuity of derivatives at the vertex, $
\Psi_{a_1}'(L)=\cdots=\Psi_{a_N}'(L)$, 
is equivalent to the existence of a common parameter $\mu>0$ such that
\begin{equation}\label{eqmu}
\mu\equiv \operatorname{sech}(a_j)\tanh(a_j),
\qquad j=1,\dots,N.
\end{equation}
Setting $t_j:=\tanh(a_j)\in(0,1)$ and using $\operatorname{sech}^2(a_j)=1-t_j^2$, 
equation~\eqref{eqmu} reduces to
\begin{equation}\label{eq:gp}
t_j\sqrt{1-t_j^2}=\mu, \qquad j=1,\dots,N.
\end{equation}
The function $t\mapsto t\sqrt{1-t^2}$ attains its unique maximum on $(0,1)$ at 
$t=1/\sqrt{2}$, with value $1/2$, and for every $\mu\in(0,1/2)$ equation~\eqref{eq:gp} 
admits exactly two solutions (see Figure~\ref{fig:gp}),
\begin{figure}
    \centering
    \resizebox{0.5\linewidth}{!}{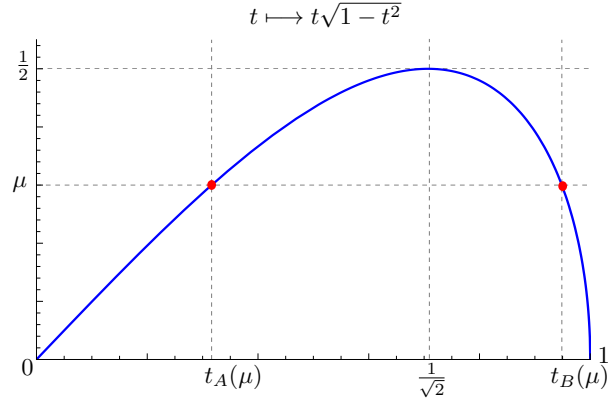}
    \caption{Plot of $t\sqrt{1-t^2}$ on $(0,1)$. For any $\mu\in (0,\frac{1}{2})$ there are two solutions $t_A(\mu)<1/\sqrt{2}<t_B(\mu)$.}
    \label{fig:gp}
\end{figure}
\begin{equation}
\label{eq:42a}
t_A(\mu)\in(0,1/\sqrt{2})
\quad\text{and}\quad
t_B(\mu)\in(1/\sqrt{2},1).
\end{equation}
Note $t_A(\mu)$ and $t_B(\mu)$ satisfy the convenient algebraic identity
\begin{equation}\label{eq:tAtB}
t_A^2(\mu)+t_B^2(\mu)=1.
\end{equation}

From the uniqueness of $t_A(\mu)$ and $t_B(
\mu)$, equation~\eqref{eqmu} admits, for each $j$, exactly two distinct 
shift solutions
\begin{equation}
\label{eq:42b}
a_A(\mu)\equiv \operatorname{arctanh}(t_A(\mu))
\quad\text{and}\quad
a_B(\mu)\equiv \operatorname{arctanh}(t_B(\mu)).
\end{equation}
Hence, for fixed $\mu$, each half-line shift $a_j$ in \eqref{soliton} must belong 
to the set $\{a_A(\mu),a_B(\mu)\}$. Moreover 
one has 
\begin{equation}\label{AB}
0<a_A<a_*<a_B, \;\;\text{with threshold}\;\;  a_*:=\operatorname{arctanh}(1/\sqrt{2}).
\end{equation} 
The latter motivates the following definition:

\begin{definition}\label{tipoAB}
The profile \eqref{soliton} is said to be a tail of type $A$ if 
$0<a_j<a_*$ and of type $B$ if $a_j>a_*$. We may refer to type $B$ 
tail profiles as \textit{pure tail} profiles in the sense $\Psi''(x)>0$ for all $x>L$. 
\end{definition}

It is worth mentioning that $a_*$ represents some critical scenarios in our study:
\begin{enumerate}\label{criticos}
\item[1)] the specificity of the profile type because $\Psi_j''(L)=0$ at $a_*$ and we may have a transition from type A to type B profiles (or vice versa) at $a_*$.
\item[2)] a threshold value for the Morse index of the half-line Neumann problem \eqref{eq:Lplus-halfline-p-intro} central to our stability study (see Lemma \ref{lem:LA-negative-generalp}).
\item[3)] the sum condition in \eqref{Cdelta'} for $\Psi_{a_*}$ implies $\omega=\frac{2N^2}{Z^2}$. This value of the phase velocity is a threshold value for Morse index values and (in)stability scenarios (see 
Theorem \ref{thm:morse1} and Theorem \ref{thm:insta} or Tables \ref{table:1}-\ref{table:2}).
\end{enumerate}

In the subsequent we seek instead for a clean formulation of the initial shifts 
depending on $\omega$. 

\subsection{The sum condition and its constraints}
Let $n\in\{0,\dots,N\}$ denote the number of half-lines selecting one of the two
shifts $a_A, a_B$, and $N-n$ the number selecting the complementary one. Without loss of 
generality assume they are grouped in the same edge ordering 
$a_1=\dots=a_n$ and $a_{n+1}=\dots=a_N$. Introducing the parameter
\[
\alpha=\alpha(\omega):=-Z\sqrt{\omega}>0,
\]
for $Z<0$ fixed, the sum condition
\begin{equation}\label{sumcond}
\sum_{j=1}^N \Psi_j(L)=Z\,\Psi_1'(L)
\end{equation}
reduces, after dividing through by a constant, to
\[
n\,\operatorname{sech}(a_1)+(N-n)\,\operatorname{sech}(a_N)
=\alpha\,\operatorname{sech}(a_1)\tanh(a_1).
\]
Using \eqref{eqmu} in the form $\operatorname{sech}(a_j)=\mu/t_j$, 
the parameter $\mu$ cancels out and the relation simplifies to
\begin{equation}\label{eq:kn}
\frac{n}{t_1}+\frac{N-n}{t_N}=\alpha.
\end{equation}
In particular if $t_1=t_N$, that is, if all half-line profiles are continuous at the vertex (mandatory for $n=0$,
$n=N$ or for a single half-line $N=1$), then \eqref{eq:kn} reduces to (see \cite{AngGol17b}) \[
t_1=\cdots=t_N=\frac{N}{-Z\sqrt{\omega}}, \qquad\text{implying}\quad \omega>\frac{N^2}{Z^2},
\] 
and we recover the profiles studied in Angulo and Goloshchapova  \cite{AngGol17b}.  Here, we are interested in discontinuous configurations, we assume from here that $t_1\neq t_N$, which forces us to consider only $N\ge 2$ and $n\in \{1,\dots,N-1\}$. By \eqref{eq:tAtB}, one has $t_N=\sqrt{1-t_1^2}$, so that the compatibility 
relation \eqref{eq:kn} reads
\begin{equation}\label{eq:alphan}
\text{for}\;\;\;\;\alpha_{n}(t):=\frac{n}{t}+\frac{N-n}{\sqrt{1-t^2}},\quad \alpha_{n}(t)=\alpha(\omega), 
\qquad t\in(0,1),
\end{equation}
where now $t_1=t$ and $t_N=\sqrt{1-t^2}$.

\medskip
A direct computation shows that for fixed $n\in\{1,\dots,N-1\}$ the function 
$\alpha_n$ is smooth on $(0,1)$, blows up to infinity as $t\to 0^+$ or 
$t\to 1^-$, and satisfies
\[
\alpha_n'(t)=-\frac{n}{t^2}+\frac{(N-n)\,t}{(1-t^2)^{3/2}}.
\]
Setting $\alpha_n'(t)=0$ yields a unique critical point $t_n^*\in(0,1)$, which 
is therefore the unique global minimizer of $\alpha_n$ (see Figure~\ref{fig:p=3}),
\begin{figure}
    \centering
    \resizebox{0.8\linewidth}{!}{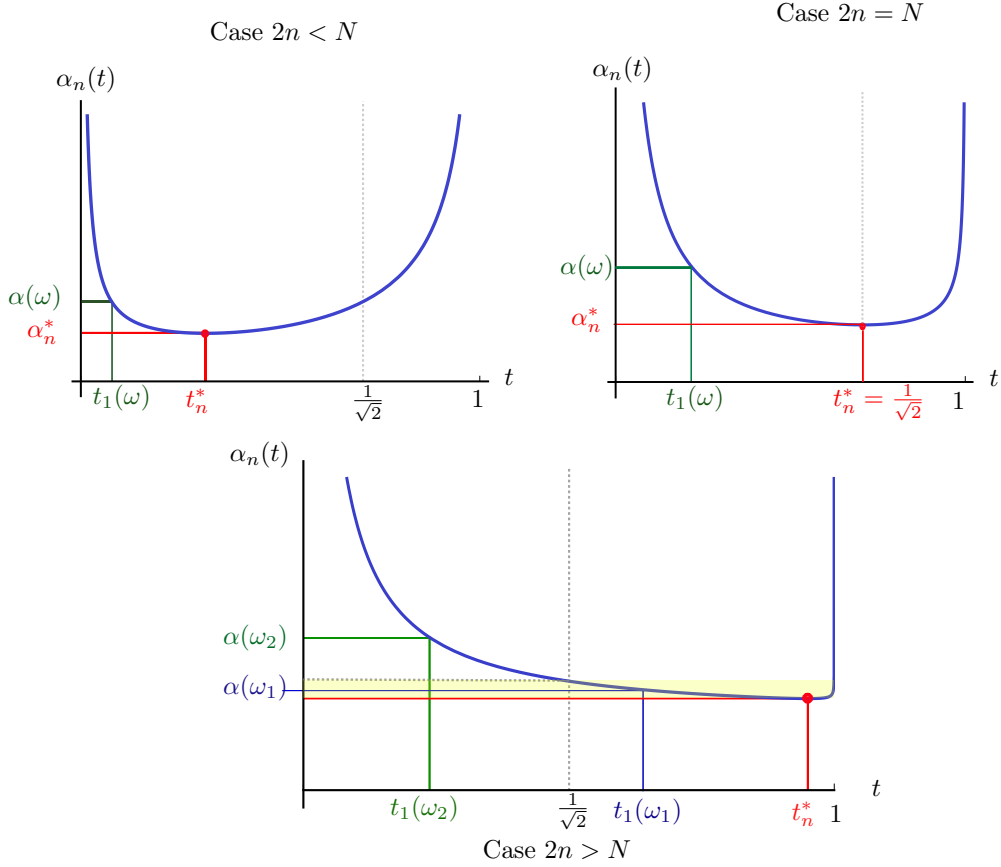}
    \caption{Three scenarios for $\alpha_n$: When $2n<N$ the unique minimizer for $\alpha_n$  is less than $1/\sqrt{2}$ so $t_1(\omega)$ is below $1/\sqrt{2}$ for all admissible $\omega> \omega_n^*$. When $2n=N$ then the unique minimizer is attained exactly at $1/\sqrt{2}$ but still, in the decreasing branch, $t_1(\omega)$ does not exceed $1/\sqrt{2}$. When $2n>N$, the yellow strip has ceiling $\alpha(\widehat{\omega})$ and floor $\alpha_{n}^*$. For frequencies $\omega_n^*\le\omega_1<\widehat{\omega}$ (inside the yellow strip) the position of $t_1(\omega_1)$ is on the right  of $1/\sqrt{2}$ while for large frequencies $\omega_2>\widehat{\omega}$ (above the yellow strip) $t_1(\omega_2)$ remains to the left of $1/\sqrt{2}$.}
    \label{fig:p=3}
\end{figure}
explicitly given by
\begin{equation}\label{eq:cubic-infimum}
(t_n^*)^2=\frac{n^{2/3}}{n^{2/3}+(N-n)^{2/3}}, 
\qquad 
\alpha_n(t_n^*)=\big(n^{2/3}+(N-n)^{2/3}\big)^{3/2}.
\end{equation}
The minimum value $\alpha_n(t_n^*)$ is monotone increasing in 
$n\in\{1,\dots,N-1\}$, and the position of the minimizer satisfies
\[
t_n^*\le \tfrac{1}{\sqrt{2}}\iff 2n\le N, 
\qquad 
t_n^*\ge \tfrac{1}{\sqrt{2}}\iff 2n\ge N,
\]
with equality precisely when $2n=N$ (only possible for an even amount of half-lines $N$).

\medskip 
The position of $t_n^*$ relative to $1/\sqrt{2}$ has a direct geometric 
consequence for the type of tail profile defined by $t_1(\omega)$ (see Definition \ref{tipoAB}, \eqref{eq:42a} and \eqref{AB}). By 
construction $\alpha_n$ is strictly decreasing on $(0,t_n^*)$, so for 
$\alpha(\omega)=-Z\sqrt{\omega}>\alpha_n(t_n^*)$ there exists a unique solution 
$t_1(\omega)\in(0,t_n^*)$ of \eqref{eq:alphan}. Recall that $t_1(\omega)
\in(0,1/\sqrt{2})$ defines a type-$A$ profile, while 
$t_1(\omega)\in(1/\sqrt{2},1)$ defines a type-$B$ profile.

When $2n\le N$ one has $t_n^*\le 1/\sqrt{2}$ and the whole decreasing branch 
on $(0,t_n^*)$ lies in the type-$A$ region. Therefore the configuration always produces 
exactly $n$ type-$A$ profiles together with $N-n$ type-$B$ profiles, for any 
$\omega$ above the {\it existence threshold }
$\omega_n^*$ defined by
\begin{equation}\label{ExThr}
\omega_n^*= \frac{\alpha^2_n(t_n^*)}{Z^2}=  \frac{(n^{2/3}+(N-n)^{2/3})^{3}}{Z^2}.
\end{equation}

When $2n>N$, on the contrary, $t_n^*>1/\sqrt{2}$ and for $\alpha(\omega)>\alpha_n(t_n^*)$ (see Figure \ref{fig:p=3}) the position of $t_1(\omega)$ 
inside $(0,t_n^*)$ clearly now depends on the size of $\omega$ and  a natural new {\it threshold value emerges}. Indeed, evaluating $\alpha_n$ 
at $1/\sqrt{2}$ yields $\alpha_n(1/\sqrt{2})=N\sqrt{2}$, which corresponds to $\alpha(\widehat{\omega})$ for the frequency \[
\widehat{\omega}=\frac{2N^2}{Z^2}.
\]
By the strict monotonicity of $\alpha_n$ on $(0,t_n^*)$, for $2n>N$ one has
\[
t_1(\omega)\in(0,1/\sqrt{2}) \iff -Z\sqrt{\omega}> N\sqrt{2} \iff \omega>\widehat{\omega},
\]
whereas for the low frequencies $\omega<\widehat{\omega}$ ($-Z\sqrt{\omega}< N\sqrt{2}$) 
the solution of equation $\alpha_{n}(t)=\alpha(\omega)$ satisfies $t_1(\omega)\in(1/\sqrt{2},t_n^*)$ and the $n$ half-lines 
sharing the shift $t_1$ are then of type $B$, and only the $N-n$ complementary 
half-lines are of type $A$. The increasing branch $(t_n^*,1)$ of $\alpha_n$, when $2n>N$, 
also lies entirely above $1/\sqrt{2}$ and produces the same type of former configurations. We can  
therefore conclude the following:

\begin{proposition}\label{nonexis}
If $2n>N$ and $\omega<2N^2/Z^2$ ($-Z\sqrt{\omega}< N\sqrt{2}$), no discontinuous-at-$\nu=L$ standing waves carrying exactly $n$ type-$A$ profiles exists.  In other words, in our notation, when $2n>N$ and $\omega<\widehat{\omega}$ we have always
$$
a_1=a_B(\omega), \;\; a_N=a_A(\omega) 
$$
and thus the profile $\Theta_{\omega,n}=( \underbrace{\Psi_{a_B},...,\Psi_{a_B}}_{n\; \text{times}}, \underbrace{\Psi_{a_A},...,\Psi_{a_A}}_{N-n\; \text{times}})$ has the $n$ first components always of type-$B$.
\end{proposition}

It may result useful to visualize Proposition~\ref{nonexis} in the particular case $N=4$. For $n=1$ and $n=2$ one has $2n\le 4$ and we have for every $\omega> \omega_n ^*$ exactly two standing waves profiles $\Theta_{\omega,1}=(\Psi_{a_A}, \Psi_{a_B}, \Psi_{a_B}, \Psi_{a_B})$, $\Theta_{\omega,2}=(\Psi_{a_A}, \Psi_{a_A}, \Psi_{a_B}, \Psi_{a_B})$ (see the two-top figures in Figure \ref{fig:placeholder}). Note that for $n=2$, $\omega_2^*= 2N^2/Z^2$. On the other hand, for $n=3$ one has $2n>4$, $t_3^*>1/\sqrt{2}$ and so we have the following options for $\omega$ and profiles: for $\omega > 2N^2/Z^2$, emerges another configuration $\Theta_{\omega,3}=(\Psi_{a_A}, \Psi_{a_A}, \Psi_{a_A}, \Psi_{a_B})$ (see the bottom-left graph in Figure \ref{fig:placeholder}) while for $\omega<2N^2/Z^2$ the resulting configuration $\widetilde{\Theta}_{\omega,3}=(\Psi_{a_B},\Psi_{a_B},\Psi_{a_B},\Psi_{a_A})$ is, up to permutation of the edges, the same configuration of the case $n=1$ (compare the bottom-right and top-left graphs in Figure~\ref{fig:placeholder}). We emphasize $\widetilde{\Theta}_{\omega,3}\neq \Theta_{\omega,1}$ since the definition of $t_1(\omega)$ is $n$-dependent. 

\begin{remark}\label{rem:w=2N^2/Z^2}
In the cluster configuration $2n>N$ and at the particular frequency $\omega=2N^2/Z^2$ the decreasing branch solution $t_1(2N^2/Z^2)$ takes the value $t_1=1/\sqrt{2}=t_N$ and therefore defines a continuous profile $\Theta_{2N^2/Z^2,\,n}$. However, the latter does not imply that only continuous configurations in $\dom$ exists when $\omega=2N^2/Z^2$ and $2n>N$. If one choose instead $t_1(2N^2/Z^2)$ to be the solution of \eqref{eq:alphan} in the increasing branch $(t_n^*,1)$ of $\alpha_{n}$ then $t_1>1/\sqrt{2}$ (see Figure~\ref{fig:p=3}) and the resulting (increasing branch) profile $\Theta_{2N^2/Z^2,\,n}$ would be genuinely discontinuous. 

It is worth mentioning that when $2n\le N$ the decreasing branch solution $t_1(2N^2/Z^2)$ is always less than $1/\sqrt{2}$ and therefore no pathological frequencies generate continuous stationary solutions. 
\end{remark}

\subsection{The existence of tail-type profiles for the cubic NLS via  a phase-plane analysis}
\label{rem:phase}
One can foresee the existence of tail-type profiles solutions of the stationary cubic NLS and in the set $\dom$ by looking at the phase-plane. In Figure~\ref{fig:phase-plane-weighted} the purple curve represents the right-hand piece of the homoclinic orbit. The derivative level $u'=d$ (dashed orange) represents the derivative matching condition $\Phi_{a_j}'(L)=d$. The points $P_1, P_2$ determined by the crossing at the derivative level $d$ determine the shifts $a_A, a_B$ and therefore their $u$-coordinates agrees with $u_1=\Psi_{a_A}(L)$ and $u_2=\Psi_{a_B}(L)$ respectively. The yellow trajectory represents the weighted middle point $\tfrac{n}{N}u_1+\tfrac{(N-n)}{N}u_2$ along the values of $d\in(0,-\omega/\sqrt{2})$. The red line agrees with the parametric line $(\tfrac{Z}{N}d,d)$. The sum condition in $\dom$ then represents the intersection $P_3$ at the common derivative level $d$ since its $u$-coordinate satisfies \[
\frac{n}{N}u_1+\frac{(N-n)}{N}u_2=\frac{n}{N}\Psi_{a_A}(L)+\frac{(N-n)}{N}\Psi_{a_B}(L)=\frac{Z}{N}d.
\]
    \begin{figure}[t]
  \centering
  \resizebox{0.4\linewidth}{!}{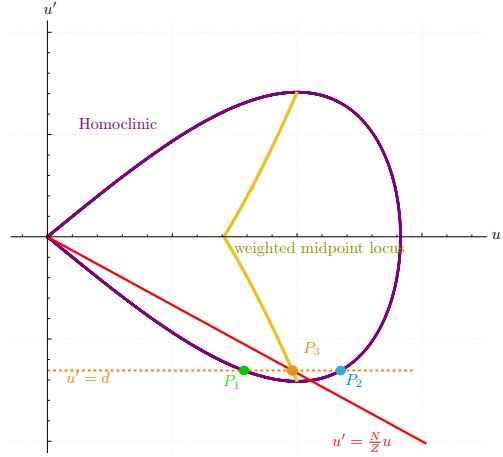}
  \caption{Phase plane analysis visualization of Remark~\ref{rem:phase}}
  \label{fig:phase-plane-weighted}
\end{figure}

\subsection{Existence of solutions for  \eqref{eq:alphan}: the formal analysis}
For fixed $n\in\{1,\dots,N-1\}$, equation \eqref{eq:alphan} admits a solution 
$t\in(0,1)$ if and only if $\alpha(\omega)\ge \alpha_n(t_n^*)$. Defining the threshold
\begin{equation}\label{eq:omega-n-star}
\omega_n^*:=\frac{\alpha_n(t_n^*)^2}{Z^2}
=\frac{\big(n^{2/3}+(N-n)^{2/3}\big)^{3}}{Z^2},
\end{equation}
existence is then guaranteed for $\omega\ge \omega_n^*$.

The map $n\mapsto \omega_n^*$ is symmetric with respect to $2n=N$, being strictly increasing for $2n<N$ and decreasing for $2n>N$, so it attains its minimum at both $n=1$ and $n=N-1$, defining the threshold
\begin{equation}\label{eq:omega-1-star}
\omega_1^*=\frac{\big(1+(N-1)^{2/3}\big)^{3}}{Z^2}
=\min_{1\le n\le N-1}\omega_n^*.
\end{equation}
No positive discontinuous configuration exists in $\dom$ whenever 
$\omega<\omega_1^*$ independent of $n$. Moreover, from the tail-type inversion analysis from 
the previous subsection, we conclude that when $2n>N$, no discontinuous standing wave carrying exactly $n$ 
type-$A$ profiles exists when $\omega\le 2N^2/Z^2$, whereas when $2n\le N$ the 
threshold is exactly $\omega_n^*$.

\medskip
We now record the existence of solutions to \eqref{eq:alphan} with a smooth 
dependence of $t_1=t_1(\omega)$ on the frequency parameter $\omega$.

\begin{proposition}\label{prop:existencet}
Let $Z<0$, $N\ge 2$, and fix $n\in\{1,\dots,N-1\}$. For every 
$\omega>\omega_n^*$ there exists a unique $t_1(\omega)\in(0,t_n^*)$ satisfying
\begin{equation}\label{eq:leftbranch}
\alpha_n(t_1(\omega))=-Z\sqrt{\omega}.
\end{equation}
Moreover, the map $\omega\mapsto t_1(\omega)$ is smooth on $(\omega_n^*,\infty)$ 
and satisfies $t_1'(\omega)<0$ with $\alpha_n'(t_1(\omega))<0$.
\end{proposition}

\begin{proof}
Existence and uniqueness of $t_1(\omega)\in(0,t_n^*)$ with $\alpha_n'(t_1(\omega))<0$ follow from the strict 
monotonicity of $\alpha_n$ on $(0,t_n^*)$ together with the fact that 
$\alpha(\omega)>\alpha_n(t_n^*)$ whenever $\omega>\omega_n^*$. 

For smoothness, define $F:(0,t_n^*)\times(\omega_n^*,\infty)\to\R$ by
\[
F(t,\omega):=\alpha_n(t)+Z\sqrt{\omega}.
\]
Then $F(t_1(\omega),\omega)=0$ and 
$\partial_t F(t_1(\omega),\omega)=\alpha_n'(t_1(\omega))\neq 0$, so the 
implicit function theorem applies. Differentiating $F(t_1(\omega),\omega)=0$ 
in $\omega$ gives
\begin{equation}
    \label{eq:t1linha} t_1'(\omega)=-\frac{Z}{2\sqrt{\omega}\,\alpha_n'(t_1(\omega))}<0,
\end{equation}
since $Z<0$ and $\alpha_n'(t_1(\omega))<0$ on $(0,t_n^*)$.
\end{proof}

\medskip
As a direct consequence of Proposition~\ref{prop:existencet}, we obtain 
Theorem~\ref{thm:existence}. The choice of the shifts of the two tail types are recovered 
from $t_1(\omega)$ and are
governed by $\omega$ and by the relative size of $n$ and $N$. Note
\begin{equation}\label{t1-0}
t_1(\omega)\in(0,1/\sqrt{2}) 
\iff 
2n\le N, \ \text{or}\ \ 2n>N \ \text{and}\ \omega>\frac{2N^2}{Z^2},
\end{equation}
\begin{equation}\label{t1-1}
t_1(\omega)\in(1/\sqrt{2},1) 
\iff 
2n>N \ \text{and}\ \ \omega_n^*<\omega<\frac{2N^2}{Z^2}.
\end{equation}
Accordingly, the type-$A$ and type-$B$ shifts are defined by
\[
a_A(\omega):=
\begin{cases}
\operatorname{arctanh}\!\big(t_1(\omega)\big) & \text{if } 2n\le N,\text{ or } 2n>N\text{ and }\omega>2N^2/Z^2,\\[4pt]
\operatorname{arctanh}\!\big(t_N(\omega)\big) & \text{if } 2n>N\text{ and }\omega_n^*<\omega<2N^2/Z^2,
\end{cases}
\]
\[
a_B(\omega):=
\begin{cases}
\operatorname{arctanh}\!\big(t_N(\omega)\big) & \text{if } 2n\le N,\text{ or } 2n>N\text{ and }\omega>2N^2/Z^2,\\[4pt]
\operatorname{arctanh}\!\big(t_1(\omega)\big) & \text{if } 2n>N\text{ and }\omega_n^*<\omega<2N^2/Z^2.
\end{cases}
\]
The resulting standing wave has $n$ tails with shift $a_A$ and $N-n$ tails with shift $a_B$ when $t_1(\omega)\in(0,1/\sqrt{2})$,while when $t_1(\omega)\in(1/\sqrt{2},1)$ the roles are exchanged and the configuration 
has $n$ tails with shift $a_B$ and $N-n$ tails with shift $a_A$.

\section{Spectral Framework: Morse and Nullity Indices}\label{sec:spectral}

In this section we develop in detail the spectral analysis outlined in the introduction, using the decomposition \eqref{eq:l2orth} to localize the instability directions on each of the subspaces $\xa$, $\xb$ and $\xm$.

\medskip
Fix $\omega>\omega^*_n$ and consider the standing wave solution
\[
\mathbf{U}(x,t)=e^{i\omega t}\Theta_{\omega,n}(x),
\qquad 
\Theta_{\omega,n}\in D(\mathcal{H}_Z^{\delta'}),
\]
of \eqref{NLS} (see \eqref{eq:thetan}) provided by Theorem~\ref{thm:existence}. Introducing the action functional
\begin{equation}\label{S1}
\mathbf{S}(\mathbf{U})=E_Z(\mathbf{U})+\omega Q(\mathbf{U}),
\qquad \mathbf{U}\in H^1(\dg),
\end{equation}
we note that $\Theta_{\omega,n}$ is a critical point of $\mathbf{S}$, namely,
$\mathbf{S}'(\Theta_{\omega,n})=\mathbf{0}$.

Next, write $\mathbf{U}=\mathbf{U}_1+i\mathbf{U}_2$ and
$\mathbf{W}=\mathbf{W}_1+i\mathbf{W}_2$, where each $\mathbf{U}_j$ and
$\mathbf{W}_j$ ($j=1,2$) has real-valued components. A standard computation shows
that the Hessian of $\mathbf{S}$ at $\Theta_{\omega,n}$ can be expressed as
\begin{equation}\label{S2}
\mathbf{S}''(\Theta_{\omega,n})(\mathbf{U},\mathbf{W})
=
\langle \mathcal{L}_{+,Z}\mathbf{U}_1,\mathbf{W}_1\rangle
+
\langle \mathcal{L}_{-,Z}\mathbf{U}_2,\mathbf{W}_2\rangle,
\end{equation}
where $\mathcal{L}_{\pm,Z}$ are diagonal $N\times N$ operators given by
\begin{equation}\label{L+}
\begin{aligned}
\mathcal{L}_{+,Z}
&=\mathrm{diag}\Big(\underbrace{\mathcal{L}_{+,1},\dots,\mathcal{L}_{+,1}}_{n\ \text{times}}, \underbrace{ \mathcal{L}_{+,N}, \dots, \mathcal{L}_{+,N}}_{N-n\ \text{times}}\Big),\\[0.5em]
\mathcal{L}_{-,Z}
&=\mathrm{diag}\Big(\underbrace{\mathcal{L}_{-,1},\dots,\mathcal{L}_{-,1}}_{n\ \text{times}}, \underbrace{ \mathcal{L}_{-,N}, \dots, \mathcal{L}_{-,N}}_{N-n\ \text{times}}\Big)
\end{aligned}
\end{equation}
with 
\begin{equation}
\mathcal{L}_{+,j}:=-\partial_x^2+\omega-3\,\Psi_{a_j}^{\,2}, 
\qquad 
\mathcal{L}_{-,j}:=-\partial_x^2+\omega-\Psi_{a_j}^{\,2}, 
\qquad j\in\{1,N\}.
\end{equation}

The operators $\mathcal{L}_{\pm,Z}$ are self-adjoint on the common domain
$D(\mathcal{L}_{\pm,Z})\equiv D(\mathcal{H}_Z^{\delta'})$. Furthermore, since each component of
$\Theta_{\omega,n}\in D(\mathcal{H}_{Z}^{\delta'})$ solves the stationary system
\begin{equation}\label{eq:stationary-p}
-\Psi_{a_j}''+\omega\Psi_{a_j}-\Psi_{a_j}^{\,3}=0, \qquad j\in\{1,N\},
\end{equation}
it follows that $\mathcal{L}_{-,Z}\Theta_{\omega,n}=\mathbf{0}$, and consequently 
$\ker(\mathcal{L}_{-,Z})$ is nontrivial.

\medskip
It is well known from the GSS framework that a fruitful stability theory for the
profiles $\Theta_{\omega,n}$ requires a detailed study of the self-adjoint operators
$\mathcal{L}_{\pm,Z}$ on $D(\mathcal{H}_Z^{\delta'})$, and this represents the core of our
analysis. Before specializing to either operator, we record the basic vertex 
boundary values that will be used throughout this section.

\begin{lemma}\label{lem:vertex-vals}
Let $Z<0$, $N\ge2$ and $n\in\{1,\dots,N-1\}$ be fixed. Assume $\omega>\omega_n^*$, 
and additionally $\omega\neq 2N^2/Z^2$ when $2n>N$. Then
\begin{enumerate}
\item[\rm (i)] The boundary values of the half-line profiles satisfy
\begin{equation}\label{eq:Theta-boundary-p}
\Psi_{a_j}'(L)=-\sqrt{\omega}\,\Psi_{a_j}(L)\,t_j,
\qquad
\Psi_{a_j}''(L)=\omega\,\Psi_{a_j}(L)\bigl(2t_j^2-1\bigr),
\qquad j\in\{1,N\}.
\end{equation}
Moreover, 
$\{t_1,t_N\}=\{t_A(\omega),t_B(\omega)\}$ with $0<t_A<1/\sqrt{2}<t_B<1$, so 
exactly one of $\Psi_{a_1}''(L)$, $\Psi_{a_N}''(L)$ is strictly negative and 
the other strictly positive.

\item[\rm (ii)] Define
\begin{equation}\label{eq:Gdef}
\mathfrak H(t):=n\,(1-t^2)^{3/2}-(N-n)\,t^3.
\end{equation}
Then $\mathfrak H$ is connected to the function $\alpha_n$ from 
Section~\ref{sec:exists} through the algebraic identity
\begin{equation}\label{eq:Galpha}
\alpha_n'(t)
=-\frac{\mathfrak H(t)}{t^2\,(1-t^2)^{3/2}},
\qquad t\in(0,1).
\end{equation}
\end{enumerate}
\end{lemma}

\begin{proof}
For (i), the expression of $\Psi_{a_j}'(L)$ follows directly from 
differentiating \eqref{soliton}. The formula for $\Psi_{a_j}''(L)$ comes from 
the stationary equation \eqref{eq:stationary-p} together with the soliton 
identity $\Psi_{a_j}^{\,2}(L)=2\omega(1-t_j^2)$. Indeed
\[
\Psi_{a_j}''(L)=\omega\,\Psi_{a_j}(L)-\Psi_{a_j}^{\,3}(L)
=\omega\,\Psi_{a_j}(L)\bigl[1-2(1-t_j^2)\bigr]
=\omega\,\Psi_{a_j}(L)\bigl(2t_j^2-1\bigr).
\]
Hence $\Psi_{a_j}''(L)=0\iff t_j=1/\sqrt{2}$. Under the existence regime, the 
exclusion $\omega\neq 2N^2/Z^2$ when $2n>N$ (the unique frequency at which the 
two shifts collapse to $1/\sqrt{2}$, cf.\ Remark~\ref{rem:w=2N^2/Z^2}) guarantees 
$\{t_1,t_N\}=\{t_A,t_B\}$ with $0<t_A<1/\sqrt{2}<t_B<1$, and therefore one of 
$\Psi_{a_1}''(L)$, $\Psi_{a_N}''(L)$ is negative and the other positive.

For (ii), a direct computation gives
\[
\alpha_n'(t)=-\frac{n}{t^2}+\frac{(N-n)\,t}{(1-t^2)^{3/2}}
=\frac{-n(1-t^2)^{3/2}+(N-n)\,t^3}{t^2\,(1-t^2)^{3/2}}
=-\frac{\mathfrak H(t)}{t^2\,(1-t^2)^{3/2}}.
\]
\end{proof}

We now move to the study of the operator $\mathcal{L}_{-,Z}$.

\subsection{Spectral analysis of $\mathcal{L}_{-,Z}$}

The operator
$\mathcal{L}_{-,Z}$ acts diagonally on $N$ components, the first $n$ of which 
carry the shift $a_1$ and the remaining $N-n$ the shift $a_N$. Let us show the operator is non-negative.

\begin{proposition}\label{prop:l2} 
Let $Z<0$, $N\ge 2$ and $n\in\{1,\dots,N-1\}$ be fixed. Assume 
$\omega>\omega_n^*$. The 
operator 
\begin{equation}
\label{L2} 
\mathcal{L}_{-,Z}:=\operatorname{diag}\bigl(\mathcal{L}_{-,1},\dots,\underset{n\text{th}}{\mathcal{L}_{-,1}},\underset{(n+1)\text{th}}{\mathcal{L}_{-,N}},\dots,\mathcal{L}_{-,N}\bigr),
\quad D(\mathcal{L}_{-,Z})=D(\mathcal{H}_Z^{\delta'}),
\end{equation}
where 
\begin{equation}
\label{LALB}
\mathcal{L}_{-,1}:=-\partial_x^2+\omega-\Psi_{a_1}^{\,2}, 
\qquad 
\mathcal{L}_{-,N}:=-\partial_x^2+\omega-\Psi_{a_N}^{\,2},
\end{equation}
satisfies
\begin{enumerate}
\item If $\dom\cap X_{j,\perp}\neq \{\mathbf{0}\}$ then 
$\mathcal{L}_{-,Z}\geqq 0$ on $\dom\cap X_{j,\perp}$, where $j\in\{1,N\}$.
\item $\mathcal{L}_{-,Z}\ge 0$ on $\dom \cap X_{\mathrm{mean}}$.
\item $\ker(\mathcal{L}_{-,Z})=\operatorname{span}\{\Theta_{\omega,n}\}\subset 
\dom\cap X_{\mathrm{mean}}$.
\end{enumerate}
\end{proposition}

\begin{proof}
We emphasize that in what follows the 
assignment of $a_1$ and $a_N$ to either type-$A$ or type-$B$ may vary with 
$\omega$ (see the discussion after Theorem~\ref{thm:existence}), but the proof 
relies only on the vertex relations $\Psi_{a_j}(L)>0$, $\Psi_{a_j}'(L)< 0$ and 
on the common-derivative condition $\Psi_{a_1}'(L)=\Psi_{a_N}'(L)$, all of 
which hold uniformly in $\omega$ in the admissible range.

\smallskip
First note that the projections \eqref{eq:projs} are given as linear 
combinations of the components and therefore commute with derivatives and 
multiplications. Thus, $\ldos$ commutes with the projections \eqref{eq:projs} 
and the decomposition is preserved by $\ldos$ in the sense that $\ldos(S)\subset S$ 
for $S\in\{\dom\cap X_{1,\perp},\,\dom\cap X_{N,\perp},\,\dom\cap X_{\mathrm{mean}}\}$, 
with the possibility of at most one of the decoupling spaces being trivial.

\smallskip
Let $j\in\{1,N\}$. Since each profile $\Psi_{a_j}$ satisfies 
\eqref{eq:stationary-p}, dividing through by $\Psi_{a_j}>0$ gives 
$\Psi_{a_j}''/\Psi_{a_j}=\omega-\Psi_{a_j}^{\,2}$, and a direct computation 
yields the factorization identity
\begin{equation}\label{eq:factorLp}
-v''+\omega v-\Psi_{a_j}^{\,2} v
=-\frac{1}{\Psi_{a_j}}\frac{d}{dx}\!\left[\Psi_{a_j}^{\,2}\,\frac{d}{dx}\!\left(\frac{v}{\Psi_{a_j}}\right)\right], 
\qquad x>L.
\end{equation}
Without loss of generality assume each component of $\mathbf{v}$ is real-valued. 
Integrating by parts on $[L,\infty)$, using that the boundary terms at $+\infty$ 
vanish by the exponential decay of $\Psi_{a_j}$ and the $H^2$-decay of $v$, we 
obtain for each $j\in\{1,\dots,N\}$
\begin{equation}\label{eq:eachj}
\int_L^\infty v_j \,\mathcal{L}_{-,k} v_j \,dx 
= \int_L^\infty \Psi_{a_k}^{\,2} \left[\frac{d}{dx}\!\left(\frac{v_j}{\Psi_{a_k}}\right)\right]^2 dx 
+ v_j'(L)v_j(L)-v_j^2(L)\,\frac{\Psi_{a_k}'(L)}{\Psi_{a_k}(L)},
\end{equation}
where $k=1$ for $1\le j\le n$ and $k=N$ for $n+1\le j\le N$.

Summing \eqref{eq:eachj} over all $j$ we obtain
\begin{equation}
\label{eq:positive}
\begin{split}
\langle \ldos \mathbf{v}, \mathbf{v} \rangle 
&= \sum_{j=1}^n\int_L^\infty \Psi_{a_1}^{\,2} \left[\frac{d}{dx}\!\left( \frac{v_j}{\Psi_{a_1}}\right)\right]^2dx
+\sum_{j=n+1}^N\int_L^\infty \Psi_{a_N}^{\,2} \left[\frac{d}{dx}\!\left( \frac{v_j}{\Psi_{a_N}}\right)\right]^2dx\\
&\qquad +\sum_{j=1}^N v_j(L)v_j'(L)
-\sum_{j=1}^n v_j^2(L)\frac{\Psi_{a_1}'(L)}{\Psi_{a_1}(L)}
-\sum_{j=n+1}^N v_j^2(L)\frac{\Psi_{a_N}'(L)}{\Psi_{a_N}(L)}.
\end{split}
\end{equation}

\medskip
Let $\mathbf{v}\in \dom\cap X_{1,\perp}$. By definition of $X_{1,\perp}$, the 
components $v_{n+1},\dots,v_N$ vanish, and the boundary condition 
$v_1'(L)=\cdots=v_n'(L)=0$ holds. Equation \eqref{eq:positive} reduces to
\[
\langle \ldos \mathbf{v}, \mathbf{v} \rangle 
= \sum_{j=1}^n\int_L^\infty \Psi_{a_1}^{\,2} \left[\frac{d}{dx}\!\left( \frac{v_j}{\Psi_{a_1}}\right)\right]^2dx
-\sum_{j=1}^n v_j^2(L)\frac{\Psi_{a_1}'(L)}{\Psi_{a_1}(L)}\ge 0,
\]
since $\Psi_{a_1}'(L)<0$ and $\Psi_{a_1}(L)>0$ by Lemma~\ref{lem:vertex-vals}\,(i). 

Now suppose $\mathbf{v}\in \dom\cap X_{1,\perp}$ satisfies $\ldos\mathbf{v}=0$. 
This is equivalent to $\mathcal{L}_{-,1} v_j=0$ for all $j\in\{1,\dots,n\}$, and 
therefore each $v_j$ is of the form $v_j=c_j\Psi_{a_1}$ by Sturm--Liouville 
theory (see for instance \cite{BerShu91}). The 
boundary condition $v_j'(L)=0$ together with $\Psi_{a_1}'(L)\neq 0$ then forces 
$c_j=0$ for all $j$. We conclude $\mathbf{v}\equiv 0$, proving that $\ldos$ is 
strictly positive on $\dom\cap X_{1,\perp}\neq\{\mathbf{0}\}$. The analogous 
argument applies to $\dom\cap X_{N,\perp}$ (assuming it is non-trivial), 
yielding item 1.

\medskip
Now consider $\mathbf{v}\in\dom\cap X_{\mathrm{mean}}$, so that 
$v_1=\dots=v_n$ and $v_{n+1}=\dots=v_N$ on $[L,\infty)$. From the definition of 
this space and using the sum condition in $\dom$, the boundary terms in 
\eqref{eq:positive} can be rewritten as
\begin{equation}\label{eq:510}
\begin{split}
&n\left\{v_1'(L) v_1(L)-v_1^2(L)\frac{\Psi_{a_1}'(L)}{\Psi_{a_1}(L)} \right\}
+(N-n)\left\{v_N'(L) v_N(L)-v_N^2(L)\frac{\Psi_{a_N}'(L)}{\Psi_{a_N}(L)} \right\}\\
&=\frac{1}{Z}\!\left(\sum_{j=1}^N v_j(L)\right)^{\!2}
-\underbrace{\frac{\Psi_{a_1}'(L)}{n\,\Psi_{a_1}(L)}\!\left(\sum_{j=1}^n v_j(L) \right)^{\!2}
-\frac{\Psi_{a_N}'(L)}{(N-n)\,\Psi_{a_N}(L)}\!\left(\sum_{j=n+1}^N v_j(L) \right)^{\!2}}_{\RNum{1}}.
\end{split}
\end{equation}
Using the boundary conditions on $\dom$, the definition of $X_{\mathrm{mean}}$, 
and the common-derivative property from Lemma~\ref{lem:vertex-vals}\,(ii), 
$\RNum{1}$ in \eqref{eq:510} satisfies
\begin{equation}
\begin{split}
Z\,\RNum{1}&= \left(1+\frac{N-n}{n}\frac{\Psi_{a_N}(L)}{\Psi_{a_1}(L)}\right)\!\left(\sum_{j=1}^n v_j(L)\right)^{\!2}
+\left(1+\frac{n}{N-n}\frac{\Psi_{a_1}(L)}{\Psi_{a_N}(L)}\right)\!\left(\sum_{j=n+1}^N v_j(L)\right)^{\!2}\\
&=\bigl(n\,v_1(L)\bigr)^2+\bigl((N-n)\,v_N(L)\bigr)^2
+n(N-n)\!\left(v_1^2(L)\frac{\Psi_{a_N}(L)}{\Psi_{a_1}(L)}+v_N^2(L)\frac{\Psi_{a_1}(L)}{\Psi_{a_N}(L)}\right).
\end{split}
\end{equation}
Substituting back into \eqref{eq:510} we obtain
\[
\eqref{eq:510}=-\frac{n(N-n)}{Z}\!\left(v_1(L)\sqrt{\frac{\Psi_{a_N}(L)}{\Psi_{a_1}(L)}}\, +\, v_N(L)\sqrt{\frac{\Psi_{a_1}(L)}{\Psi_{a_N}(L)}}\,\right)^{\!2}\ge 0,
\]
since $Z<0$. This proves item 2.

\medskip
Finally, to prove item 3, we note items 1 and 2 imply that 
$\ker(\ldos)\subset\dom\cap X_{\mathrm{mean}}$. Assume that 
$\mathbf{v}=(v_1,\dots,\underset{n\text{th}}{v_1},\underset{(n+1)\text{th}}{v_N},\dots,v_N)\in\ker(\ldos)$. 
Then $\mathcal{L}_{-,1}v_1=0$ and $\mathcal{L}_{-,N}v_N=0$. Since 
$v_1,v_N\in H^2([L,\infty))$, Sturm--Liouville theory (see for instance 
\cite{BerShu91}) gives $v_1=c_1\Psi_{a_1}$ and $v_N=c_N\Psi_{a_N}$. From the 
derivative condition in $\dom$ we conclude 
$c_1\Psi_{a_1}'(L)=c_N\Psi_{a_N}'(L)$, and since 
$\Psi_{a_1}'(L)=\Psi_{a_N}'(L)\neq 0$ this forces $c_1=c_N$, proving 
$\mathbf{v}\in\operatorname{span}\{\Theta_{\omega,n}\}$.
\end{proof}

\subsection{Spectral analysis of $\luno$}

We now turn to $\mathcal L_{+,Z}$, beginning with its Morse index.

\begin{theorem}[Morse index of $\mathcal L_{+,Z}$]
\label{thm:morse1}
Let $Z<0$, $N\ge 2$ and $n\in\{1,\dots,N-1\}$ be fixed. Assume $\omega>\omega_n^*$. 
\begin{enumerate}
\item If $2n\le N$, or $2n>N$ and $\omega>2N^2/Z^2$, then
\[
n(\mathcal L_{+,Z})=n,
\qquad
n\bigl(\mathcal L_{+,Z}\,\bigr|_{\dom\cap X_{\mathrm{mean}}}\bigr)=1.
\]
\item If $2n>N$ and $\omega_n^*<\omega<2N^2/Z^2$, then
\[
n(\mathcal L_{+,Z})=N-n+1,
\qquad
n\bigl(\mathcal L_{+,Z}\,\bigr|_{\dom\cap X_{\mathrm{mean}}}\bigr)=2.
\]
\item If $2n>N$ and $\omega=2N^2/Z^2$, then \[
n(\mathcal L_{+,Z})=n\bigl(\mathcal L_{+,Z}\,\bigr|_{\dom\cap X_{\mathrm{mean}}}\bigr)=1.
\]
\end{enumerate}
\end{theorem}

By convenience of the reader, we record a fast-reference summary of the Morse index of $\mathcal{L}_{+,Z}$ in Table~\ref{table:2} of Appendix D.

\begin{proof}
For short we write $t_1:=t_1(\omega)\in (0,1)$ and $t_N:=t_N(\omega)=\sqrt{1-t_1^2}$ 
throughout the proof. Recall from  
\eqref{t1-0} and \eqref{t1-1} that the three regimes in the statement of the theorem correspond 
exactly to $t_1\in(0,1/\sqrt{2})$ ($t_N> 1/\sqrt{2}$), $t_1\in(1/\sqrt{2},1)$ ($t_N< 1/\sqrt{2}$) and $t_1=1/\sqrt{2}=t_N$ respectively.

\medskip
\noindent\textbf{Case 1: $t_1\in(0,1/\sqrt{2})$ ($t_N>1/\sqrt{2}$).}

\smallskip
As in the proof of Proposition~\ref{prop:l2}, since $\luno$ commutes with the
projections \eqref{eq:projs} the decomposition \eqref{eq:decompdom} is invariant
under $\luno$, and we may compute the Morse index block by block.

\smallskip
For $j\in\{1,N\}$, any $\mathbf v\in \dom\cap X_{j,\perp}$ the derivative
matching condition in $D(\mathcal H_Z^{\delta'})$ forces the common derivative
at the vertex to be zero, hence $v_k'(L)=0$ for all $k$. On each active
edge the components of the diagonal operator $\luno$ reduces to the scalar Neumann problem
\begin{equation}
\label{eq:neumann}
\mathfrak L_{+,j}:=-\partial_x^2+\omega-3\,\Psi_{a_j}^{\,2},
\qquad
D(\mathfrak L_{+,j})=\{u\in H^2([L,\infty)):\ u'(L)=0\},
\end{equation}
where we write $\mathfrak L_{+,j}$ instead of simply $\mathcal L_{+,j}$ again to emphasize the Neumann-type domain of $\mathfrak L_{+,j}$. 

By Lemma~\ref{lem:LA-negative-generalp} with $p=3$, the threshold is 
$a_3^{\sharp}=\operatorname{arctanh}(1/\sqrt{2})\equiv a_*$, and the operator
$\mathfrak L_{+,j}$ has a simple negative Neumann eigenvalue if and only if
the shift satisfies $a_j< a_*$ (equivalently $t_j<1/\sqrt{2}$). 
Otherwise $\mathfrak L_{+,j}$ has no negative eigenvalue. In case of existence 
we denote the Neumann eigenvalue by $\lambda_j<0$ and the eigenfunction by 
$\chi_j$.  In this form, under the case assumption $t_1<1/\sqrt{2}$, the operator $\mathfrak L_{+,1}$ has a potential of type-$A$ and only 
$\lambda_1$ is guaranteed to exists leading to $n(\mathfrak L_{+,1})=1$, while $\mathfrak L_{+,N}\ge 0$ since $\mathfrak L_{+,N}$ has a potential of type-$B$.

\bigskip
 In the following, we account the Morse index contribution of the decoupling spaces $X_{j,\perp}$ for $\luno$.

\noindent On $\dom\cap X_{1,\perp}$: Note that if $\lambda\in(-\infty,0)$ is such that 
$\mathcal{L}_{+,Z}\mathbf{v}=\lambda \mathbf{v}$ with 
$\mathbf{0}\neq \mathbf{v}=(v_j)_{j=1}^N\in \dom\cap X_{1,\perp}$, then $\lambda=\lambda_1$. In fact, 
in such case one immediately gets
\[
\mathcal{L}_{+,1}v_k=\lambda v_k, \quad v_k'(L)=0, \quad k=1,\dots,n.
\]
From the former comments and since at least one of the 
$v_k$'s is not null, we conclude $\lambda=\lambda_1$.  In this form, it is not hard to see that (because $\chi_1\neq 0$)
\begin{equation}
    \label{eq:K1n-1} \ker\!\left(\mathcal{L}_{+,Z}\bigr|_{\dom\cap X_{1,\perp}}-\lambda_1 I\right)
=\left\{ \mathbf{v}\in \dom\cap X_{1,\perp} \Bigm| v_k=c_k \chi_1, \quad \sum_{k=1}^n c_k=0 \right\}
:=\mathcal{K}_{1}.
\end{equation}
Since $\mathcal{K}_{1}$ is isomorphic to the $(n-1)$-dimensional space 
$\{\mathbf{c}=(c_k)\in \R^n \mid \sum c_k=0\}$, we conclude
\[
n\bigl(\mathcal{L}_{+,Z}\bigr|_{\dom\cap X_{1,\perp}}\bigr)=n-1.
\]
(When $n=1$, the set $X_{1,\perp}$ is trivial and the contribution is $0$).

On $\dom\cap X_{N,\perp}$:  the null derivative matching reduces each active 
edge to $\mathfrak L_{+,N}$. Since $\mathfrak L_{+,N}\geqq 0$ by the remarks above, we get
\[
\langle\mathcal L_{+,Z}\mathbf v,\mathbf v\rangle
=\sum_{k=n+1}^N\langle\mathfrak L_{+,N}v_k,v_k\rangle\ge 0,
\qquad\mathbf v\in\dom\cap X_{N,\perp},
\]
and so we conclude $n\bigl(\mathcal L_{+,Z}|_{\dom\cap X_{N,\perp}}\bigr)=0$.

\bigskip
We now focus on the mean space $\xm$ contribution to the Morse index of  $\mathcal L_{+,Z}$, this will be the most delicate part of the spectral analysis. We want to show that $n\bigl(\mathcal L_{+,Z}|_{\dom\cap\xm}\bigr)=1$. 

\smallskip
On $\dom\cap \xm$,  the 
eigenvalue problem $\mathcal L_{+,Z}\mathbf v=\lambda \mathbf v$ with $\lambda<0$, 
$\mathbf{v}=(v_1,\dots,v_1,v_N,\dots,v_N)$ reduces to the coupled half-line
system
\begin{equation}\label{eq:mean-eig-ODE}
\mathcal L_{+,1}v_1=\lambda v_1,
\qquad
\mathcal L_{+,N}v_N=\lambda v_N,
\qquad x>L,
\end{equation}
subject to the $\delta'$ vertex conditions
\begin{equation}\label{eq:mean-vertex}
v_1'(L)=v_N'(L)=:d,
\qquad
n\,v_1(L)+(N-n)\,v_N(L)=Z\,d.
\end{equation}

\smallskip
Fix $j\in\{1,N\}$. Since for arbitrary $\lambda<0$ and each index $j$ the potential of the operator $\mathcal{L}_{+,j}-\lambda$ is bounded below by a positive constant (recall  $\sigma_{\mathrm{ess}}(\mathcal L_{+,j})=[\omega, +\infty)$), there exists (see for instance
\cite[Theorem~3.3-Chapter 2]{BerShu91}), up to a multiplicative constant, a unique
decaying solution $\rho_{j,\lambda}\in H^2([L,\infty))$ of the ODE's problem
\begin{equation}\label{Peigenvalue}
(\mathcal L_{+,j}-\lambda)\rho_{j,\lambda}=0,
\qquad \rho_{j,\lambda}(x)\to0\ \text{as }x\to+\infty.
\end{equation}
Therefore, if $\lambda$ is not an eigenvalue of the Dirichlet problem
\[
\mathcal{L}_{+,j}f=\lambda f, \qquad f\in H^2([L,\infty))\cap\{f\mid f(L)=0\},
\]
then $\rho_{j,\lambda}(L)\neq 0$ and we may normalize
$u_{j,\lambda}:=\rho_{j,\lambda}/\rho_{j,\lambda}(L)$ so that
$u_{j,\lambda}(L)=1$. Inspired by some of the ideas present in \cite{NPS,G} we define
\[
m_j(\lambda):=u_{j,\lambda}'(L).
\]

\medskip
At a (discrete) Dirichlet eigenvalue $\lambda_0<0$   one has $\rho_{j,\lambda_0}(L)=0$ but
$\rho_{j,\lambda_0}'(L)\neq 0$ (otherwise $\rho_{j,\lambda_0}\equiv 0$),
hence $m_j(\lambda)$ has a pole at $\lambda_0$ while
\[
\frac{1}{m_j(\lambda)}=\frac{\rho_{j,\lambda}(L)}{\rho_{j,\lambda}'(L)}
\longrightarrow 0
\qquad \mbox{as }\quad \lambda\to\lambda_0,
\]
so Dirichlet eigenvalues yield only removable discontinuities for
$\frac{1}{m_j}$.

\medskip
Now, from the discussion associated to problem \eqref{Peigenvalue} on $[L, +\infty)$ follows that any decaying solution of \eqref{eq:mean-eig-ODE} has the form
$v_j=c_j\,\rho_{j,\lambda}$, with $c_j$ constant. Imposing the vertex
conditions \eqref{eq:mean-vertex} and eliminating the constants, we get the
scalar eigenvalue equation
\begin{equation}\label{eq:F-lambda}
F(\lambda):=\frac{n}{m_1(\lambda)}+\frac{N-n}{m_N(\lambda)}=Z,
\end{equation}
where $F$ is understood with the continuous extension
$\frac{1}{m_j(\lambda)}:=0$ at Dirichlet eigenvalues.

\medskip
Returning to \eqref{Peigenvalue}, from Appendix~\ref{ap:2} we can differentiate
$(\mathcal L_{+,j}-\lambda)u_{j,\lambda}=0$ with respect to $\lambda$.
Denoting $w_{j,\lambda}:=\partial_\lambda u_{j,\lambda}$, we obtain
\[
(\mathcal L_{+,j}-\lambda)w_{j,\lambda}=u_{j,\lambda},
\qquad
w_{j,\lambda}(L)=0.
\]
Taking the $L^2$ inner product with $u_{j,\lambda}$ and integrating by parts
over $[L,\infty)$ yields
\[
\int_L^\infty u_{j,\lambda}^2\,dx
=
\langle (\mathcal{L}_{+,j}-\lambda)w_{j,\lambda},u_{j,\lambda}\rangle
=
\big[-w_{j,\lambda}'(x)u_{j,\lambda}(x)+w_{j,\lambda}(x)u_{j,\lambda}'(x)\big]_{L}^{\infty}.
\]
Since $u_{j,\lambda},w_{j,\lambda}$ decay exponentially (see  \cite{NPS}) and
$u_{j,\lambda}(L)=1$, $w_{j,\lambda}(L)=0$ we obtain
\[
\partial_\lambda u_{j,\lambda}'(L)
=
\int_L^\infty u^2_{j,\lambda}(x)\,dx>0.
\]
Therefore
\begin{equation}\label{eq:m-derivative}
m_j'(\lambda)>0
\qquad\text{for all }\lambda<0.
\end{equation}
Hence whenever $m_1(\lambda)m_N(\lambda)\neq0$ we have
\[
F'(\lambda)
=
-\frac{n\,m_1'(\lambda)}{m^2_1(\lambda)}
-\frac{(N-n)\,m_N'(\lambda)}{m^2_N(\lambda)}
<0,
\]
proving $F$ is strictly decreasing on each connected component of its domain.
The vertical asymptotes of $F$ occur at the zeros of $m_1$ or $m_N$, which
coincide with the negative Neumann eigenvalues of $\mathfrak L_{+,1}$ and
$\mathfrak L_{+,N}$ respectively. Since $\mathfrak L_{+,N}\ge 0$ has no
negative eigenvalue, $m_N$ has no zero on $(-\infty,0)$, and $F$ has exactly
one pole on $(-\infty,0)$, namely at $\lambda_1$. This pole partitions
$(-\infty,0)$ into two decreasing branches.

\medskip
As $\lambda\to-\infty$, the potentials $-3\,\Psi_{a_j}^{\,2}$
are negligible compared to $\omega-\lambda$, giving
$m_j(\lambda)\sim -\sqrt{\omega-\lambda}$ (see, for instance, formula (5.6) in \cite{NPS}). Then $F(\lambda)\to 0^-$. At the
pole $\lambda_1$ we have $F(\lambda)\to-\infty$ from the left and
$F\to\infty$ from the right. We have the following counting for solutions of
$F(\lambda)=Z$ (see Figure~\ref{fig:F-lambda} for illustration in the particular case $N=3$, $n=2$ and $Z=-1$):
\begin{figure}
    \centering
    \resizebox{\linewidth}{!}{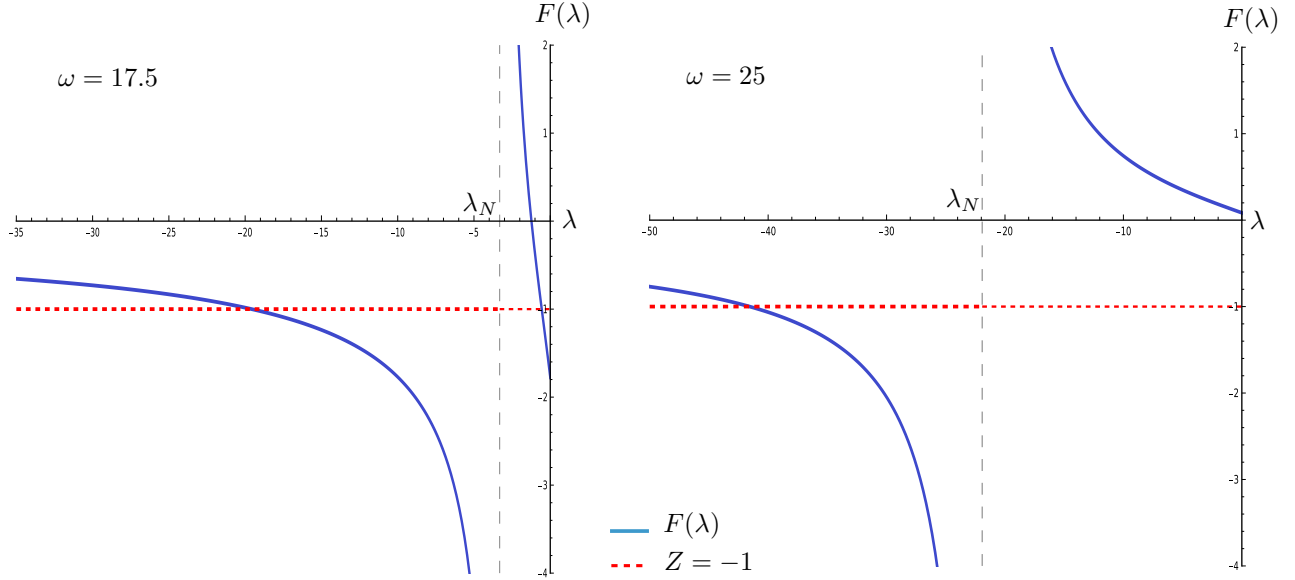}
    \caption{Two plots of $F(\lambda)$ for $N=3$, $n=2$ and $Z=-1$. The existence threshold $\omega_n^*\approx 17.32$ while the secondary threshold $2N^2/Z^2$ agrees with $18$. When $\omega>18$ (right-hand plot) The function $F$ has a unique pole at $\lambda_N$ 
    (resp.\ at $\lambda_1$ if $n=1$). The leftmost branch contributes one 
    solution to $F(\lambda)=Z$; the rightmost branch contributes none ($F(0)>Z$). When $\omega_n^*<\omega<18$, $F(0)<Z$ and therefore there is an extra negative $\lambda$ solving $F(\lambda)=Z$}
    \label{fig:F-lambda}
\end{figure}
\begin{itemize}
\item On the leftmost branch $(-\infty,\lambda_1)$, $F$ runs from $0^-$ down
to $-\infty$, crossing $Z<0$ exactly once.
\item On the rightmost branch $(\lambda_1,0)$, $F$ decreases from $+\infty$
to $F(0^-)$; an additional solution of $F(\lambda)=Z$ exists if and only if $F(0)\le Z$.
\end{itemize}

\medskip
We show $F(0)>Z$ strictly. Differentiating \eqref{eq:stationary-p} in $x$ 
shows that $\Psi_{a_j}'$ solves $\mathcal L_{+,j}\Psi_{a_j}'=0$ and decays 
exponentially at infinity, hence $u_{j,0}(x)=\Psi_{a_j}'(x)/\Psi_{a_j}'(L)$ 
and $1/m_j(0)=\Psi_{a_j}'(L)/\Psi_{a_j}''(L)$. Substituting the boundary 
values \eqref{eq:Theta-boundary-p} from Lemma~\ref{lem:vertex-vals} - (i),
\begin{equation}
    \label{eq:1-over-m} 
\frac{1}{m_j(0)}=\frac{\Psi_{a_j}'(L)}{\Psi_{a_j}''(L)}
=\frac{-\sqrt{\omega}\,\Psi_{a_j}(L)\,t_j}{\omega\,\Psi_{a_j}(L)\,(2t_j^2-1)}
=\frac{-t_j}{\sqrt{\omega}\,(2t_j^2-1)}.
\end{equation}
Writing $D_1:=1-2t_1^2>0$ and $D_N:=2t_N^2-1>0$,
\[
F(0)=\frac{n\,t_1}{\sqrt{\omega}\,D_1}-\frac{(N-n)\,t_N}{\sqrt{\omega}\,D_N}.
\]
Combining with the vertex sum condition 
$Z\sqrt{\omega}=-\bigl(n/t_1+(N-n)/t_N\bigr)$ we obtain
\[
\sqrt{\omega}\,(F(0)-Z)
=\frac{n(1-t_1^2)}{t_1\,D_1}-\frac{(N-n)(1-t_N^2)}{t_N\,D_N}.
\]
Substituting the level-set identity $t_1^2+t_N^2=1$ and noting that 
$D_1=1-2t_1^2=t_N^2-t_1^2=2t_N^2-1=D_N$, we obtain after combination over a 
common denominator
\begin{equation}\label{eq:F0minusZ-bis}
F(0)-Z
=\frac{n\,t_N^{3}-(N-n)\,t_1^{3}}{\sqrt{\omega}\,t_1\,t_N\,(t_N^2-t_1^2)}
=\frac{\mathfrak H(t_1)}{\sqrt{\omega}\,t_1\,t_N\,(t_N^2-t_1^2)},
\end{equation}
with \[
\mathfrak H(t_1)=n\,(1-t_1^2)^{3/2}-(N-n)\,t_1^{3}=n\,t_N^{3}-(N-n)\,t_1^{3}.
\] as in \eqref{eq:Gdef}. 

The denominator of \eqref{eq:F0minusZ-bis} is strictly positive in the regime 
$t_1\in(0,1/\sqrt{2})$, $t_N\in(1/\sqrt{2},1)$ (this case hypothesis). Hence the sign of $F(0)-Z$ 
coincides with the sign of $\mathfrak H(t_1)$, which is strictly positive by 
Lemma~\ref{lem:vertex-vals} - (ii) since on $(0,1/\sqrt{2})$ the denominator in 
\eqref{eq:Galpha} is strictly positive and $\alpha_n$ is strictly decreasing.

We conclude $F(0)-Z>0$ strictly, and the rightmost branch contributes no 
extra solution.

\bigskip
From all the above, summing the three contributions in Case~1:
\begin{equation}\label{Morse1}
n(\mathcal L_{+,Z})=(n-1)+0+1=n,
\qquad 
n\bigl(\mathcal L_{+,Z}|_{\dom\cap\xm}\bigr)=1.
\end{equation}
Finally, from \eqref{t1-0}, we conclude the Morse index values in \eqref{Morse1} happen for $2n\le N$, or $2n>N$ and $\omega>2N^2/Z^2$.

\bigskip
\noindent\textbf{Case 2: $t_1\in(1/\sqrt{2},1)$  ($t_N<1/\sqrt{2}$).}

\smallskip
The structure of the proof is identical to Case~1; we only highlight the 
differences arising from the swap of the position of $t_1$ and $t_N$ relative 
to $1/\sqrt{2}$. Now $t_N=\sqrt{1-t_1^2}\in(0,1/\sqrt{2})$ with $t_N<t_1$, so by 
Lemma~\ref{lem:LA-negative-generalp} the operator $\mathfrak L_{+,N}$ has a 
simple negative Neumann eigenvalue $\lambda_N<0$ with eigenfunction $\chi_N$, 
while $\mathfrak L_{+,1}\ge 0$.

\smallskip
The decoupling-space analysis carries over with the roles of $\{1\}$ and 
$\{N\}$ exchanged in the sense that $\mathcal L_{+,Z}|_{\dom\cap X_{1,\perp}}\ge 0$, hence 
$n(\mathcal L_{+,Z}|_{\dom\cap X_{1,\perp}})=0$, whereas the $\lambda_N$-block 
on $\dom\cap X_{N,\perp}$ contributes a kernel isomorphic to 
$\{\mathbf c=(c_k)\in\R^{N-n}\mid \sum c_k=0\}$, of dimension $N-n-1$ and
\[
n\bigl(\mathcal L_{+,Z}|_{\dom\cap X_{N,\perp}}\bigr)=N-n-1.
\]
(When $N-n=1$, that is $n=N-1$, the set $X_{N,\perp}$ is trivial and the 
contribution is again $0$).

\smallskip
On the mean space, the construction of $m_j(\lambda)$ and the function 
$F(\lambda)$ as in \eqref{eq:F-lambda} is unchanged; so are the monotonicity 
properties \eqref{eq:m-derivative} and the asymptotics $F(\lambda)\to 0^-$ as 
$\lambda\to-\infty$. The only structural change is that the unique pole of 
$F$ on $(-\infty,0)$ is now located at $\lambda_N$ instead of $\lambda_1$ 
(since now $m_1$, not $m_N$, has no zero on $(-\infty,0)$). The leftmost 
branch $(-\infty,\lambda_N)$ still contributes exactly one solution of 
$F(\lambda)=Z$ on the same grounds as in Case~1.

\smallskip
The sign of $F(0)-Z$ is the crucial change. Repeating the computation 
leading to \eqref{eq:1-over-m}–\eqref{eq:F0minusZ-bis} verbatim, but now with 
$D_1=2t_1^2-1>0$, $D_N=1-2t_N^2>0$, $D_1=D_N=t_1^2-t_N^2$, the same combination 
over a common denominator yields
\[
F(0)-Z
=\frac{(N-n)\,t_1^{3}-n\,t_N^{3}}{\sqrt{\omega}\,t_1\,t_N\,(t_1^2-t_N^2)}
=-\frac{\mathfrak H(t_1)}{\sqrt{\omega}\,t_1\,t_N\,(t_1^2-t_N^2)}.
\]
The denominator is strictly positive in this case regime $t_1\in(1/\sqrt{2},1)$, 
$t_N\in(0,1/\sqrt{2})$, and $\mathfrak H(t_1)>0$ still holds because $t_1$ lies 
on the decreasing branch of $\alpha_n$, where $\alpha_n'(t_1)<0$ forces 
$\mathfrak H(t_1)>0$ via \eqref{eq:Galpha} (see Figure \ref{fig:p=3}). Therefore $F(0)-Z<0$ strictly, and 
on the rightmost branch $(\lambda_N,0)$ the function $F$ decreases from 
$+\infty$ to $F(0^-)<Z$, crossing $Z$ exactly once (see once more Figure \ref{fig:F-lambda} for illustration).  This contributes one 
additional eigenvalue on the mean space.

\smallskip
Summing the three contributions in Case~2:
\begin{equation}\label{Morse2}
n(\mathcal L_{+,Z})=0+(N-n-1)+2=N-n+1,
\qquad
n\bigl(\mathcal L_{+,Z}|_{\dom\cap\xm}\bigr)=2.
\end{equation}
As before, from \eqref{t1-1}, the Morse index relations in \eqref{Morse2} happen exactly when $2n>N$ and $\omega_n^*<\omega<2N^2/Z^2$ as stated.

\medskip
\noindent\textbf{Case 3}: $t_1=1/\sqrt{2}=t_N$. 

\smallskip
At this frequency $a_1=a_N=:a_*$ with $a_*=\operatorname{arctanh}(1/\sqrt2)$, both
clusters carry the same half-line profile $\Psi_{a_*}:=\Psi_{a_1}=\Psi_{a_N}$, and
$\mathcal L_{+,Z}=\operatorname{diag}(\mathcal L_{+,*},\dots,\mathcal L_{+,*})$
with $\mathcal L_{+,*}=-\partial_x^2+\omega-3\Psi_{a_*}^2$. Since $\mathcal L_{+,*}$
still commutes with the projections \eqref{eq:projs}, the decomposition
\eqref{eq:decompdom} remains invariant and we again argue block by block. The
only --- but decisive --- difference with Case~1 is that, by
Lemma~\ref{lem:vertex-vals}-(i) evaluated at $t_*=1/\sqrt2$,
\begin{equation}\label{eq:Psi-star-pp-zero}
\Psi_*''(L)=\omega\,\Psi_*(L)\,(2t_*^2-1)=0,
\end{equation}
so the quantity $1/m_j(0)=\Psi_{a_j}'(L)/\Psi_{a_j}''(L)$ from \eqref{eq:1-over-m}
is no longer defined and the $F(0)$ analysis of Case~1 must be replaced.

\smallskip
\noindent\emph{Decoupling blocks.} Exactly as in Case~1, derivative matching
forces $v_k'(L)=0$ on every active edge, decoupling the problem into the Neumann
operator $\mathfrak L_{+,*}:=\mathcal L_{+,*}$ with $u'(L)=0$. By
Lemma~\ref{lem:LA-negative-generalp}, the negative Neumann eigenvalue present for
$t<1/\sqrt2$ now sits exactly at $0$, with eigenfunction $\Psi_{a_*}'$ (indeed
$(\Psi_*')'(L)=\Psi_*''(L)=0$ by \eqref{eq:Psi-star-pp-zero}), and
$\mathfrak L_{+,*}\ge 0$. Hence both decoupling blocks contribute $0$ to the
Morse index,
\[
n\bigl(\mathcal L_{+,Z}|_{\dom\cap X_{1,\perp}}\bigr)
=n\bigl(\mathcal L_{+,Z}|_{\dom\cap X_{N,\perp}}\bigr)=0.
\]

\smallskip
\noindent\emph{Mean block.} Since $\mathcal L_{+,1}=\mathcal L_{+,N}=\mathcal
L_{+,*}$, the two functions $m_1,m_N$ of Case~1 collapse to a single
$m(\lambda):=m_1(\lambda)=m_N(\lambda)$ (see equation (5.5) in \cite{NPS}), and the vertex conditions
\eqref{eq:mean-vertex} read $v_1(L)\,m(\lambda)=v_N(L)\,m(\lambda)=d$ together with
$n\,v_1(L)+(N-n)\,v_N(L)=Z\,d$. Where $m(\lambda)\neq0$ one has
$v_1(L)=v_N(L)=d/m(\lambda)$, and the system reduces to the scalar equation
\begin{equation}\label{eq:trans-mean-eig}
m(\lambda)=\frac{N}{Z}.
\end{equation}
The monotonicity $m'(\lambda)>0$ from \eqref{eq:m-derivative} and the asymptotics
$m(\lambda)\sim-\sqrt{\omega-\lambda}\to-\infty$ as $\lambda\to-\infty$ carry over
verbatim from Case~1, while \eqref{eq:1-over-m} is now replaced by $m(0)=\Psi_*''(L)/\Psi_*'(L)=0$. Thus $m$ is a strictly increasing
bijection from $(-\infty,0)$ to $(-\infty,0)$ with no pole (recall $\mathfrak L_{+,*}\ge 0$),
and since $N/Z<0$ equation \eqref{eq:trans-mean-eig} has exactly one root
$\lambda_{\mathrm{mean}}<0$. The equation $m(\lambda)=0$ has its only solution on
$(-\infty,0]$ at $\lambda=0$ and contributes nothing to the negative count. Hence
\[
n\bigl(\mathcal L_{+,Z}|_{\dom\cap\xm}\bigr)=1.
\]

\smallskip
Summing the three blocks,
\[
n(\mathcal L_{+,Z})=0+0+1=1=n\bigl(\mathcal L_{+,Z}|_{\dom\cap\xm}\bigr),
\]
which is the content of item~(3).

This completes the proof.
\end{proof}

We now focus on the kernel of $\mathcal{L}_{+,Z}.$

\begin{proposition}[Kernel of $\mathcal{L}_{+,Z}$]
\label{prop:ker-L1Z}
Let $Z<0$, $N\ge 2$ and $n\in\{1,\dots,N-1\}$ be fixed. Assume 
$\omega>\omega_n^*$. Then 
\begin{enumerate}
    \item If $2n\le N$ or if $2n>N$ with $\omega\neq 2N^2/Z^2$ then $\ker(\mathcal{L}_{+,Z})=\{\mathbf{0}\}$.
    \item If $2n>N$ and $\omega=2N^2/Z^2$ then $\operatorname{dim}(\ker(\mathcal{L}_{+,Z}))=N-1$. 
\end{enumerate}

\end{proposition}

\begin{proof}
We write $t_1:=t_1(\omega)$ and $t_N:=t_N(\omega)=\sqrt{1-t_1^2}$ throughout. 

Let us prove item (1). Both case assumptions $2n\le N$ or $2n>N$ with $\omega\neq 2N^2/Z^2$ guarantee that $t_1\neq 1/\sqrt{2}$ (see Figure~\ref{fig:p=3}). Therefore $t_1\neq t_N$ 
and the two profiles $\Psi_{a_1},\Psi_{a_N}$ are genuinely distinct, so 
Lemma~\ref{lem:vertex-vals} applies with $\Psi_{a_j}''(L)\neq 0$.

\smallskip
Differentiating \eqref{eq:stationary-p} in $x$ shows that $\Psi_{a_j}'$ solves 
$\mathcal{L}_{+,j}\Psi_{a_j}'=0$ and decays exponentially at infinity.  Then, by the implications of the ODE's problem \eqref{Peigenvalue} with $\lambda=0$, the decaying solution space is 
one-dimensional with any $v_j\in H^2([L,\infty))$ satisfying 
$\mathcal{L}_{+,j}v_j=0$ being of the form
\begin{equation}\label{eq:v=cTheta'}
v_j(x)=c_j\,\Psi_{a_j}'(x),\qquad x\ge L,\qquad j\in\{1,N\}.
\end{equation}

\smallskip
Let $\mathbf v=(v_1,\dots,v_N)\in\ker(\mathcal{L}_{+,Z})$. Then 
\eqref{eq:v=cTheta'} implies the clustered form
 \begin{equation}\label{eq:cluster-form}
v_k=c_k\,\Psi_{a_1}',\ \,\   k=1,\dots,n,\qquad v_k=c_k\,\Psi_{a_N}',\ \ \ k=n+1,\dots,N,
\end{equation}
for some constants $c_k\in\R$. The relations $v'_j(L)=c_j\,\Psi_{a_j}''(L)$ and the $\delta'$ vertex conditions 
\begin{equation}\label{eq:deltaprime-BC-v}
v_1'(L)=\cdots=v_N'(L)=:d_0,
\qquad
\sum_{k=1}^N v_k(L)=Z\,d_0,
\end{equation}
imply immediately $c_k=c_1$ for $k=1,\dots,n$ and  $c_k=c_N$ for   $k=n+1,\dots,N$. Hence $d_0=c_1\Psi_{a_1}''(L)=c_N\Psi_{a_N}''(L)$. Thus, 
\begin{equation}\label{eq:cA-cB-d0}
c_1=\frac{d_0}{\Psi_{a_1}''(L)},\qquad c_N=\frac{d_0}{\Psi_{a_N}''(L)}.
\end{equation}
Using the common derivative of the standing wave profiles $d=\Psi_{a_1}'(L)=\Psi_{a_N}'(L)\neq 0$, the sum 
condition above gives
\[
Z\,d_0
=\sum_{k=1}^N v_k(L)
=d\bigl(nc_1+(N-n)c_N\bigr)
=d_0\,d\!\left(\frac{n}{\Psi_{a_1}''(L)}+\frac{N-n}{\Psi_{a_N}''(L)}\right).
\]
Note in case $d_0\neq 0$, then former equality implies the scalar compatibility condition
\begin{equation}\label{eq:exceptional-compat}
Z=d\!\left(\frac{n}{\Psi_{a_1}''(L)}+\frac{N-n}{\Psi_{a_N}''(L)}\right).
\end{equation}
A nontrivial $\mathbf{v}$ in $\ker(\mathcal{L}_{+,Z})$ therefore forces 
\eqref{eq:exceptional-compat} to be true, otherwise $d_0$ must vanish and one gets $c_1=c_N=0$ proving $\mathbf v\equiv 0$.

\smallskip 
In the following we see that \eqref{eq:exceptional-compat} contradicts the existence assumption $\omega>\omega_n^*$. From \eqref{eq:1-over-m}, we note 
\eqref{eq:exceptional-compat} is equivalent to
\begin{equation}\label{eq:alpha-compat-generalp}
-Z\sqrt\omega
=\frac{n\,t_1}{2t_1^2-1}+\frac{(N-n)\,t_N}{2t_N^2-1}.
\end{equation}
Subtracting the vertex relation $-Z\sqrt{\omega}=\alpha_n(t_1)$ in \eqref{eq:alphan} from the expression in \eqref{eq:alpha-compat-generalp} we obtain
\[
0=n\!\left(\frac{t_1}{2t_1^2-1}-\frac{1}{t_1}\right)
+(N-n)\!\left(\frac{t_N}{2t_N^2-1}-\frac{1}{t_N}\right)
=\frac{n(1-t_1^2)}{t_1(2t_1^2-1)}+\frac{(N-n)(1-t_N^2)}{t_N(2t_N^2-1)}.
\]
Using $t_1^2+t_N^2=1$ together 
with $2t_1^2-1=t_1^2-t_N^2$, the latter simplifies to
\[
0=\frac{n\,t_N^{2}}{t_1\,(t_1^2-t_N^2)}-\frac{(N-n)\,t_1^{2}}{t_N\,(t_1^2-t_N^2)}
=\frac{n\,t_N^{3}-(N-n)\,t_1^{3}}{t_1\,t_N\,(t_1^2-t_N^2)}.
\]
Since the denominator is non-zero (recall $t_1\neq t_N$ under our assumptions), 
this is equivalent to
\begin{equation}\label{eq:system-generalp}
n\,t_N^{3}=(N-n)\,t_1^{3},
\end{equation}
i.e., to $\mathfrak{H}(t_1)=0$, with $\mathfrak{H}$ as in \eqref{eq:Gdef}.

\smallskip
By Lemma~\ref{lem:vertex-vals} - (ii), $\mathfrak{H}(t_1)=0\iff\alpha_n'(t_1)=0$, 
which by subsection 4.2  forces $t_1=t_n^*$, equivalently 
$\omega=\omega_n^*$ (see \eqref{eq:alphan} and  \eqref{ExThr}).  Therefore \eqref{eq:exceptional-compat} fails in all 
regimes (both $t_1\in(0,1/\sqrt{2})$ and 
$t_1\in(1/\sqrt{2},1)$). Therefore, in the case hypothesis, it must be the case $d_0=0$ proving (1). 

\medskip
To prove item (2) we note from the case assumption that $t_1(\omega)=\tfrac{1}{\sqrt{2}}=t_N(\omega)$. Therefore the stationary solution $\Theta_{\omega,n}$ becomes continuous with all of its components equal to $\Psi_{a_*}$ where  $a_*=\operatorname{arctanh(1/\sqrt{2})}.$ Arguing as in the previous case, Sturm-Liouville theory imply that any element $\mathbf{v}=(v_1,\dots,v_N)$ in the kernel of $\mathcal{L}_{+,Z}$ is of the form \begin{equation}\label{eq:transition-ansatz}
v_j(x)=c_j\,\Psi_{a_*}'(x),\qquad c_j\in\R,\qquad j=1,\dots,N.
\end{equation} Since $\Psi_{a_*}''(L)=0$ we have that $v_j'(L)=0$ for all $j$ so that the derivative matching condition in $\dom$ is satisfied automatically with common value $d=0$. The sum condition then translates to \[
\Psi_{a_*}'(L)\sum_{j=1}^N c_j=\sum_{j=1}^N v_j(L)=Z \cdot 0= 0.
\] Since $\Psi_{a_*}'(L)\neq 0$ we then conclude it must be the case $\sum c_j=0$. Arguing as done with the set $\mathcal{K}_1$ in \eqref{eq:K1n-1} one can see that $\ker(\mathcal{L}_{+,Z})$ is $N-1$ dimensional. Which finishes the proof.
\end{proof}
\begin{remark}
It is worth mentioning that when $2n>N$ and $\omega=2N^2/Z^2$ the stationary profile $\Theta_{\omega,n}$ becomes continuous and the kernel and Morse index analysis for $\mathcal{L}_{+,Z}$ corresponds to the one in \cite[Proposition~3.24(iii)]{AngGol17b}. For sake of completeness we included a detailed proof using our machinery involving decomposition \eqref{eq:l2orth}.
\end{remark}

\section{Slope condition} \label{sec:insta}

In this section we  compute the sign of $\partial_\omega Q(\Theta_{\omega,n})$, with the profiles $\Theta_{\omega,n}$ given  in Theorem \ref{thm:existence} and the mass $Q$ defined in \eqref{mass}. 

\begin{theorem}[Slope condition]\label{thm:slope-condition}
Let $Z<0$, $N\ge 2$ and $n\in\{1,\dots,N-1\}$ be fixed. 
Assume $\omega>\omega_n^*$. 
Then:
\begin{enumerate}
\item[\textup{(i)}] If $2n\ge N$, then 
$\partial_\omega Q(\Theta_{\omega,n})>0$ for every $\omega>\omega_n^*$.
\item[\textup{(ii)}] If $2n<N$, then there exists $\omega_n^\sharp>\omega_n^*$ 
such that
\[
\begin{cases}
\partial_\omega Q(\Theta_{\omega,n})>0, & \omega>\omega_n^\sharp,\\[2pt]
\partial_\omega Q(\Theta_{\omega,n})=0, & \omega=\omega_n^\sharp,\\[2pt]
\partial_\omega Q(\Theta_{\omega,n})<0, & \omega_n^*<\omega<\omega_n^\sharp.
\end{cases}
\]
\end{enumerate}
\end{theorem}

By convenience of the reader, we record a fast-reference summary in Table~\ref{table:3} of the Appendix D.

\begin{proof}
Write $s=s(t):=\sqrt{1-t^2}$. In particular $s(t_1)=t_N$. By 
Proposition~\ref{prop:existencet}, $\omega\mapsto t_1(\omega)$ is a smooth 
strictly decreasing bijection from $(\omega_n^*,\infty)$ onto the existence 
branch $(0,t_n^*)$.

\medskip
Since the half-line soliton is $\mathfrak S_\omega(y)=\sqrt{2\omega}\,\operatorname{sech}(\sqrt\omega\,y)$, 
the substitution $u=\sqrt\omega(x-L)+a_j$ gives
\[
\int_L^\infty\!\Psi_{a_j}^2\,dx
=2\sqrt\omega\int_{a_j}^\infty\!\operatorname{sech}^2(u)\,du
=2\sqrt\omega\,(1-t_j),
\qquad t_j=\tanh(a_j),
\]
so that
\begin{equation}\label{eq:Q-explicit-cubic}
Q(\Theta_{\omega,n})=\frac12\sum_{j=1}^N\!\int_L^\infty\!\Psi_{a_j}^2\,dx
=\sqrt\omega\;h(t_1),
\qquad
h(t):=n(1-t)+(N-n)\bigl(1-s(t)\bigr),
\end{equation}
with $h'(t)=-n+(N-n)t/s$.

\medskip
By Proposition~\ref{prop:existencet}, 
$\frac{d}{d\omega} t_1(\omega)=-Z/\bigl(2\sqrt\omega\,\alpha_n'(t_1)\bigr)$, where the derivative of $\alpha_{n}$ is with respect to $t$. Differentiating 
\eqref{eq:Q-explicit-cubic} and using $-Z\sqrt\omega=\alpha_n(t_1(\omega))$ we get for $t_1=t_1(\omega)$
\begin{equation}
    \label{eq:der-Qw} \partial_\omega Q(\Theta_{\omega,n})
=\frac{h(t_1)}{2\sqrt\omega}-\frac{Z\,h'(t_1)}{2\,\alpha_n'(t_1)}
=\frac{(h\,\alpha_n)'(t_1)}{2\sqrt\omega\,\alpha_n'(t_1)}.
\end{equation}

A direct computation, in which the identity 
$s^2+t^2=1$ cancels every cross term, gives the closed form
\begin{equation}\label{eq:Nn-cubic-closed}
(h\,\alpha_n)'(t)=\frac{\mathcal N_n(t)}{t^2\,s^3},
\qquad
\mathcal N_n(t):=N\bigl((N-n)\,t^3-n\,s^3\bigr)-n(N-n)\,(2t^2-1).
\end{equation}
Since $t_1^2 s^3(t_1)>0$ and 
$\alpha_n'(t_1)<0$ on the existence branch,
\begin{equation}\label{eq:sign-slope-cubic}
\operatorname{sign}\bigl(\partial_\omega Q(\Theta_{\omega,n})\bigr)
=-\operatorname{sign}\bigl(\mathcal N_n(t_1(\omega))\bigr).
\end{equation}

\medskip
From \eqref{eq:Nn-cubic-closed} we have the end-point behavior 
\[
\mathcal N_n(0^+)=-Nn+n(N-n)=-n^2<0,
\qquad
\mathcal N_n(1^-)=N(N-n)-n(N-n)=(N-n)^2>0.
\]
At the minimizer $t_n^*$ one has $\alpha_n'(t_n^*)=0$, hence 
$(h\,\alpha_n)'(t_n^*)=h'(t_n^*)\,\alpha_n(t_n^*)$ and thus 
$\operatorname{sign}\left(\mathcal N_n(t_n^*)\right)$ agrees with $\operatorname{sign} \left( h'(t_n^*)\right)$. By 
\eqref{eq:cubic-infimum} we have  $t_n^*/s(t_n^*)=(n/(N-n))^{1/3}$. Therefore
\[
h'(t_n^*)=-n+(N-n)\,\Big(\frac{n}{N-n}\Big)^{1/3}=n^{1/3}\bigl[(N-n)^{2/3}-n^{2/3}\bigr],
\] from where $\operatorname{sign}\left(\mathcal N_n(t_n^*)\right)=\operatorname{sign}(N-2n)$.

\medskip
Differentiating \eqref{eq:Nn-cubic-closed},
\[
\mathcal N_n'(t)=t\,\phi(t),
\qquad
\phi(t):=3N\,[(N-n)t+ns]-4n(N-n).
\]
Since $s''=-s^{-3}<0$, the function $s$ is strictly concave down on $(0,1)$, and 
therefore so is $\phi$. A strictly concave down function has $\{t\,|\,\phi(t)>0\}$ equal to 
an open subinterval of $(0,1)$. Since the sign of  $\mathcal N_n'$ and $\phi$ agree 
on $(0,1)$, the function $\mathcal N_n$ is of decreasing–increasing–decreasing 
type, with each phase possibly absent. the latter combined with $\mathcal N_n(0^+)=-n^2<0$ 
and $\mathcal N_n(1^-)=(N-n)^2>0$, on any initial decreasing phase $\mathcal N_n$ 
stays negative, and on any final decreasing phase it stays above 
$\mathcal N_n(1^-)>0$; hence the single sign change occurs within the middle 
increasing phase. Thus $\mathcal N_n$ has exactly one zero 
$t_n^\sharp\in(0,1)$, with $\mathcal N_n<0$ on $(0,t_n^\sharp)$ and 
$\mathcal N_n>0$ on $(t_n^\sharp,1)$. 

\medskip
\emph{Case \textup{(i)}, $2n\ge N$.}\ In the balanced split case $2n=N$ (obviously with $N$ even), $\mathcal N_n(t_n^*)=0$, so 
$t_n^\sharp=t_n^*$. Thus $\mathcal N_n<0$ on the whole existence 
branch $(0,t_n^*)$.  Similarly, if $2n>N$, $\mathcal N_n(t_n^*)<0$ 
while $\mathcal N_n(1^-)>0$, so the zero satisfies 
$t_n^\sharp\in(t_n^*,1)$, outside the decreasing existence branch. Therefore 
$\mathcal N_n<0$ on all of $(0,t_n^*)$. In both cases \eqref{eq:sign-slope-cubic} gives 
$\partial_\omega Q(\Theta_{\omega,n})>0$ for every $\omega>\omega_n^*$. 

\emph{Case \textup{(ii)}, $2n<N$.}\ Here $\mathcal N_n(t_n^*)>0$ while 
$\mathcal N_n(0^+)<0$, so the zero satisfies 
$t_n^\sharp\in(0,t_n^*)$, with $\mathcal N_n<0$ on $(0,t_n^\sharp)$ and 
$\mathcal N_n>0$ on $(t_n^\sharp,t_n^*)$. Setting
\begin{equation}\label{velha}
\omega_n^\sharp:=\frac{\alpha^2_n(t_n^\sharp)}{Z^2},
\end{equation}
the strict monotonicity of $\alpha_n$ on $(0,t_n^*)$ gives 
$\alpha_n(t_n^\sharp)>\alpha_n(t_n^*)$, hence $\omega_n^\sharp>\omega_n^*$. 
Since both $$\omega<\omega_n^\sharp\iff t_1(\omega)> t_n^\sharp \ \mbox{ and }\ \omega>\omega_n^\sharp\iff t_1(\omega)< t_n^\sharp,$$ the 
sign rule \eqref{eq:sign-slope-cubic} yields the stated trichotomy, with 
$\partial_\omega Q=0$ exactly at $\omega=\omega_n^\sharp$.
\end{proof}

\begin{remark} 
Consider again the case $N=4$ and, for simplicity, set $Z=-1$. In this case $\omega_1^*=\omega_3^*\approx 29.2205$, $\omega_2^*=\widehat{\omega}=32$ and $\omega_1^\sharp\approx 31.3094$.

\begin{enumerate}
\item In Figure~\ref{fig:nnp} we plot the behavior of $\mathcal{N}_n(t)$ in $(0,1)$ for all the admissible values of $n$. We highlight the fact the root of $\mathcal{N}_n$ is placed exactly at $1/\sqrt{2}$ exactly in the balanced case $n=2$. 

\item  Fix $\omega\in(\omega_1^*,\omega_1^\sharp)$. We note that even when the configuration (types of profiles) of $\Theta_{\omega,1}$ and $\Theta_{\omega,3}$ are the same, we have different charges behavior because $t_1(\omega)$ depends on $n$. See Figure~\ref{fig:13}.
\end{enumerate}

\begin{figure}
    \centering
    \resizebox{\linewidth}{!}{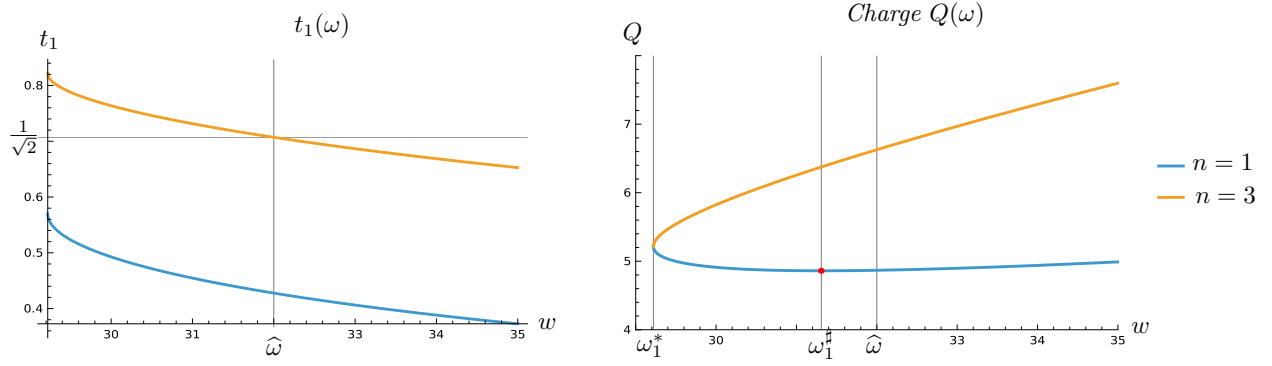}
    \caption{For $N=4$ and $Z=-1$ we plot the charge $Q$ (on the right) and $t_1$ (on the left) in terms of $\omega$ for $n=1$ and $n=3$.}
    \label{fig:13}
\end{figure}

\begin{figure}
    \centering
    \resizebox{0.55\linewidth}{!}{
\begingroup%
  \makeatletter%
  \providecommand\color[2][]{%
    \errmessage{(Inkscape) Color is used for the text in Inkscape, but the package 'color.sty' is not loaded}%
    \renewcommand\color[2][]{}%
  }%
  \providecommand\transparent[1]{%
    \errmessage{(Inkscape) Transparency is used (non-zero) for the text in Inkscape, but the package 'transparent.sty' is not loaded}%
    \renewcommand\transparent[1]{}%
  }%
  \providecommand\rotatebox[2]{#2}%
  \newcommand*\fsize{\dimexpr\f@size pt\relax}%
  \newcommand*\lineheight[1]{\fontsize{\fsize}{#1\fsize}\selectfont}%
  \ifx\svgwidth\undefined%
    \setlength{\unitlength}{314.87226105bp}%
    \ifx\svgscale\undefined%
      \relax%
    \else%
      \setlength{\unitlength}{\unitlength * \real{\svgscale}}%
    \fi%
  \else%
    \setlength{\unitlength}{\svgwidth}%
  \fi%
  \global\let\svgwidth\undefined%
  \global\let\svgscale\undefined%
  \makeatother%
  \begin{picture}(1,0.54123794)%
    \lineheight{1}%
    \setlength\tabcolsep{0pt}%
    \put(0,0){\includegraphics[width=\unitlength,page=1]{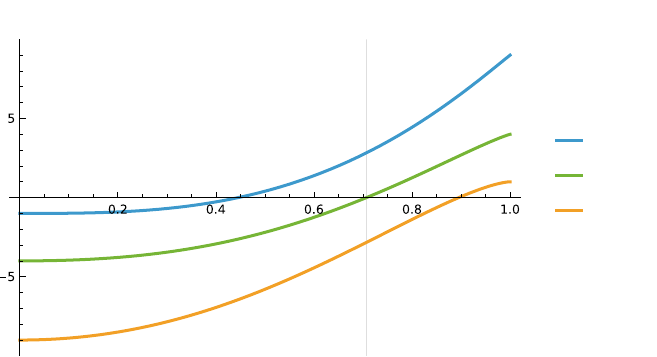}}%
    \put(0.90137573,0.32040147){\color[rgb]{0,0,0}\makebox(0,0)[lt]{\lineheight{1.25}\smash{\begin{tabular}[t]{l}$n=1$\end{tabular}}}}%
    \put(0.90325899,0.26500179){\color[rgb]{0,0,0}\makebox(0,0)[lt]{\lineheight{1.25}\smash{\begin{tabular}[t]{l}$n=2$\end{tabular}}}}%
    \put(0.90564091,0.21230914){\color[rgb]{0,0,0}\makebox(0,0)[lt]{\lineheight{1.25}\smash{\begin{tabular}[t]{l}$n=3$\end{tabular}}}}%
    \put(0.80067673,0.23228987){\color[rgb]{0,0,0}\makebox(0,0)[lt]{\lineheight{1.25}\smash{\begin{tabular}[t]{l}$t$\end{tabular}}}}%
    \put(0.02510092,0.49759295){\color[rgb]{0,0,0}\makebox(0,0)[lt]{\lineheight{1.25}\smash{\begin{tabular}[t]{l}$\mathcal{N}_n$\end{tabular}}}}%
    \put(0.44532553,0.51889466){\color[rgb]{0,0,0}\makebox(0,0)[lt]{\lineheight{1.25}\smash{\begin{tabular}[t]{l}Plot of $\mathcal{N}_n(t)$\end{tabular}}}}%
    \put(0.56732623,0.03573316){\color[rgb]{0,0,0}\makebox(0,0)[lt]{\lineheight{1.25}\smash{\begin{tabular}[t]{l}$t=\tfrac{1}{\sqrt{2}}$\end{tabular}}}}%
  \end{picture}%
\endgroup%
}
    \caption{Plot of the function $\mathcal{N}_n$ defined in \eqref{eq:Nn-cubic-closed} over $(0,1)$.}
    \label{fig:nnp}
\end{figure}
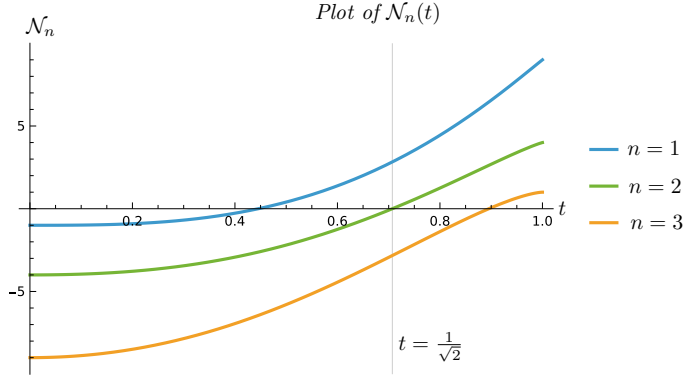

\end{remark}

\section{Proof of Theorem~\ref{thm:insta}} \label{sec:mainproof}

In the following we show the  (in)stability  behavior of the standing waves solutions establishes in Theorem \ref{thm:existence} by the flow of the cubic-NLS based in part on Theorem \ref{2main} in Appendix C.

\begin{proof}
From Theorem~\ref{global} we know that \eqref{NLS} is globally well-posed in $H^1(\dg)$ for any $N\ge 1$ and $p=3$ with data-solution map at least of class $C^2$. Moreover, from Theorem~\ref{thm:existence} we have the existence of a smooth curve of bound states \(\omega \mapsto e^{it\omega}\Theta_{\omega,n}\) in the energy space $H^1(\dg)$. Let us also note that such properties are inherited by the subspace $H^1(\dg)\cap \xm$ because of the uniqueness of the solutions to the Cauchy problem in $H^1(\dg)$ (Theorem \ref{global}) and the fact the group $e^{it\mathcal{H}_Z^{\delta'}}$ preserves the orthogonal decomposition $L^2(\dg)=\xa\oplus\xb\oplus\xm$ (see Lemma~\ref{lem:group}. See also a similar study in \cite{AC, AP3}).

\medskip
Using Theorem~\ref{2main} (and Remark~\ref{lintononlin}), and taking into account Propositions~\ref{prop:l2}, \ref{prop:ker-L1Z} and Theorem~\ref{thm:morse1}; we have the following summary. 
\begin{itemize}
    \item For the cases in (1) it follows $n(\mathcal{H})=1=\rho(\omega)$ and therefore Theorem~\ref{2main} yields orbital stability in $H^1(\dg)$. 
    \item For the cases in (2) it follows $n\big(\mathcal{H}\bigr|_{\xm\cap\dom}\big)=1=\rho(\omega)$ and therefore Theorem~\ref{2main} yields orbital stability in $\xm\cap H^1(\dg)$.    
    \item For case (3)(c) we note $n\big(\mathcal{H}\bigr|_{\xm\cap\dom}\big)=1$ while $\rho(\omega)=0$. Therefore $n\big(\mathcal{H}\bigr|_{\xm\cap\dom}\big)-\rho(\omega)$ is odd and Theorem~\ref{2main} combined with Remark~\ref{lintononlin} yields orbital instability in $\xm\cap H^1(\dg)$. Instability in a subspace induces instability in the whole $H^1(\dg)$.
    \item In case (3)(e) it follows $n\big(\mathcal{H}\bigr|_{\xm\cap\dom}\big)=2$ and $\rho(\omega)=1$. From Theorem~\ref{2main} and Remark~\ref{lintononlin} it follows orbital instability in $\xm\cap H^1(\dg)$. Finally, instability in a subspace induces instability in the whole $H^1(\dg)$.
    \end{itemize}
It remains to address cases (3)(a), (3)(b), and (3)(d). Propositions~\ref{prop:l2}, \ref{prop:ker-L1Z} and Theorem~\ref{thm:morse1} gives $n(\mathcal{H})=n$, while Theorem~\ref{thm:slope-condition} yields $\rho(\omega)=1$. Consequently, for  $n$ even, the difference $n(\mathcal{H})-\rho(\omega)$ is odd, so Theorem~\ref{2main} together with Remark~\ref{lintononlin} establishes orbital instability in $H^1(\dg)$. The parity restriction can nevertheless be removed by adapting the argument of \cite[Theorem~1.2]{Gril1988}: in view of \cite[Section~5.1]{G} and Remark~\ref{lintononlin}, it is enough to exhibit some $\lambda>0$ for which
\begin{equation}
    \label{eq:spectral-insta}
    \begin{pmatrix}
        0 & \mathcal{L}_{-,Z}\\
        -\mathcal{L}_{+,Z}&0
    \end{pmatrix} \begin{pmatrix}
        u_1\\ u_2
    \end{pmatrix}=\lambda \begin{pmatrix}
        u_1\\u_2
    \end{pmatrix}, \qquad u_1, u_2\in \dom.
\end{equation}
Combining \cite[Theorem~1.2]{Gril1988} with the spectral results of Section~\ref{sec:spectral}, the number of positive eigenvalues $\lambda$ satisfying \eqref{eq:spectral-insta} is at least
\[
n(P\mathcal{L}_{+,Z})-n\!\left(P(\mathcal{L}_{-,Z})^{-1}\right)=n-1-0=n-1,
\]
where $P$ denotes the orthogonal projection onto the codimension-one subspace $[\operatorname{span}\{\Theta_{\omega,n}\}]^\perp$. Since $n\ge 2$ in each of the cases the desired instability follows.

\end{proof}

\section{The general case \texorpdfstring{$p>1$, $p\neq 3$}{p>1}}\label{sec:generalp}

Throughout this section the half--line profiles are the positive decaying
solutions of
\[
-\Psi''+\omega \Psi-\Psi^{p}=0, \qquad x\ge L,
\]
namely the shifted one--dimensional solitons
\[
\mathfrak{S}_\omega(y)=\Big(\tfrac{(p+1)\omega}{2}\Big)^{1/(p-1)}
\operatorname{sech}^{2/(p-1)}\!\Big(\tfrac{p-1}{2}\sqrt{\omega}\,y\Big),
\]
so that
\begin{equation}\label{solitonpp}
\Psi_{a_j}(x) = \mathfrak{S}_\omega\!\Big(x-L+\tfrac{2}{\sqrt{\omega}(p-1)}\,a_j\Big),
\qquad x \geq L, \qquad 1\le j\le N .
\end{equation}
The construction of discontinuous-at-the-vertex stationary solutions carried out in Section~\ref{sec:exists} for the cubic NLS ($p=3$) extends to a general focusing power nonlinearity $p>1$, subject to certain restrictions; several objects that were explicit in the cubic case are now only implicit, which introduces additional obstacles. We develop in turn the branch system and its existence (Section~\ref{sec:generalp-existence}), the spectral analysis (Section~\ref{sec:spectral-generalp}, which remains unchanged for any $p>1$), and the slope condition (Section~\ref{sec:slope-generalp}). These ingredients are combined in Section~\ref{sec:insta-generalp} to yield a \emph{high-frequency} (in)stability statement valid for all $p>1$ (stability for $1<p\le 5$ and instability for $p>5$), which complements and extends previous results in the literature on delta-prime type interactions on star graphs (see \cite{ABG, AngGol17a, AngGol17b, G}). Remark~\ref{rem:generalp-remarks} discusses the low-frequency obstructions and states a global conjecture.

\subsection{Existence of standing waves}
\label{sec:generalp-existence}

We collect the points at which the general analysis departs from the cubic one;
all qualitative conclusions of Section~\ref{sec:exists} persist, but several
explicit objects available for $p=3$ must be replaced by their implicit
counterparts.

\subsubsection{Derivative matching.}
Setting $t_j:=\tanh(a_j)$, the derivative matching at the vertex reads
\[
g_p(t_j):=t_j(1-t_j^2)^{1/(p-1)}=\mu, \qquad j=1,\dots,N,
\]
and $g_p$ attains its unique maximum on $(0,1)$ at
\[
t_p^*=\sqrt{\tfrac{p-1}{p+1}},
\qquad a_*^{(p)}:=\operatorname{arctanh}(t_p^*),
\]
the analogue of the cubic value $1/\sqrt 2$. As before, for each admissible
$\mu$ there are exactly two shifts $a_A(\mu)<a_*^{(p)}<a_B(\mu)$, separating
type--$A$ ($t<t_p^*$) from type--$B$ ($t>t_p^*$) profiles.

The essential simplification of the cubic case is the explicit involution
$t_N=\sqrt{1-t_1^2}$ coming from the identity $t_A^2+t_B^2=1$. For general $p$
this relation becomes implicit: raising $g_p(t_A)=g_p(t_B)$ to the power $p-1$
gives the level--set equation
\[
h(t):=t^{p-1}(1-t^2),\qquad h(t_A)=h(t_B),
\]
which defines a smooth involution $\xi_p$ on $(0,1)\setminus\{t_p^*\}$ with
$\xi_p(t)\in[t_p^*,1)$ for $t\in(0,t_p^*]$ and $\xi_p(t)\in(0,t_p^*]$ for
$t\in[t_p^*,1)$. For $p=3$ this factors explicitly, $\xi_3(t)=\sqrt{1-t^2}$.

\subsubsection{The sum condition, a decreasing branch system and a existence threshold}
With $t_1=t$ and $t_N=\xi_p(t)$, the sum condition takes the form
\begin{equation}\label{eq:alphanp}
\alpha_{n,p}(t):=\frac{n}{t}+\frac{N-n}{\xi_p(t)}=\alpha(\omega)=-Z\sqrt{\omega},
\qquad t\in(0,1).
\end{equation}
For $p=3$ the function $\alpha_{n,3}$ has a closed--form minimizer; for general
$p$ no closed form for the infimum of $\alpha_{n,p}$ is available, and when $p$
is not an integer a unified proof of uniqueness of the minimizer on $(0,1)$
seems out of reach. Nevertheless the asymptotics $\alpha_{n,p}(t)\to\infty$ as
$t\to0^+$ and $t\to1^-$, together with the sign of
$\alpha_{n,p}'(t_p^*)=(N-2n)/(t_p^*)^2$, guarantee at least one interior
minimizer $t_{n,p}^*$ and a well--defined decreasing branch system
\begin{equation}\label{branch}
\{(0,t_{n,p}^*),\ \omega_{n,p}^*,\ t_1(\omega)\},
\end{equation}
 with existence threshold
\begin{equation}\label{eq:wnp*}
\omega_{n,p}^*=\frac{\big(\alpha_{n,p}(t_{n,p}^*)\big)^2}{Z^2}.
\end{equation}

\begin{remark}\label{p-thres}The threshold $\omega_{n,p}^*$ in \eqref{eq:wnp*} is optimal whenever the minimizer $t_{n,p}^*$ is unique on
$(0,1)$; numerical evidence suggests uniqueness for all $p>1$ (see
Remark~\ref{rem:generalp-remarks}), but no analytic unified result is available. The situation is markedly different for $p\in\{2,3,4\}$: in these cases the level-set equation $h(t)=h(\xi_p(t))$ reduces to a polynomial identity that can be explicitly factored, thereby yielding a closed-form expression for $\xi_p$. This explicit representation is precisely what makes the analytic proof of uniqueness feasible for $p\in\{2,3,4\}$, whereas for other values of $p$ no such factorization is available. 
\end{remark}
\subsubsection{The secondary threshold.}
The $n$--independent secondary threshold
\[
\widehat\omega=\frac{p+1}{p-1}\,\frac{N^2}{Z^2}
\]
replaces the cubic value $2N^2/Z^2$. It arises by evaluating $\alpha_{n,p}$ at
the symmetric point $t=t_p^*$, where $\alpha_{n,p}(t_p^*)=N/t_p^*$ independently
of $n$, and governs the inversion of tail type exactly as in the cubic case.

\medskip
The above construction yields the following existence statement, the analogue of
the cubic Theorem~\ref{thm:existence}.

\begin{theorem}\label{thm:existence-p}
Let $Z<0$, $N\ge 2$, $p>1$ and fix $n\in\{1,\dots,N-1\}$, and let $\omega_{n,p}^*$
be given by \eqref{eq:wnp*}. Then for every $\omega>\omega_{n,p}^*$ there exist
two shifts $0<a_A(\omega)<a_*^{(p)}<a_B(\omega)$ such that the stationary profile
$\Theta_{\omega,n}=(\Psi_{a_1},\dots,\Psi_{a_N})$, carrying $a_1$ on the first
$n$ edges and $a_N$ on the remaining $N-n$, belongs to $\dom$. Moreover
\[
(\omega_{n,p}^*,\infty)\ni \omega \longmapsto e^{it\omega}\Theta_{\omega,n}\in
D(\mathcal{H}_{Z}^{\delta'})
\]
is a smooth curve of standing waves for \eqref{NLS} on $\dg$, discontinuous at
the vertex.
\end{theorem}



\subsection{Spectral analysis}
\label{sec:spectral-generalp}

The spectral analysis of Section~\ref{sec:spectral} is, to a large extent,
insensitive to $p$. The Hessian splitting \eqref{S2}, the diagonal operators
$\mathcal L_{\pm,Z}$ of \eqref{L+} with
\begin{equation}
\mathcal{L}_{+,j}:=-\partial_x^2+\omega-p\,\Psi_{a_j}^{\,p-1},
\qquad
\mathcal{L}_{-,j}:=-\partial_x^2+\omega-\Psi_{a_j}^{\,p-1},
\qquad j\in\{1,N\},\ p>1,
\end{equation}
and the index decomposition
$\dom=(\dom\cap X_{1,\perp})\oplus(\dom\cap X_{N,\perp})\oplus(\dom\cap\xm)$ use
only that each cluster profile solves the stationary equation on a half--line,
that it is positive and exponentially decaying, and meets the $\delta'$ vertex
conditions. Consequently every statement of Section~\ref{sec:spectral} has a
$p>1$ counterpart; we highlight the modifications.

\subsubsection{The vertex--value lemma.}
Lemma~\ref{lem:vertex-vals} carries over with the $p$--dependent boundary
identities
\[
\Psi_{a_j}'(L)=-\sqrt{\omega}\,\Psi_{a_j}(L)\,t_j,
\qquad
\Psi_{a_j}''(L)=\frac{\omega\,\Psi_{a_j}(L)}{2}\bigl[(p+1)t_j^2-(p-1)\bigr],
\qquad j\in\{1,N\},
\]
so $\Psi_{a_j}''(L)=0$ iff $t_j=t_p^*$. The cubic involution $t_1^2+t_N^2=1$ is
replaced by 
$$
\xi^{p-1}_p(t)[1-\xi^2_p(t)]=t^{\,p-1}(1-t^2), 
$$
$t_N=\xi_p(t_1)$,
and the auxiliary function of Lemma~\ref{lem:vertex-vals}\,(ii) becomes
\[
\mathfrak H(t):=n\,\xi^p_p(t){\,p}\bigl[(p+1)\xi^2_p(t)-(p-1)\bigr]
-(N-n)\,t^{\,p}\bigl[(p-1)-(p+1)t^2\bigr],
\]
with implicit differentiation giving the structural identity
\begin{equation}\label{eq:5.1iip}
\alpha_{n,p}'(t)
=-\frac{\mathfrak H(t)}{t^2\,\xi^p_p(t)\,\bigl[(p+1)\xi^2_p(t)-(p-1)\bigr]},
\qquad t\in(0,t_p^*).
\end{equation}

\subsubsection{Non--negativeness and kernel of $\mathcal L_{-,Z}$.}
Proposition~\ref{prop:l2} transfers verbatim: its proof rests on the
factorization \eqref{eq:factorLp}, valid for every $p>1$, and invokes only
$\Psi_{a_j}(L)>0$, $\Psi_{a_j}'(L)<0$, the common--derivative condition and
$Z<0$.

\subsubsection{The kernel of $\mathcal L_{+,Z}$ on the decreasing branch system \eqref{branch}}
Proposition~\ref{prop:ker-L1Z} persists: a nontrivial kernel element reduces to
the compatibility $\mathfrak H(t_1)=0$, which by \eqref{eq:5.1iip} is equivalent
to $\alpha_{n,p}'(t_1)=0$ and happens only at $\omega=\omega_{n,p}^*$  where $t_1$ attains the decreasing branch minimum $t_{n,p}^*$.

\subsubsection{The Morse index.}
The Morse--index count of Theorem~\ref{thm:morse1} goes through unchanged, both
ingredients being $p$--general: the half--line Neumann theory of
Appendix~\ref{ap:2} produces monotone functions $m_1(\lambda),m_N(\lambda)$ and
the scalar equation $F(\lambda)=Z$ for any $p>1$, and
Lemma~\ref{lem:LA-negative-generalp} shows that $\mathfrak L_{+,j}$ has a simple
negative Neumann eigenvalue exactly when $t_j<t_p^*$. The quantity $F(0)-Z$
equals, up to a strictly positive factor, the sign of $\mathfrak H(t_1)$.
Phrasing the dichotomy directly through the position of $t_1(\omega)$ relative
to $t_p^*$, the conclusion of Theorem~\ref{thm:morse1} holds verbatim for all
$p>1$:
\begin{equation}\label{eq:morse-dichotomy-p}
\begin{aligned}
t_1(\omega)\in(0,t_p^*)\ &\Longrightarrow\
n(\mathcal L_{+,Z})=n,\qquad
n\bigl(\mathcal L_{+,Z}|_{\dom\cap\xm}\bigr)=1,\\[2pt]
t_1(\omega)\in(t_p^*,1)\ &\Longrightarrow\
n(\mathcal L_{+,Z})=N-n+1,\qquad
n\bigl(\mathcal L_{+,Z}|_{\dom\cap\xm}\bigr)=2.
\end{aligned}
\end{equation}
Translating the position of $t_1$ into clean inequalities on $N,n,\omega$ (as in
the cubic case statement) requires, for general $p$, either a closed form for
$t_{n,p}^*$ or its uniqueness as minimizer of $\alpha_{n,p}$; this is the
only point at which the cubic and general cases truly diverge, however it is purely
presentational.

\subsection{The slope condition}
\label{sec:slope-generalp}

\subsubsection{The charge and the slope functional on the decreasing branch system \eqref{branch}}
The cubic computation of the charge rests on the elementary primitive of
$\operatorname{sech}^2$. For general $p$ the same substitution produces
\[
J_p(t):=\int_t^1(1-\tau^2)^{\beta(p)}\,d\tau,
\qquad \beta(p)=\frac{3-p}{p-1}>-1,
\]
so that
\[
Q(\Theta_{\omega,n})=K_p\,\omega^{\gamma(p)}\,h_p(t_1),
\qquad
h_p(t):=n\,J_p(t)+(N-n)\,J_p\bigl(\xi_p(t)\bigr),
\qquad
\gamma(p)=\frac{5-p}{2(p-1)},
\]
with $K_p>0$ an $\omega$--independent constant; $\gamma(p)>0$ exactly on
$1<p<5$, vanishes at the mass--critical exponent $p=5$, and reverses sign
beyond it. Differentiating $Q$ along the branch yields the $p$--analogue of the
cubic identity \eqref{eq:Nn-cubic-closed}:
\begin{equation}\label{eq:Nnp}
\partial_\omega Q(\Theta_{\omega,n})
=\frac{K_p\,\omega^{\gamma(p)-1}}{2(p-1)\,\alpha_{n,p}'(t_1)}\,
\mathcal N_n^p\bigl(t_1(\omega)\bigr),
\qquad
\mathcal N_n^p(t):=(5-p)\,h_p(t)\,\alpha_{n,p}'(t)+(p-1)\,\alpha_{n,p}(t)\,h_p'(t).
\end{equation}
Since $K_p>0$, $\omega^{\gamma(p)-1}>0$, $(p-1)>0$ and $\alpha_{n,p}'(t_1)<0$ on the decreasing
branch \eqref{branch}, \eqref{eq:Nnp} gives the sign rule
\begin{equation}\label{eq:sign-slope-p}
\operatorname{sign}\bigl(\partial_\omega Q(\Theta_{\omega,n})\bigr)
=-\operatorname{sign}\bigl(\mathcal N_n^p(t_1(\omega))\bigr),
\end{equation}
exactly as in \eqref{eq:sign-slope-cubic}.

\subsubsection{Endpoint asymptotics of $\mathcal N_n^p$.}
As $t\to0^+$ one has $\xi_p(t)\to1^-$, hence
\[
\alpha_{n,p}(t)\sim \frac{n}{t},\qquad
\alpha_{n,p}'(t)\sim -\frac{n}{t^2},\qquad
h_p(t)\to n\,J_p(0)>0,\qquad
h_p'(t)\to -n,
\]
the last two following from $J_p(\xi_p(t))\to0$, $J_p'(t)\to-1$ and
$J_p'(\xi_p(t))\,\xi_p'(t)\to0$. The two contributions to $\mathcal N_n^p$ in
\eqref{eq:Nnp} are of \emph{different} orders: the $\alpha_{n,p}'$ term is
$O(t^{-2})$ while the $\alpha_{n,p}$ term is only $O(t^{-1})$. Therefore, the former term
dictates the limit whenever its coefficient $(5-p)$ does not vanish and the sign of
$\mathcal N_n^p$ near $t=0$ is \emph{definite for every} $p>1$. Analyzing $\mathcal{N}_n^p$ in \eqref{eq:Nnp} as $t\to0^+$:
\begin{itemize}
\item For $1<p<5$ the leading term $(5-p)\,h_p(t)\,\alpha_{n,p}'(t)\sim
-(5-p)\,n^2 J_p(0)\,t^{-2}\to-\infty$ and dominates the lower--order
$(p-1)\alpha_{n,p}(t)h_p'(t)\sim-(p-1)n^2 t^{-1}$. Therefore $\mathcal N_n^p(t)\to-\infty$.
\item For $p=5$ the leading term vanishes identically and the surviving term
$(p-1)\alpha_{n,p}(t)h_p'(t)\sim-4n^2 t^{-1}\to-\infty$, so again
$\mathcal N_n^p(t)\to-\infty$.
\item For $p>5$ the leading term reverses sign and so
$(5-p)\,h_p(t)\,\alpha_{n,p}'(t)\sim(p-5)\,n^2 J_p(0)\,t^{-2}\to+\infty$. However, it still dominates the lower-order $\alpha_{n,p}$ term, leading to $\mathcal N_n^p(t)\to+\infty$.
\end{itemize}
Hence
\begin{equation}\label{eq:Nn-limit}
\mathcal N_n^p(t)\xrightarrow[t\to0^+]{}
\begin{cases}
-\infty, & 1<p\le 5,\\[2pt]
+\infty, & p>5.
\end{cases}
\end{equation}
We conclude $\mathcal N_n^p$ is strictly negative on a right neighbourhood of $0$ when
$1<p\le5$ and strictly positive there when $p>5$.

\subsubsection{The high--frequency threshold for the decreasing branch system \eqref{branch}}
By \eqref{eq:Nn-limit}, $\mathcal N_n^p$ keeps a constant sign on a right
neighborhood of $0$, therefore the value
\[
t_{n,p}^\flat:=\sup\bigl\{\tau\in(0,t_p^*]:\ \mathcal N_n^p\ \text{has constant sign on}\ (0,\tau)\bigr\}\in(0,t_p^*]
\]
is well defined. Setting $\hat t_{n,p}:=\min\{t_{n,p}^\flat,\,t_{n,p}^*\}\le t_p^*$
we define the \emph{high--frequency threshold}
\begin{equation}\label{eq:omega-starstar}
\omega_{n,p}^{**}:=\frac{\bigl(\alpha_{n,p}(\hat t_{n,p})\bigr)^2}{Z^2}
\ \ \ge\ \omega_{n,p}^*.
\end{equation}
By construction $\omega>\omega_{n,p}^{**}$ is equivalent to
$t_1(\omega)\in(0,\hat t_{n,p})$, which places $t_1(\omega)$ in the type--$A$
region $(0,t_p^*)$ and pins down the sign of $\mathcal N_n^p(t_1(\omega))$. By the
sign rule \eqref{eq:sign-slope-p} the latter is translated as follows:
\begin{equation}\label{eq:slope-positive-hf}
\omega>\omega_{n,p}^{**}\ \Longrightarrow\
t_1(\omega)\in(0,t_p^*)
\quad\text{and}\quad
\partial_\omega Q(\Theta_{\omega,n})
\begin{cases}>0,&1<p\le5,\\[2pt]<0,&p>5.\end{cases}
\end{equation}

\subsection{(In)stability result for the decreasing branch system \eqref{branch}}
\label{sec:insta-generalp}

The Morse--index dichotomy \eqref{eq:morse-dichotomy-p} and the high--frequency
slope sign \eqref{eq:slope-positive-hf} combine into the main result of the
section, in which the mass--critical exponent $p=5$ separates a stability regime
from an instability one.
\begin{theorem}
\label{thm:insta-highfreq-p}
Let $N\ge 2$, $Z<0$, $p>1$ and $n\in\{1,\dots,N-1\}$, and let
$\omega_{n,p}^{**}\ge\omega_{n,p}^*$ be the high--frequency threshold
\eqref{eq:omega-starstar}. Then for every $\omega>\omega_{n,p}^{**}$ the
standing wave
\[
e^{i\omega t}\Theta_{\omega,n}=e^{i\omega t}\bigl(\Psi_{a_1(\omega)},\dots,\Psi_{a_N(\omega)}\bigr)
\]
provided by Theorem~\ref{thm:existence-p} satisfies:
\begin{enumerate}
\item[\rm(i)] if $1<p\le5$:
  \begin{enumerate}
  \item[\rm(a)] it is orbitally stable in $H^1(\dg)$ whenever $n=1$;
  \item[\rm(b)] it is orbitally stable in $\xm\cap H^1(\dg)$ for every
  $n\in\{1,\dots,N-1\}$;
  \item[\rm(c)] it is spectrally unstable (see Remark~\ref{lintononlin}) in $H^1(\dg)$ whenever $n\ge 2$ and orbitally unstable in $H^1(\dg)$ provided $n\ge2$ and $2<p\le 5$;
  \end{enumerate}
\item[\rm(ii)] if $p>5$: it is orbitally unstable in $\xm\cap H^1(\dg)$, and
hence in $H^1(\dg)$, for every $n\in\{1,\dots,N-1\}$.
\end{enumerate}
\end{theorem}

\begin{proof}
Fix $\omega>\omega_{n,p}^{**}$. By \eqref{eq:slope-positive-hf},
$t_1(\omega)\in(0,t_p^*)$, so the dichotomy \eqref{eq:morse-dichotomy-p} gives the Morse index values
\[
n(\mathcal L_{+,Z})=n,\qquad
n\bigl(\mathcal L_{+,Z}|_{\dom\cap\xm}\bigr)=1,
\]
while the sign \eqref{eq:slope-positive-hf} governs the value of $\rho(\omega)$.
We feed these into Theorem~\ref{2main} as in the proof of Theorem~\ref{thm:insta} to land the conclusion for cases (i)(a), (i)(b) and (ii). It is worth mentioning that, at the $L^2$-critical exponent $p=5$, the stability statements in items (i)(a) and (i)(b) are conditional on global existence of the solutions near $\Theta_{\omega,n}$ which is extracted from the orbital confinement, as follows.
Let $\mathbf{U}$ denote the solution with initial datum $\mathbf{U}_0$, defined on its maximal interval $[0,T^*)$. By the localized Lyapunov argument underlying Theorem~\ref{2main}, there exists $\varepsilon_0>0$ such that for every $\varepsilon\in(0,\varepsilon_0)$ one can find $\delta>0$ with the following property: if $\|\mathbf{U}_0-\Theta_{\omega,n}\|_{H^1}<\delta$, then
\[
\inf_{\theta\in\mathbb{R}}\,\bigl\|\mathbf{U}(t)-e^{i\theta}\Theta_{\omega,n}\bigr\|_{H^1(\dg)}<\varepsilon
\qquad\text{for all }t\in[0,T^*).
\]
The reverse triangle inequality yields
\[
\|\mathbf{U}(t)\|_{H^1(\dg)}<\varepsilon+\|\Theta_{\omega,n}\|_{H^1(\dg)}
\qquad\text{for all }t\in[0,T^*).
\]
Thus $\sup_{t\in[0,T^*)}\|\mathbf{U}(t)\|_{H^1(\dg)}<\infty$, and the blow-up alternative forces $T^*=+\infty$. Hence $\mathbf{U}$ exists globally in $H^1(\dg)$, and the orbital stability statements of items (i)(a) and (i)(b) are unconditional (for item (b) we note the space $\xm$ is invariant under the flow and that the norm of $\xm\cap H^1(\dg)$ coincides with the
$H^1(\dg)$ norm).

Finally to see item (i)(c) we again use the same argument of the proof of Theorem~\ref{thm:insta} involving \cite[Theorem~1.2]{Gril1988} to obtain spectral instability for $p>1$. We note from Remark~\ref{lintononlin} that spectral instability can be improved to orbital instability provided the data-solution map is at least of class $C^2$, which is the case for $p>2$ according to Theorem~\ref{global}. This finishes the proof. 
\end{proof}

\begin{remark}
\label{rem:generalp-remarks}

We close the section by recording the obstructions that confine the
\emph{stability} conclusions of Theorem~\ref{thm:insta-highfreq-p} to high
frequencies, and by formulating the expected full statement for $1<p\le5$ as a
conjecture.
\begin{enumerate}
\item  The low--frequency analysis in the cubic case ($p=3$) depends on fine properties of the minimum of
$\alpha_{n,p}=\alpha_n$ in \eqref{eq:alphan}: the location of $t_{n,p}^*$ relative to $t_p^*$ and the
monotonicity of $n\mapsto t_{n,p}^*$ provided by the \emph{uniqueness} of
the minimizer of $\alpha_{n,p}$ on $(0,1)$. For general $p>1$, an analytic proof of this uniqueness is delicate, since it requires controlling the implicit involution $\xi_p$, which is not available in closed form, however numerical evidence consistently indicates a unique minimizer for all $p>1$.  The high--frequency
threshold \eqref{eq:omega-starstar}, by contrast, requires none of this (see Remark \ref{p-thres}), it uses
only $\mathcal N_n^p(t)\to\pm\infty$ as $t\to0^+$ and the (possibly) not-optimal decreasing branch system of Theorem~\ref{thm:existence-p}.

\item In the cubic case the control in the sign $\partial_\omega Q(\Theta_{\omega,n})$ was given by the uniqueness of a root of the function $\mathcal{N}_n^3$ in $(0,1)$. 
However, for any $p>1$, establishing such uniqueness is a
delicate task because the factor $\alpha_{n,p}'$ entering $\mathcal N_n^p$ ties this
question to the uniqueness of the minimizer of $\alpha_{n,p}$ discussed in (1).
We are therefore unable to draw an analytic low-frequency conclusion, even
though numerical evidence again points to a unique root of $\mathcal N_n^p$.

\item We conjecture that for $1<p\le5$ the complete conclusion of the cubic
Theorem~\ref{thm:insta} holds verbatim with the same case differentiation inequalities in terms of
$N$, $n$ and $\omega$, while for $p>5$ one only need to do a natural correction accounting the sign inversion of $\partial_\omega Q$ as $\omega$ is large. This conjecture can be proved under two uniqueness
assumptions: that $\alpha_{n,p}$ has a unique minimizer
$t_{n,p}^*\in(0,1)$ and that $\mathcal N_n^p$ has a unique root
$t_{n,p}^\sharp\in(0,1)$. Under these hypotheses the existence threshold
$\omega_{n,p}^*$ becomes optimal and $t_{n,p}^\flat=t_{n,p}^\sharp$ defines a single
frequency $\omega_{n,p}^\sharp$.

\item In case of failure of either uniqueness assumption in the conjecture presented in (4), every spectral and slope argument
above remains valid on a non--optimal decreasing branch system (one in which
the minimum $t_{n,p}^*$ of $\alpha_{n,p}$ is not necessarily the infimum).
The (in)stability conclusions then survive, but must be expressed through the
position of the branch parameter $t_1(\omega)$ relative to $t_p^*$ (which fixes
the Morse index via \eqref{eq:morse-dichotomy-p}) and relative to the possibly
several values $t_{n,p}^\sharp$ at which $\mathcal N_n^p$ changes sign (which fix
the slope of $Q$ via \eqref{eq:sign-slope-p}), rather than through clean
inequalities in $N$, $n$ and $\omega$.
\end{enumerate}
\end{remark}

\section*{Acknowledgments}
J. Angulo was partially funded by CNPq/Brazil Grant and  S\~ao Paulo Research Foundation-FAPESP/Thematic Project, process number 2023/13426-8. A.  Mu\~noz was financed by FAPESP, Brazil, process number 2024/20623-7. 
\vskip0.1in
 \noindent
{\bf Data availability statement}. Data sharing not applicable to this article as no datasets were generated or analyzed during the current study.

\vskip0.1in
 \noindent
{\bf Conflict of interest statement}. The authors declare that there are no conflicts of interest regarding the content of this paper.

\appendix

\section{Neumann problem on the half-line for $p>1$}
\label{ap:1}

In this appendix we record the one-dimensional spectral fact needed in the
Morse index computation of $\mathcal L_{+,Z}$, namely a sharp description of
the negative-eigenvalue count of the scalar Neumann half-line operator arising
from the linearization around a soliton tail.

\medskip

Let $p>1$ and $\omega,L>0$. The unique positive decaying solution of
\[
-\Psi''+\omega\Psi-\Psi^{p}=0,
\qquad x\ge L,
\]
is given (up to translation) by the one-dimensional soliton
\begin{equation}
    \label{eq:OmegaS_w}\mathfrak{S}_\omega(y)
=
\Big(\frac{(p+1)\omega}{2}\Big)^{\frac{1}{p-1}}
\operatorname{sech}^{\frac{2}{p-1}}
\!\Big(\frac{p-1}{2}\sqrt{\omega}\,y\Big).
\end{equation}

Fix a shift parameter $a>0$ and define the half-line tail profile $\Psi_{a,p, \omega}=\Psi_{a,p}$
\begin{equation}\label{eq:Theta-p}
\Psi_{a,p}(x)
:=
\Big(\frac{(p+1)\omega}{2}\Big)^{\frac{1}{p-1}}
\operatorname{sech}^{\frac{2}{p-1}}
\!\Big(\frac{p-1}{2}\sqrt{\omega}\,(x-L)+a\Big),
\qquad x\ge L.
\end{equation}

\medskip

We consider the scalar Neumann operator
\begin{equation}\label{eq:Lplus-halfline-p}
\mathfrak L_{+,a}^{(p)}
:=
-\partial_x^2+\omega-p\,\Psi_{a,p}^{p-1}(x),
\qquad
D(\mathfrak L_{+,a}^{(p)})
=
\{u\in H^2([L,\infty)):\ u'(L)=0\}.
\end{equation}

\medskip

It will be convenient to introduce the dimensionless parameter
\begin{equation}\label{eq:kappa-p}
\kappa_p:=\frac{2p(p+1)}{(p-1)^2}>0,
\end{equation}
and the shift threshold
\begin{equation}\label{eq:ap-threshold}
a_p^{\sharp}:=\;
\operatorname{arccosh}\!\biggl(\sqrt{\tfrac{p+1}{2}}\biggr)
=
\operatorname{arctanh}\!\biggl(\sqrt{\tfrac{p-1}{p+1}}\biggr).
\end{equation}
For $p=3$ one has $\kappa_3=6$ and $a_3^{\sharp}=\operatorname{arccosh}(\sqrt{2})\approx 0.881$.

\begin{lemma}\label{lem:LA-negative-generalp}
Let $p>1$ and $\omega,L>0$. For any shift $a>0$, the operator
$\mathfrak L_{+,a}^{(p)}$ defined in \eqref{eq:Lplus-halfline-p} satisfies
\[
n\bigl(\mathfrak L_{+,a}^{(p)}\bigr)
=
\begin{cases}
1, & 0<a<a_p^{\sharp},\\[2pt]
0, & a\ge a_p^{\sharp},
\end{cases}
\]
where $a_p^{\sharp}$ is defined in \eqref{eq:ap-threshold}.
Moreover, on the range $0<a<a_p^{\sharp}$ the unique negative eigenvalue
$\lambda_p(a)<0$ is strictly increasing, with
$\lambda_p(a)\uparrow 0$ as $a\uparrow a_p^{\sharp}$.
\end{lemma}

\begin{proof}
Set
\[
y:=\frac{p-1}{2}\sqrt{\omega}\,(x-L)+a,
\qquad y\in[a,\infty),
\]
and $u(y):=f(x(y))$. Since
\begin{equation}\label{changes}
\Psi_{a,p}^{p-1}(x)
=
\frac{(p+1)\omega}{2}\,\operatorname{sech}^{2}(y),
\qquad
\partial_x^2=\frac{(p-1)^2\omega}{4}\,\partial_y^2,
\end{equation}
the eigenvalue problem $\mathfrak L_{+,a}^{(p)}f=\lambda f$ becomes
\begin{equation*}\label{eqsoliton}
-\frac{(p-1)^2\omega}{4}\,u''+\omega\,u
-\frac{p(p+1)\omega}{2}\operatorname{sech}^2(y)\,u=\lambda u.
\end{equation*}
Dividing by $\tfrac{(p-1)^2\omega}{4}$, this is equivalent to
\[
\widetilde{\mathfrak L}_{a,p}\,u=\mu\,u,
\qquad \mu:=\frac{4\lambda}{(p-1)^2\omega},
\]
where
\begin{equation}\label{eq:Ltilde-p}
\widetilde{\mathfrak L}_{a,p}
:=
-\partial_y^2+\frac{4}{(p-1)^2}-\kappa_p\,\operatorname{sech}^2(y),
\qquad
D(\widetilde{\mathfrak L}_{a,p})
=
\{u\in H^2([a,\infty)):\ u'(a)=0\},
\end{equation}
and $\kappa_p$ is the constant \eqref{eq:kappa-p}. Note that the expression
$\mu=4\lambda/((p-1)^2\omega)$ preserves the order for $\lambda, \mu$ in $(-\infty,0)$. Hence it suffices to prove
\[
n(\widetilde{\mathfrak L}_{a,p})=
\begin{cases}
1, & 0<a<a_p^{\sharp},\\[2pt]
0, & a\ge a_p^{\sharp}.
\end{cases}
\]

\medskip
Let us first show that $n(\tilde{\mathfrak{L}}_{a,p})\le 1$ for any $a>0$.
\medskip
Consider the whole-line operator
\begin{equation}\label{allR}
\widetilde{\mathfrak L}_{p}
:=
-\partial_y^2+\tfrac{4}{(p-1)^2}-\kappa_p\,\operatorname{sech}^2(y),
\qquad
D(\widetilde{\mathfrak L}_{p})=H^2(\R),
\end{equation}
then
\[
n(\widetilde{\mathfrak L}_{p})=1,
\qquad
\ker(\widetilde{\mathfrak L}_{p})=\mathrm{span}\{\mathfrak{S}_\omega'\}.
\]
Indeed, \eqref{changes} implies $\widetilde{\mathfrak L}_{p}(\mathfrak{S}_\omega')=0$, moreover, $\mathfrak{S}_\omega'(x)=0$ if and only if $x=0$, then classical oscillation theory implies $n(\widetilde{\mathfrak L}_{p})=1$ and the simplicity of the kernel (see \cite{BerShu91}, Chapter 2, \S2 and  \S3).

\medskip
Denote by $V_p(y):=\tfrac{4}{(p-1)^2}-\kappa_p\operatorname{sech}^2(y)$ the
rescaled potential. Consider the quadratic forms $q: H^1(\mathbb R)\to \mathbb R$ and $q_a: H^1([a,\infty))\to  \mathbb R $, associated to the operators in \eqref{allR} and \eqref{eq:Ltilde-p},  defined by
\[
q[v]=\int_{\R}\!\bigl(|v'|^2+V_p(y)|v|^2\bigr)\,dy,
\qquad
q_a[v]=\int_{a}^{\infty}\!\bigl(|v'|^2+V_p(y)|v|^2\bigr)\,dy.
\]
For $v\in H^1([a,\infty))$ with $v'(a)=0$, define the \emph{reflection} extension
$Ev\in H^1(\R)$ by
\[
(Ev)(y):=
\begin{cases}
v(2a-y), & y\le a,\\
v(y), & y\ge a.
\end{cases}
\]
Then $(Ev)'(y)=-v'(2a-y)$ for $y<a$, so
$|(Ev)'(y)|^2=|v'(2a-y)|^2$, and by change of variable
\[
\int_{-\infty}^{a}|(Ev)'|^2\,dy=\int_{a}^{\infty}|v'|^2\,dy,
\qquad
\int_{-\infty}^{a}|Ev|^2\,dy=\int_{a}^{\infty}|v|^2\,dy.
\]
Moreover, for $y>a$ we have $y>|2a-y|$, and since $\operatorname{sech}^2(y)$ is even and
strictly decreasing for $y\in[0,\infty)$, $\operatorname{sech}^2(y)\le\operatorname{sech}^2(2a-y)$,
so $V_p(y)\ge V_p(2a-y)$. Hence
\[
\int_{-\infty}^{a}V_p(y)|Ev|^2\,dy
=\int_{a}^{\infty}V_p(2a-y)|v(y)|^2\,dy
\le \int_a^\infty V_p(y)|v|^2\,dy.
\]
Summing contributions,
\[
q[Ev]\le 2\,q_a[v],
\qquad
\|Ev\|_{L^2(\R)}^2=2\,\|v\|_{L^2([a,\infty))}^2.
\]
By the min-max principle, if $\widetilde{\mathfrak L}_{a,p}$ had two linearly
independent negative directions $v_1,v_2$, then $Ev_1,Ev_2$ would form a two-dimensional
negative subspace for $\widetilde{\mathfrak L}_{p}$, contradicting
$n(\widetilde{\mathfrak L}_{p})=1$. Therefore $n(\widetilde{\mathfrak L}_{a,p})\le 1$, as claimed.


\medskip
We now see that for $a$ large the operator $\tilde{\mathfrak{L}}_{p,a}$ is non-negative. In fact, using $\kappa_p-\tfrac{4}{(p-1)^2}=\tfrac{2p(p+1)-4}{(p-1)^2}>0$, we
find
\[
V_p(y)<0\iff \operatorname{sech}^2(y)>\frac{4}{(p-1)^2\kappa_p}=\frac{2}{p(p+1)}.
\]

\medskip
For $a\ge \operatorname{arccosh}(\sqrt{\tfrac{p(p+1)}{2}}):=C_0$ one has $\operatorname{sech}^2(y)\le\operatorname{sech}^2(a)\le \tfrac{4}{(p-1)^2\kappa_p}$ for all $y\ge a$, so $V_p\ge 0$ on $[a,\infty)$. For any $u\in D(\widetilde{\mathfrak L}_{a,p})$, an integration by parts (using $u'(a)=0$) gives
\[
\langle \widetilde{\mathfrak L}_{a,p}u,u\rangle
=
\int_a^\infty |u'|^2\,dy+\int_a^\infty V_p(y)|u|^2\,dy\ge 0.
\]

\medskip
We now restrict ourselves to the case $0<a<C_0$. 

\medskip
We now show that whenever eigenvalues $\mu_p(a)$ exist for $\tilde{\mathfrak{L}}_{a,p}$ they must be strictly increasing in $a\in (0,C_0)$.

\smallskip
\noindent Suppose that for  $0<a_1<a_2<C_0$ there exist $L^2$-normalized eigenfunctions $u_{a_1}$ and $ u_{a_2}$ of $\widetilde{\mathfrak L}_{a_i,p}$ with eigenvalues $\mu_p(a_i)<0$, $i=1,2$.
The Neumann condition $u_{a_2}'(a_2)=0$ together with the ODE forces
$u_{a_2}(a_2)\ne 0$, since otherwise $u_{a_2}\equiv 0$.

Define $Eu_{a_2}\in H^1([a_1,\infty))$ by
\[
(Eu_{a_2})(y):=\begin{cases}
u_{a_2}(a_2), & y\in[a_1,a_2],\\
u_{a_2}(y), & y\ge a_2.
\end{cases}
\]
Then $(Eu_{a_2})'=0$ on $(a_1,a_2)$, so
\begin{equation}\label{eq:mono1-p}
q_{a_1}[Eu_{a_2}]
=
q_{a_2}[u_{a_2}]
+
|u_{a_2}(a_2)|^2\int_{a_1}^{a_2}V_p(y)\,dy<q_{a_2}[u_{a_2}],
\end{equation}
\begin{equation}\label{eq:mono2-p}
\|Eu_{a_2}\|_{L^2([a_1,\infty))}^2
=
(a_2-a_1)|u_{a_2}(a_2)|^2+\|u_{a_2}\|_{L^2([a_2,\infty))}^2>\|u_{a_2}\|_{L^2([a_2,\infty))}^2,
\end{equation}
where we used that for $a_2<C_0$ one has $V_p(y)<0$ on $[a_1,a_2]$, so the
integral in \eqref{eq:mono1-p} is strictly negative. 

Using the min-max characterization we get,
\[
\mu_p(a_1)
\le
\frac{q_{a_1}[Eu_{a_2}]}{\|Eu_{a_2}\|_{L^2([a_1,\infty))}^2}<
\frac{q_{a_2}[u_{a_2}]}{\|u_{a_2}\|_{L^2([a_2,\infty))}^2}
=
\mu_p(a_2),
\] proving monotonicity. 

\medskip
Now, note the (rescaled) soliton derivative $f(y)=\operatorname{sech}^{2/(p-1)}(y)\tanh(y)$ solves $\tilde{\mathfrak{L}}_{a,p}f=0$ for any $a>0$ but $f'(a)=0$ (equivalently $\mathfrak{S}_\omega''(L)=0$) happens if and only if $a=a^\sharp_p$. Since $f$ is strictly positive on $[a_p^\sharp,\infty)$ we conclude from Sturm-Liouville oscillation theory on the half-line (see \cite{BerShu91}) that $0$ is the least eigenvalue of $\tilde{\mathfrak{L}}_{a_p^\sharp,p}$ with eigenfunction $g_p^\sharp=f\bigr|_{[a^\sharp_p,\infty)}$. From the monotonicity, we derive $n(\tilde{\mathfrak{L}}_{a,p})=0$ for any $a\ge a_p^\sharp$. 

\medskip
To prove $n(\tilde{\mathfrak{L}}_{a,p})=1$ for $0<a<a_p^\sharp$ we simple note that the extension $E[g_p^\sharp]$ of $g_p^\sharp$ to $[a,\infty)$ satisfies $$\frac{q_a[Eg_p^\sharp]}{\|Eg_p^\sharp\|^2_{L^2([a,\infty))}}<\frac{q_{a_p^\sharp}[g_p^\sharp]}{\|g_p^\sharp\|^2_{L^2([a_p^\sharp,\infty))}}=0, \qquad 0<a<a_p^\sharp.$$

\medskip
Finally, converting back to the original eigenvalue via
$\lambda_p(a)=\tfrac{(p-1)^2\omega}{4}\,\mu_p(a)$ preserves sign, strict monotonicity, and the limit $\lambda_p(a)\uparrow 0$ as $a\uparrow a_p^{\sharp}$.
\end{proof}

\begin{remark}We note that the threshold value $a_p^{\sharp}$ in Lemma \ref{lem:LA-negative-generalp}, governing the Morse index of $(\mathfrak L_{+,a}^{(p)}, D(\mathfrak L_{+,a}^{(p)}))$ defined in \eqref{eq:Lplus-halfline-p}, arises naturally from a perturbation argument. For completeness, we sketch this approach below. 

We split the analysis according to the position of the shift $a$ relative to the threshold $a_p^{\sharp}$, treating the threshold case itself via perturbation theory.

We first consider $a\neq a_p^{\sharp}$. In this case $\Psi_{a,p}''(L)\neq 0$, and it is easy to see that $\ker(\mathfrak L_{+,a}^{(p)})=\{0\}$.

We next turn to the threshold value $a=a_p^{\sharp}$. Here $\Psi_{a,p}'\in D(\mathfrak L_{+, a}^{(p)})$ and belongs to the kernel of $\mathfrak L_{+, a}^{(p)}$, so by oscillation theory $\ker(\mathfrak L_{+, a}^{(p)})=\{\Psi_{a,p}'\}$ (see \cite[Chapter~2, \S 2 and \S 3]{BerShu91}). Moreover, since $\Psi_{a,p}'$ does not change sign on $[L,+\infty)$, it follows that $n\bigl(\mathfrak L_{+,a}^{(p)}\bigr)=0$. One can also show that $\langle \mathfrak L_{+,a}^{(p)} v, v\rangle\geqq 0$ for all $v\in D(\mathfrak L_{+, a}^{(p)})$ (see the proof of Proposition~10.9, item~(ii), in \cite{AC}).

It remains to determine the Morse index for $a\neq a_p^{\sharp}$, which we obtain by perturbing about the threshold case just analyzed. It is not difficult to show that $\lim_{a\to a_p^{\sharp}}\Psi_{a,p}=\Psi_{a_p^{\sharp},p}$ in $H^1([L,+\infty))$. Moreover, the family $\{\mathfrak L_{+, a}^{(p)}\}_{a\in\mathbb R^{+}}$ is a real-analytic family of self-adjoint operators of type (B) in the sense of Kato (see \cite[Theorem VII-4.2]{kato}), and
\[
\lim_{a\to a_p^{\sharp}}\widehat{\delta}\big(\mathfrak L_{+, a}^{(p)}, \mathfrak L_{+, a_p^{\sharp} }^{(p)}\big)=0,
\]
where $\widehat{\delta}$ denotes the gap metric (see \cite[Theorem 2.14]{kato}). Combining the spectral structure of $\mathfrak L_{+, a_p^{\sharp} }^{(p)}$ established above with \cite[Theorem IV-3.16]{kato} and the Kato-Rellich Theorem (\cite[Theorem XX.8]{RS4}), we obtain two analytic functions $\Lambda$ and $\Pi$ on a neighborhood $I(a_p^{\sharp})=(a_p^{\sharp}-\epsilon, a_p^{\sharp}+\epsilon)$ of $a_p^{\sharp}$, with $\Lambda: I(a_p^{\sharp})\to \mathbb R$ and $\Pi: I(a_p^{\sharp})\to L^2([L,+\infty))$, such that $\Lambda(a_p^{\sharp})=0$ and $\Pi(a_p^{\sharp})=\Psi'_{a_p^{\sharp},p}$. For every $a\in I(a_p^{\sharp})$, $\Lambda(a)$ is the simple isolated first eigenvalue of $\mathfrak L_{+, a}^{(p)}$, and $\Pi(a)$ is its associated eigenfunction. Furthermore, $\epsilon$ can be chosen small enough that, for $a\in I(a_p^{\sharp})$, the spectrum of $\mathfrak L_{+, a}^{(p)}$ is positive except possibly for the first eigenvalue (recall that zero is an isolated eigenvalue of $\mathfrak L_{+, a_p^{\sharp} }^{(p)}$).

By Taylor's Theorem we then have the expansions
\begin{equation}\label{Taylor}
\Lambda(a)= \Lambda'(a_p^{\sharp})(a-a_p^{\sharp})+O(|a-a_p^{\sharp}|^2)\quad\text{and}\quad \Pi(a)= \Psi'_{a_p^{\sharp},p} + \Pi'(a_p^{\sharp})(a-a_p^{\sharp})+O(|a-a_p^{\sharp}|^2).
\end{equation}
Following the ideas in the proof of Proposition~10.3 in \cite{AC}, one finds
\[
\Lambda'(a_p^{\sharp})=-\frac{2p}{\|\Psi'_{a_p^{\sharp},p}\|^2}\int_L^\infty (\Psi'_{a_p^{\sharp},p})^3 \Psi^{p-2}_{a_p^{\sharp},p}\, dx + O(a-a_p^{\sharp}).
\]
Since $\Psi'_{a_p^{\sharp},p}$ is negative, it follows that $\Lambda'(a_p^{\sharp})>0$ for sufficiently small $|a-a_p^{\sharp}|$, which in view of \eqref{Taylor} shows that $n\bigl(\mathfrak L_{+,a}^{(p)}\bigr)=1$ for $a<a_p^{\sharp}$ and $n\bigl(\mathfrak L_{+,a}^{(p)}\bigr)=0$ for $a\geqq a_p^{\sharp}$, at least for $a\approx a_p^{\sharp}$. Finally, to extend this description of the Morse index to all $a>0$, we use a continuation argument based on the case $a\neq a_p^{\sharp}$ treated above together with the Riesz projection (see the proof of Proposition~10.4 in \cite{AC}).
\end{remark}

\section{Differentiation with respect to \texorpdfstring{$\lambda$}{lambda}}\label{ap:2}

Throughout this section we assume $\omega>0$ and $p>1$ are fixed, $\Psi_a$ denotes the
positive soliton tail \eqref{eq:Theta-p} with shift $a>0$, and $L\in\R$ is fixed. We
write $C_b([L,\infty))$ for the Banach space of bounded continuous
functions equipped with the supremum norm.

The idea of this section is to show the following result which was used in the proof of Theorem \ref{thm:morse1}.

\begin{lemma}\label{lem:apb-generalp}
Let $R_p:=-\partial_x^2+\omega-p\,\Psi^{p-1}_a(x)$.  Consider the EDO's problem on $[L,\infty)$
\begin{equation}\label{edo}
(R_p-\lambda)f=0,
\qquad f(x)\to 0 \;\;\text{as}\;\; x\to +\infty.
\end{equation}
For every $\lambda<0$ there exists a nontrivial solution of \eqref{edo} (unique up to a multiplicative constant)
$f_\lambda\in H^2([L,\infty))$ with
\[
\lim_{x\to +\infty}e^{\sqrt{\omega-\lambda}(x-L)}f_\lambda(x)=1.
\]
 Moreover, for every open interval $I\subset(-\infty,0)$ the maps
\[
I\ni\lambda\;\longmapsto\; f_\lambda\in H^2([L,\infty)),
\qquad
I\ni\lambda\;\longmapsto\; f_\lambda(L)\in\C
\]
are of class $C^1$.
\end{lemma}

\begin{proof}
Write $l(x):=-p\,\Psi^{p-1}_a(x)$. From the explicit form \eqref{eq:Theta-p} and $p>1$,
the function $\Psi_a$ is smooth and decays exponentially at $+\infty$, so
$l\in C^\infty([L,\infty))$ and there exist constants $C,\eta>0$ with
\begin{equation}\label{eq:l-decay}
|l(x)|\le C e^{-\eta x},\qquad x\ge L.
\end{equation}
In particular $l\in L^1([L,\infty))\cap L^\infty([L,\infty))$. For $\lambda<0$ set
\[
\mu=\mu(\lambda):=\sqrt{\omega-\lambda}>0.
\]
The equation $(R_p-\lambda)f=0$ rewrites as
\begin{equation}\label{eq:f-eq}
f''(x)=\bigl(\mu^2+l(x)\bigr)f(x),\qquad x\ge L.
\end{equation}

\medskip
We seek for decaying solutions of \eqref{eq:f-eq} in the form
\begin{equation}\label{eq:f-ansatz}
f(x)=e^{-\mu(x-L)}m(x).
\end{equation}
Substituting into \eqref{eq:f-eq} and dividing by $e^{-\mu(x-L)}$ gives
\begin{equation}\label{eq:m-eq}
m''(x)-2\mu\,m'(x)=l(x)\,m(x).
\end{equation}
The homogeneous equation $m''-2\mu m'=0$ has fundamental system
$\{1,e^{2\mu x}\}$ with Wronskian $W(x)=2\mu e^{2\mu x}$. Variation of
parameters, together with the normalization $m(x)\to 1$ as $x\to+\infty$
(which selects, among the one-dimensional family of decaying solutions of
\eqref{eq:f-eq}, the one with asymptotic behavior
$f(x)\sim e^{-\mu(x-L)}$), yields the Volterra equation
\begin{equation}\label{eq:m-volterra}
m(x)=1+(T_\mu m)(x),\qquad
(T_\mu g)(x):=\int_x^\infty K_\mu(x,s)\,l(s)\,g(s)\,ds,
\end{equation}
with kernel
\begin{equation}\label{eq:K-def}
K_\mu(x,s):=\frac{1-e^{-2\mu(s-x)}}{2\mu},\qquad s\ge x\ge L.
\end{equation}
A direct check confirms that any bounded $C^2$ solution of \eqref{eq:m-volterra} satisfies
\eqref{eq:m-eq}. Thus
$f=e^{-\mu(\cdot-L)}m$ is a decaying solution of $(R_p-\lambda)f=0$ if and only if
$m$ solves \eqref{eq:m-volterra}.

In the following we show the existence of such one $m$. For $s\ge x\ge L$ and $\mu>0$ we have $0\le 1-e^{-2\mu(s-x)}\le 1$, so
\begin{equation}\label{eq:K-bound}
0\le K_\mu(x,s)\le \frac{1}{2\mu}.
\end{equation}

\medskip
Define $\rho(x):=\int_x^\infty|l(s)|\,ds$. By \eqref{eq:l-decay} we have $\rho\in
C^1([L,\infty))$, $\rho'=-|l|$, $\rho$ is nonincreasing, and $\rho(x)\to 0$ as
$x\to+\infty$. We claim that for every $g\in C_b([L,\infty))$ and every $n\ge 0$,
\begin{equation}\label{eq:T-n-bound}
|(T_\mu^n g)(x)|\le \|g\|_\infty\,\frac{\rho(x)^n}{n!\,(2\mu)^n},\qquad x\ge L.
\end{equation}
The case $n=0$ is trivial. Assume \eqref{eq:T-n-bound} holds for $n$. Using
\eqref{eq:K-bound} and $\rho'=-|l|$,
\[
|(T_\mu^{n+1}g)(x)|\le\frac{\|g\|_\infty}{(2\mu)^{n+1} n!}
\int_x^\infty|l(s)|\rho(s)^n\,ds
=\frac{\|g\|_\infty}{(2\mu)^{n+1} n!}\cdot\frac{\rho(x)^{n+1}}{n+1}
=\frac{\|g\|_\infty\,\rho(x)^{n+1}}{(n+1)!\,(2\mu)^{n+1}},
\]
where we used $$\int_x^\infty|l|\rho^nds=-\int_x^\infty(\rho^{n+1}/(n+1))'ds=\rho(x)^{n+1}/(n+1).$$
This proves \eqref{eq:T-n-bound} by induction.

In particular, the relation
\[
\|T_\mu^n\|_{\mathcal B(C_b([L,\infty)))}
\le\frac{\rho(L)^n}{n!\,(2\mu)^n},
\]
implies that the spectral radius associated to the bounded operator $T_\mu: C_b([L, +\infty))\to C_b([L, +\infty))$ vanishes and so $\sigma(T_\mu)=\{0\}$. Therefore, $I-T_\mu$ is invertible with
\begin{equation}\label{eq:Neumann-norm}
(I-T_\mu)^{-1}=\sum_{n=0}^\infty T_\mu^n,
\qquad
\|(I-T_\mu)^{-1}\|_{\mathcal B(C_b)}\le e^{\rho(L)/(2\mu)}.
\end{equation}
Define
\begin{equation}\label{eq:m-def}
m_\lambda:=(I-T_\mu)^{-1}\mathbf 1\in C_b([L,\infty)),
\qquad
\|m_\lambda\|_\infty\le e^{\rho(L)/(2\mu)}.
\end{equation}
By \eqref{eq:T-n-bound} applied to $m_\lambda$ we have
$|(T_\mu m_\lambda)(x)|\le\|m_\lambda\|_\infty\,\rho(x)/(2\mu)\to 0$ as $x\to+\infty$,
so \eqref{eq:m-volterra} gives $m_\lambda(x)\to 1$ as $x\to+\infty$. Consequently
$f_\lambda(x):=e^{-\mu(x-L)}m_\lambda(x)$ is a nontrivial solution of
\eqref{eq:f-eq} satisfying $|f_\lambda(x)|\le \|m_\lambda\|_\infty e^{-\mu(x-L)}$ for
$x\ge L$, hence $f_\lambda\in L^2([L,\infty))\cap C_b([L,\infty))$.

\medskip
From $f_\lambda\in L^2$, $l\in L^\infty$, and \eqref{eq:f-eq},
\[
\|f_\lambda''\|_{L^2}\le(\mu^2+\|l\|_\infty)\|f_\lambda\|_{L^2}<\infty.
\]
On the half-line, the equivalence
\begin{equation}\label{eq:H2-equiv}
\|u\|_{H^2([L,\infty))}^2\;\sim\;\|u\|_{L^2}^2+\|u''\|_{L^2}^2,\qquad u\in H^2([L,\infty)),
\end{equation}
(proved by integration by parts with a Young-inequality absorption of the boundary
term at $L$) then yields $f_\lambda\in H^2([L,\infty))$. This shows the first part of the Lemma

\bigskip
We now focus on the regularity claims on the mapping $\lambda\to f_\lambda$. The analysis is classical and so we give some  highlights. The map $\mu\mapsto K_\mu(x,s)$ is $C^1$ in $\mu$ for each fixed $(x,s)$ and for all
$\mu>0$ and $s\ge x\ge L$, $|\frac{d^{(j)}}{d\mu^{j}}K_\mu(x,s)|\le\frac{1}{2\mu^{j+1}}$, $j=0,1$. Thus,  as  $l\in L^1([L,\infty))$ dominated convergence theorem  gives $\mu\in I\mapsto T_\mu\in \mathcal B(C_b([L,\infty))$  is of class $C^1$ for  every open interval $I\subset(-\infty,0)$ and it permits
differentiation with regard to $\mu$ under the integral sign in  \eqref{eq:m-volterra}.

\medskip
Now, since $\sigma(T_\mu)=\{0\}$ for every $\mu>0$, the operator $I-T_\mu$ is invertible,
and the inversion map
\[
\Phi:\{A\in\mathcal B(C_b):1\notin\sigma(A)\}\to\mathcal B(C_b),\qquad A\mapsto(I-A)^{-1},
\]
is real-analytic (it is given locally by the convergent Neumann series). Composition
of $C^1$ maps yields that
\[
\mu\mapsto m_\mu:=(I-T_\mu)^{-1}\mathbf 1\in C_b([L,\infty))
\quad\text{is }C^1
\]

Since the map $\mu\mapsto e_\mu:=e^{-\mu(\cdot-L)}\in L^2([L,\infty))$ is also $C^1$
we have
\[
\mu\mapsto f_\mu:=e_\mu\,m_\mu\in L^2([L,\infty))
\quad\text{is }C^1,
\]
with derivative $\partial_\mu f_\mu=(\partial_\mu e_\mu)m_\mu+e_\mu(\partial_\mu m_\mu)$.

Next, to upgrade to $H^2$ the analysis above , differentiate \eqref{eq:f-eq} in $\mu$ to get
\begin{equation}\label{eq:f-mu-eq}
(\partial_\mu f_\mu)''=(\mu^2+l)\,\partial_\mu f_\mu+2\mu\,f_\mu.
\end{equation}
Since $\partial_\mu f_\mu\in L^2$, $f_\mu\in L^2$, and $l\in L^\infty$, the right-hand
side of \eqref{eq:f-mu-eq} is in $L^2$ uniformly in $\mu$ obviously on compact sub-intervals
of $(0, +\infty)$. Combining with the original equation \eqref{eq:f-eq} and \eqref{eq:H2-equiv}
gives $\mu\mapsto f_\mu\in H^2([L,\infty))$ of class $C^1$ on  $(0, +\infty)$.

\medskip
In conclusion, being the map $\lambda\mapsto\mu(\lambda)=\sqrt{\omega-\lambda}$  real-analytic on
$(-\infty,0)$ and  $\mu\mapsto f_\mu$ of class $C^1$, we have that  $\lambda\mapsto f_\lambda$ is
$C^1$ from  $(0, +\infty)$ to $H^2([L,\infty))$. 
Finally, the half-line Sobolev embedding
$H^2([L,\infty))\hookrightarrow C^0([L,\infty))$ implies that the trace
functional
\[
\operatorname{tr}_L:H^2([L,\infty))\to\C,\qquad u\mapsto u(L),
\]
is continuous and linear. Composition with the $C^1$ map
$\lambda\mapsto f_\lambda\in H^2([L,\infty))$ yields that
$\lambda\mapsto f_\lambda(L)$ is $C^1$ on $(0, +\infty)$, with
$\partial_\lambda f_\lambda(L)=(\partial_\lambda f_\lambda)(L)$. This finishes the proof.
\end{proof}

\section{Orbital (in)stability criterion}\label{ApA}

In the Grillakis-Shatah-Straus (GSS, \cite{GSS2,GSS1}) setting we assume the existence of $C^2$-conserved functionals $E_{Z}:H^1(\dg)\to \R$ (energy) with $E_{Z}(\mathbf{U}(t))=E_{Z}(\mathbf{U}(0))$, and $Q:L^2(\dg)\to \R$ (mass) with $Q(\mathbf{U}(t))=Q(\mathbf{U}(0))=\|\mathbf{U}(0)\|^2_2$.
Note $E_{Z}$ is invariant by the phase symmetry $T(\theta)=e^{i\theta}$ for any $\theta\in[0,2\pi)$. Note also the standing waves are written in the form $T(\omega t)\Theta(x)$.\smallskip

Assume the Cauchy problem associated to \eqref{NLS} is locally well-posed in the energy space $H^1(\dg)$.

\smallskip
We suppose the existence of a $C^1$ map on $\mathcal{O}\subset \R$,  $$\mathcal{O}\ni \omega \mapsto \Theta_\omega\in H^1(\dg)$$ of stationary solutions that are critical points of the action functional $S=E_{Z}+\omega Q$. 
\bigskip

Define on $\mathcal{O}$ the function \begin{equation}
    \label{42}\rho(\omega_0):=\begin{cases}
        1,& \mbox{if } \partial_\omega \|\Theta_\omega\|^2 >0 \ \ \mbox{at }\omega=\omega_0.\\
        0, & \mbox{if } \partial_\omega \|\Theta_\omega\|^2 <0 \ \ \mbox{at }\omega=\omega_0.
    \end{cases}
\end{equation}
For the (in)stability study of $\Theta_\omega$ the main information will be given by the second variation $$S''(\Theta_\omega)(\mathbf{U},\mathbf{V})=[\mathcal{L}_{+,Z}\mathbf{U_1},\mathbf{V}_1]+[\mathcal{L}_{-,Z}\mathbf{U_2},\mathbf{V}_2], \ \ \ \mathbf{U}=\mathbf{U_1}+i\mathbf{U_2}, \ \ \mathbf{U}=\mathbf{V_1}+i\mathbf{V_2}, $$
in the following sense: 

\begin{theorem}\label{2main}
 Suppose $\mbox{Ker}(\mathcal{L}_{-,Z})=\mbox{span}\{ \Theta_\omega\}$ and $\mbox{Ker}(\mathcal{L}_{+,Z})=\{0\}$. Assume also the Morse indices of $n(\mathcal{L}_{\pm,Z})$ are finite while the rest of the spectra is bounded away from $0$. Then the following hold for $\mathcal{H}:=\mbox{diag}(\mathcal{L}_{+,Z},\mathcal{L}_{-,Z})$:
    \begin{enumerate}
        \item If $n(\mathcal H)=\rho(\omega)=1,$ then the standing wave $e^{i \omega t}\Theta_\omega$ is orbitally stable in the energy space $H^1(\dg)$.
        \item If $n(\mathcal{H})-\rho(\omega)$ is odd, then the standing wave $e^{i \omega t}\Theta_\omega$ is orbitally unstable in the energy space $H^1(\dg)$.
    \end{enumerate}
\end{theorem}

\begin{remark}
    \label{lintononlin}
The second item above deserve further discussion. From \cite{GSS2} when $n(\mathcal{H})-\rho(\omega)$ is odd it is obtained that $e^{i \omega t}\Theta_\omega$ is spectrally unstable via the existence of nonzero eigenvalues for the linearized operator $$J\mathcal{H}=\begin{pmatrix}
    0&I_{N+1}\\-I_{N+1}&0
\end{pmatrix}
\begin{pmatrix}
    \mathcal{L}_{+,Z}&0\\ 0& \mathcal{L}_{-,Z}
\end{pmatrix}=\begin{pmatrix}
    0&\mathcal{L}_{-,Z}\\ -\mathcal{L}_{+,Z}&0
\end{pmatrix}$$ (see \cite[Theorem~5.1]{GSS2} and \cite[Definition~5.4]{G} for a complete description). In order to get nonlinear instability via \cite[Theorem~6.1]{GSS2} one need the semigroup estimate $$\|e^{tJ\mathcal{H}}\|\le b e^{\mu t} \ \ \mbox{for some } \mu <2\mbox{Re}\lambda,$$ being $\lambda$ an eigenvalue of $J\mathcal{H}$ with positive real part. In the case of Schr\"odinger operators posed on metric graphs we don't know a "standardized" way to obtain such a result. One may obtain nonlinear instability using the approach of \cite[Theorem~2]{HPW} provided the data--solution map is at least of class $C^2$, which is the case of this manuscript (see the details for instance on \cite[Theorem~5.2]{G} or \cite[Theorem~3.5]{AP3}. See also applications of such ideas, for instance, in \cite{ALN,ANata1}).
\end{remark}

\section{Summary tables}\label{ap:D}

For the reader's convenience, this sections provides a complete description of the (in)stability results established Theorem \ref{thm:insta}, the Morse index results in Theorem~\ref{thm:morse1} and the slope sign analysis in Theorem~\ref{thm:slope-condition}. Recall $Z<0$, $N\geqq 2$, $n\in \{1,2,..., N-1\} $, $\omega_n^*=\tfrac{(n^{2/3}+(N-n)^{2/3})^3}{Z^2}$ was defined in \eqref{eq:omega-n-star} and $\omega_n^\sharp=\tfrac{(\alpha_n(t_n^\sharp))^2}{Z^2}$ where $t_n^\sharp$ is defined in the proof of Theorem~\ref{thm:slope-condition}. 
\begin{figure}[H]
    \centering
\begin{tabular}{|c|c|ccc|}
\hline
Cluster configuration & Frequencies & \multicolumn{1}{c|}{Constraints} & \multicolumn{1}{c|}{(In)Stability} & Space \\ \hline
\multirow{5}{*}{$2n<N$} & $\omega>\omega_n^\sharp$ & \multicolumn{1}{c|}{$n=1$, $N>2$} & \multicolumn{1}{c|}{Stable} & $H^1(\dg)$ \\ \cline{2-5} 
 & $\omega>\omega_n^\sharp$ & \multicolumn{1}{c|}{--} & \multicolumn{1}{c|}{Stable} & $\xm\cap H^1(\dg)$ \\ \cline{2-5} 
 & $\omega>\omega_n^\sharp$ & \multicolumn{1}{c|}{$N>2$, $n\ge 2$} & \multicolumn{1}{c|}{Unstable} & $H^1(\dg)$ \\ \cline{2-5} 
 & $\omega_n^*<\omega<\omega_n^\sharp$ & \multicolumn{1}{c|}{-} & \multicolumn{1}{c|}{Unstable} & $\xm\cap H^1(\dg)\, (\Rightarrow H^1(\dg))$ \\ \cline{2-5} 
 & $\omega=\omega_n^\sharp$ & \multicolumn{3}{l|}{Method fails because of degenerated slope condition (see Table~\ref{table:3})} \\ \hline
\multirow{3}{*}{$2n=N$} & $\omega>\omega_n^*$ & \multicolumn{1}{c|}{$n=1$, $N=2$} & \multicolumn{1}{c|}{Stable} & $H^1(\dg)$ \\ \cline{2-5} 
 & $\omega>\omega_n^*$ & \multicolumn{1}{c|}{--} & \multicolumn{1}{c|}{Stable} & $\xm\cap H^1(\dg)$ \\ \cline{2-5} 
 & $\omega>\omega_n^*$ & \multicolumn{1}{c|}{$N>2$} & \multicolumn{1}{c|}{Unstable} & $H^1(\dg)$ \\ \hline
\multirow{4}{*}{$2n>N$} & $\omega>\tfrac{2N^2}{Z^2}$ & \multicolumn{1}{c|}{--} & \multicolumn{1}{c|}{Stable} & $\xm\cap H^1(\dg)$ \\ \cline{2-5} 
 & $\omega>\tfrac{2N^2}{Z^2}$ & \multicolumn{1}{c|}{$N>2$} & \multicolumn{1}{c|}{Unstable} & $H^1(\dg)$ \\ \cline{2-5} 
 & $\omega_n^*<\omega<\tfrac{2N^2}{Z^2}$ & \multicolumn{1}{c|}{--} & \multicolumn{1}{c|}{Unstable} & $\xm\cap H^1(\dg)\, (\Rightarrow H^1(\dg))$ \\ \cline{2-5} 
 & $\omega=\tfrac{2N^2}{Z^2}$ & \multicolumn{3}{l|}{Method fails because $\operatorname{dim}(\ker(\mathcal{L}_{+,Z}))=N-1$ while $n(\mathcal{L}_{+,Z})=1$.} \\ \hline
All configurations & $\omega=\omega_n^*$ & \multicolumn{3}{l|}{Method fails by lack of smoothness at $t_1(\omega_n^*)$ (see for instance \eqref{eq:t1linha} or \eqref{eq:der-Qw})} \\ \hline
\end{tabular}
\setcounter{figure}{0}
\captionsetup{name=Table, labelsep=endash} 
    \caption{Summary of (in)stability results in Theorem~\ref{thm:insta}.}
    \label{table:1}
\end{figure}

\begin{figure}[H]
    \centering
\begin{tabular}{|c|c|c|c|}
\hline
Cluster configuration & Frequency & Morse index & Space \\ \hline
\multirow{2}{*}{$2n\le N$} & $\omega>\omega_n^*$ & $n$ & $\dom$ \\ \cline{2-4}
 & $\omega>\omega_n^*$ & $1$ & $\xm\cap\dom$ \\ \hline
\multirow{5}{*}{$2n>N$} & $\omega_n^*<\omega<\tfrac{2N^2}{Z^2}$ & $N-n+1$ & $\dom$ \\ \cline{2-4}
 & $\omega_n^*<\omega<\tfrac{2N^2}{Z^2}$ & $2$ & $\xm\cap\dom$ \\ \cline{2-4}
 & $\omega=\tfrac{2N^2}{Z^2}$ & $1$ & $\xm\cap\dom,\, H^1(\dg)$ \\ \cline{2-4}
 & $\omega>\tfrac{2N^2}{Z^2}$ & $n$ & $\dom$ \\ \cline{2-4}
 & $\omega>\tfrac{2N^2}{Z^2}$ & $1$ & $\xm\cap\dom$ \\ \hline
\end{tabular}
\setcounter{figure}{1}
\captionsetup{name=Table, labelsep=endash}
    \caption{Summary of the Morse index for $\mathcal L_{+,Z}$ in Theorem~\ref{thm:morse1}.}
    \label{table:2}
\end{figure}

\begin{figure}[H]
    \centering
    \begin{tabular}{|c|c|c|}
\hline
Cluster configuration & Frequency & Slope sign \\ \hline
\multirow{3}{*}{$2n<N$} & $\omega_n^*<\omega<\omega_n^\sharp$ & Negative \\ \cline{2-3}
 & $\omega=\omega_n^\sharp$ & $0$ \\ \cline{2-3}
 & $\omega>\omega_n^\sharp$ & Positive \\ \hline
$2n\ge N$ & $\omega>\omega_n^*$ & Positive \\ \hline
\end{tabular}
    \setcounter{figure}{2}
\captionsetup{name=Table, labelsep=endash}
    \caption{Summary of the slope sign analysis in Theorem~\ref{thm:slope-condition}.}
    \label{table:3}
\end{figure}

\end{document}